\documentclass{amsart}
\usepackage{graphicx, tikz-cd}
\usepackage[T1]{fontenc}
\usepackage[toc, page]{appendix}
\usepackage{amssymb, amsfonts, latexsym, mathrsfs, amsmath, amsthm, bbm, amstext}
\usepackage{txfonts,bbding,wasysym,pifont,stmaryrd,tipa, bm, amsbsy}
\usepackage[all,cmtip]{xy}
\usepackage{float, mathtools, enumitem, extarrows}
\usepackage[utf8]{inputenc}
\usepackage{hyperref}
\usepackage{cleveref}
\crefname{section}{§}{§§}
\Crefname{section}{§}{§§}
\usetikzlibrary{arrows,positioning,shapes.geometric,matrix}

\newtheorem{theorem}{Theorem}[section]
\newtheorem{lemma}[theorem]{Lemma}
\newtheorem{proposition}[theorem]{Proposition}
\newtheorem{corollary}[theorem]{Corollary}

\newtheorem{ass}[theorem]{Assumption}

\newtheorem{conv}[theorem]{Convention}

\theoremstyle{definition}
\newtheorem{definition}[theorem]{Definition}

\theoremstyle{remark}
\newtheorem{remark}[theorem]{Remark}
\newtheorem{notation}[theorem]{Notation}
\numberwithin{equation}{section}

\DeclareRobustCommand{\rchi}{{\mathpalette\irchi\relax}}
\newcommand{\irchi}[2]{\raisebox{\depth}{$#1\chi$}} 

\newcommand{\Spec}{\text{\rm Spec}}

\newcommand*{\sheafhom}{\mathcal{H}\kern -.5pt om}
\newcommand*{\sheafext}{\mathcal{E}\kern -.5pt xt}
\newcommand*{\sheafend}{\mathcal{E}\kern -.5pt nd}

\DeclareMathAlphabet{\mathpzc}{OT1}{pzc}{m}{it}

\DeclareMathOperator{\sheafHom}{\mathscr{H}\text{\kern -3pt {\calligra\large om}}\,}

\newcommand{\Z}{\mathbb{Z}}

\usepackage{color,hyperref}
\definecolor{darkblue}{rgb}{0.0,0.0,0.3}
\definecolor{olive}{rgb}{0.3, 0.4, .1}
\definecolor{dgreen}{rgb}{0.,0.6,0.}
\definecolor{gold}{rgb}{1.,0.84,0.}
\definecolor{JungleGreen}{cmyk}{0.99,0,0.52,0}
\definecolor{brn}{rgb}{0.43, 0.21, 0.1}

\begin{document}

\title{Construction of the Moduli Space of Vector Bundles on an Orbifold Curve}
\author{Soumyadip Das}
\address{Mathematics,
Indian Institute of Technology Jammu, Jagti, NH-44 , Nagrota, Jammu - 181221, J\&K, India}
\email{soumyadip.das@iitjammu.ac.in}
\author{Souradeep Majumder}
\address{Department of Mathematics,
Indian Institute of Science Education and Research Tirupati,
Srinivasapuram, Venkatagiri Road, Yerpedu Mandal, Tirupati District, Tirupati -517619. Andhra Pradesh, India}
\email{souradeep@iisertirupati.ac.in}

\makeatletter
\@namedef{subjclassname@2020}{\textup{2020} Mathematics Subject Classification}
\makeatother



\begin{abstract}
Let $k$ be an algebraically closed field of any characteristic, and let $(X,P)$ be an orbifold curve over $k$. We construct the moduli space $\mathrm{M}_{(X,P)}^{\mathrm{ss}}(n, \Delta)$ of $P$-semistable bundles on $(X,P)$ of rank $n$ and determinant $\Delta$. In the characteristic zero case, this result is well known and follows from GIT techniques. Our construction follows a different approach inspired by a GIT-free construction of Faltings. We show that when the moduli space is non-empty, it is a finite disjoint union of irreducible projective varieties.
\end{abstract}

\maketitle

\section{Introduction}

Moduli space of semistable vector bundles on a compact Riemann surface was constructed in the 1960s using Geometric Invariant Theory. This construction was generalized to the case of parabolic bundles by Mehta and Seshadri. One may adopt the viewpoint that parabolic bundles are bundles on a suitable Galois cover with a compatible action by the Galois group. This idea has been very useful, as demonstrated by the work of Mehta, Seshadri, Biswas, Nagaraj, and Boden, to name a few. One drawback of this approach is that it requires the introduction of a new parameter, namely the Galois group for the cover. An alternate approach due to Borne leads us to the notion of root stacks. If we consider parabolic bundles with weights in $\frac{1}{r}\Z$, then they are in bijective correspondence with the bundles on the $r$-th root stack associated with the given parabolic divisor. The local structure of the root stack is that of a quotient stack where the quotient group is cyclic. Thus, a root stack may be thought of as a scheme with an ``orbifold structure". One is then naturally curious to know if we can construct moduli space of bundles on a general ``orbifold". In this article, we answer this question in the case of ``orbifold curves."

Let $k$ be an algebraically closed field of an arbitrary characteristic. An orbifold curve is a smooth proper Deligne-Mumford stack over $k$ that is generically a curve and admits a smooth projective connected $k$-curve $X$ as its Coarse moduli space. Locally, an orbifold curve is a quotient stack $[V/G]$ where $V$ is some smooth affine irreducible curve admitting a generically faithful action by a finite group $G$. In particular, the orbifold curve is just an ordinary curve away from finitely many stacky points. This structure enables us to view an orbifold curve as a pair $(X,P)$, consisting of the Coarse moduli curve $X$ together with a finite data $P$ that encodes the information of certain finite Galois extensions of local fields. We would like to bring to attention that when $\mathrm{char}(k) = 0$, the groups $G$ are cyclic. This fact lies at the heart of the classical definition of parabolic bundles, which consists of a bundle along with certain filtration of its fibers at finitely many points and certain weights attached to these filtrations. Further, it becomes possible to apply Mumford's GIT technique to construct the moduli space of semistable parabolic bundles. All of these have been worked out in higher dimensions as well. 

The situation changes radically when we consider the positive characteristic case, i.e.  $\mathrm{char}(k) = p > 0$. Now, the groups $G$ are not necessarily cyclic; in fact, each $G$ is a semi-direct product of a $p$-group and a cyclic group of order prime-to-$p$. A parabolic bundle on $(X, P)$ consists of a bundle along with an action of the local Galois groups on the fibers at each stacky point. We can no longer describe a parabolic bundle in terms of weights and filtrations alone. Constructing the moduli space using GIT becomes difficult as the reductivity of certain automorphism groups is not guaranteed in general. We follow an approach inspired by Faltings ( see \cite{Faltings} and \cite{Hein}). In our setup, the approach can be summarized as follows:

\begin{enumerate}
    \item Construct the moduli functor $\mathcal{M}_{(X,P)}^{\mathrm{ss, S}}(n, \Delta)$ parametrizing $P$-semistable bundles on $(X,P)$ of rank $n$ and determinant $\Delta$ modulo suitable equivalences. This is an algebraic stack. Our goal is to construct a Coarse moduli space for this stack (see \Cref{sec_Moduli_Problem_Orbifold_Case}).
    
    \item Construct a parameter space $\mathbb{P}_{(X,P)}$, a projective space over $k$, such that each $P$-semistable bundle of rank $n$ and determinant $\Delta$ on $(X,P)$ corresponds to a $k$-point of this space. There is a universal family $\mathcal{E}_{(X,P)}$ on $\mathbb{P}_{(X,p)}$ which Zariski locally induces any other semistable family (see \Cref{sec_bdd}). 

    \item Construct suitable determinantal line bundles and determine that the locus of unstable bundles in $\mathbb{P}_{(X,P)}$ is the base locus of one such line bundle (see \Cref{sec_unstable}).

    \item Consider the blowing-up of the parameter space with respect to the unstable locus. There is a proper map from the blowing-up to a projective space, determined by the complete linear system associated with a suitable determinantal line bundle (see \Cref{sec_Main}).

    \item The Stein factorization of this map gives us the desired Coarse moduli space $\mathrm{M}_{(X,P)}^{\mathrm{ss}}(n, \Delta)$ (see \Cref{thm_main_orbifold_moduli}).
\end{enumerate}

One of the useful features of the whole construction is that we can find a Galois cover $Z \longrightarrow X$ with Galois group $\Gamma$, say, which factor as a composition
\[
Z \xlongrightarrow{g} (X,P) \longrightarrow X
\]
where $g$ is a representable tamely ramified cover. We can consider the moduli space $\mathrm{M}^{\mathrm{ss}}_Z(n,, g^*\Delta)$ of semistable bundles on $Z$ of rank $n$ and determinant $g^*\Delta$. Our construction following the ideas of Faltings enables us to embed the moduli space $\mathrm{M}_{(X,P)}^{\mathrm{ss}}(n, \Delta)$ inside $\mathrm{M}^{\mathrm{ss}}_Z(n,, g^*\Delta)$ as a closed subscheme. But in the final analysis, $\mathrm{M}_{(X,P)}^{\mathrm{ss}}(n, \Delta)$ is independent of the choice of the cover $Z$.

Another salient feature of our construction is that the parameter space $\mathbb{P}_{(X,P)}$ is a disjoint union $\mathbb{P}_{(X,P)} = \sqcup_{F \in \mathcal{F}} \, \mathbb{P}_F$ where $\mathcal{F}$ is the set of mutually non-isomorphic bundles $F$ of rank $n-1$ on $(X,P)$ such that $F \cong \oplus_{j=1}^{n-1} \, L_j$ with $L_j \in \mathrm{Pic}(X,P)$ satisfying $g^*L_j \cong \chi_j \otimes_k \mathcal{O}_Z$ as $\Gamma$-equivariant line bundles for some character $\chi_j$ of $\Gamma$ over $k$, and $\mathbb{P}_F$ is the projective space parametrizing certain extension classes of bundles determined by $F$. This stratification carries over to our description of the unstable locus, the Stein factorization and finally to the Coarse moduli space itself.

Much remains to study about the moduli space that we have constructed. For example, we would like to know whether $\mathrm{M}_{(X,P)}^{\mathrm{ss}}(n, \Delta)$ is normal, what is its Picard group, deeper understanding of the base locus of determinantal line bundle, etc. Computing its dimension is one of our future goals. The reader should note that our construction is based on the assumption that there exists a $P$-semistable bundle of rank $n$ and determinant $\Delta$ on $(X,P)$. This is equivalent to the existence of a $\Gamma$-equivariant semistable bundle on $Z$ of rank $n$ and determinant $g^*\Delta$ (see \Cref{sec_Non-Emptiness}). We would like to bring to a conclusion our investigation into the non-emptiness of the moduli space in our future work. Finally, it would be great to have an alternate construction of the moduli space using classical GIT machinery.

\section*{Acknowledgement}
It is a pleasure to thank A. J. Parameswaran and Manish Kumar for some valuable discussions. Part of the work presented in this paper was carried out at TIFR Mumbai and ISI Bangalore. We would like to acknowledge their support.

\section{Preliminaries}\label{sec_prelim}
\subsection{Moduli Space of Semistable Bundles on an Algebraic Curve}\label{sec_classical_construction}
The purpose of this section is to recall the e construction of the moduli space of the semistable bundles of a fixed rank and determinant on a smooth projective curve $X$ of genus $g$ defined over an algebraically closed field $k$ of an arbitrary characteristic.

Throughout this section, $S$ will denote a $k$-scheme. By a sheaf on $S$, we will always mean a quasi-coherent sheaf of $\mathcal{O}_S$-modules. In whatever follows, $p_S \colon S \times X \longrightarrow S$ and $q_S \colon S \times X \longrightarrow X$ denote the projection morphisms. Since $X$ is a smooth curve, the notion of a torsion-free sheaf and the locally free sheaves coincide, which we call \textit{bundles}. A classical problem is to understand the bundles on $X$. Recall that a \textit{family of bundles over} $X$, parametrized by $S$, is a coherent sheaf $E$ on $S \times X$ that is flat over $S$ and such that for each geometric point $\mathrm{Spec}(K) \xrightarrow{s} S$, the restriction $E_s \coloneqq E|_{\mathrm{Spec}(K) \times X}$ is a bundle. Two such families $E$ and $E'$ are said to be \textit{equivalent} if $E \cong E' \otimes p_S^{*}(L)$ for some line bundle $L$ on $S$. We work with the following convention.
\begin{conv}\label{conv_1}
By standard approximation techniques, in whatever follows, we only deal with finite type $k$-schemes $S$, and only consider the restrictions $E_s$ over the closed points $s \in S$. This is done to ease the notation, with the understanding that a family can always be parametrized with general $k$-schemes $S$ where the restriction is defined over geometric points. In this sense, a family of bundles over $X$, parameterized by $S$, is simply a bundle on $S \times X$. Also note that, $E_s \cong q_{S,*}\left( E \otimes p_S^* k(s) \right)$ where $k(s)$ denote the skyscraper sheaf on $S$ that is supported on $s$.
\end{conv}

In general, the collection of families of bundles on $X$ is too big to be representable or to admit a natural parameter space. One natural way to cut down the collection of families of bundles on $X$ is by suitable invariants, leading to Mumford's notion (cf. \cite{Mumford}) of slope stability. Recall that the \textit{rank} $\mathrm{rk}(E)$ of a bundle $E$ on $X$ is defined to be the dimension of the $\mathcal{O}_{X,\zeta} = k(X)$-linear space $E_\zeta$ where $\zeta$ is the generic point of $X$. Associated to a rank $n$ bundle $E$, we have the determinant line bundle $\mathrm{det}(E) = \wedge^n E$ on $X$. The \textit{degree} $\mathrm{deg}(E)$ of $E$ is defined to be the degree of the line bundle $\text{\rm det}(E)$. The \textit{slope} of $E$ is defined to be the rational number $\mu(E) \coloneqq \mathrm{deg}(E)/\mathrm{rk}(E)$. A bundle $E$ is said to be \textit{semistable} if for any proper sub-bundle $F$ of $E$, we have:
$$\mu(F) \, \leq \, \mu(E),$$
and $E$ is said to be \textit{stable} if the above inequality is strict for any proper sub-bundle $F$ of $E$. A bundle $E$ is said to be \textit{polystable} if $E$ is a direct sum of stable sub-bundles, each having the slope $\mu(E)$. Further, two semistable bundles $E$ and $E'$ are called $\mathrm{S}$-\textit{equivalent}, denoted by $E_1 \sim_{\mathrm{S}} E_2$, if their associated unique graded bundles $\mathrm{gr}_E$ and $\mathrm{gr}_{E'}$ coming from Jordan-H\"{o}lder filtrations are isomorphic.

Working over $k = \mathbb{C}$, Mumford showed in \cite{Mumford} that there is a moduli space $\mathrm{M}^{\mathrm{s}}_X(n,d)$ of stable bundles of fixed rank $n$ and degree $d$ over an arbitrary smooth projective curve $X$, and it has a structure of a non-singular projective variety. To construct a moduli space of semistable bundles of a fixed rank $n \geq 2$ and degree $d$, we first consider the following contravariant functor $\mathcal{M}_X^{\mathrm{ss}}(n,d)$ which to any $k$-scheme $S$ associates the set
\begin{equation}\label{eq_stack_F}
\mathcal{M}_X^{\mathrm{ss}}(n,d)(S) = \left\{ 
  \begin{aligned}
  &\text{Families of bundles } E \text{ on } X, \text{ parametrized by } S, \text{ such that for each} \\ 
  &\text{ closed point } s \in S, \, E_s \text{ is a semistable bundle of rank } n \text{ and }\\
  &\text{degree } d, \text{ modulo the usual equivalence of families}
  \end{aligned}
\right\}.
\end{equation}
The above definition is justified by the fact that the rank and the degree of the bundle $E_s \cong \left( E \otimes p_S^*k(s) \right)$ is independent of the closed point $s \in S$.

The stack $\mathcal{M}_X^{\mathrm{ss}}(n,d)$ is an open, quasi-compact, irreducible algebraic sub-stack of the algebraic stack $\mathcal{M}_X(n,d)$ where $\mathcal{M}_X(n,d)$ is an irreducible algebraic stack of finite type over $k$ with affine diagonal, parametrizing (not necessarily semistable) rank $n$ bundles on $X$ of degree $d$; see \cite{StackModuli} for details, and note that the proofs of the above fact in loc. cit. are independent of the characteristic of the base field $k$. The algebraic stack $\mathcal{M}_X^{\mathrm{ss}}(n,d)$ is smooth (cf. \cite[Proposition~1.3]{StackModuli} for $\text{\rm char}(k)=0$, and \cite[Section~2]{Hoffmann} when $k$ is of an arbitrary characteristic), admitting a morphism
\[
\pi_X(n,d) \colon \mathcal{M}_X^{\mathrm{ss}}(n,d) \longrightarrow \mathrm{M}^{\mathrm{ss}}_X(n,d)
\]
to a normal irreducible projective variety $\mathrm{M}^{\mathrm{ss}}_X(n,d)$. The morphism $\pi_X(n,d)$ is initial among all morphisms from $\mathcal{M}_X^{\mathrm{ss}}(n,d)$ to $k$-schemes. Seshadri gave the first ever construction of $\mathrm{M}_X^{\mathrm{ss}}(n,d)$ for $k = \mathbb{C}$ in \cite{Seshadri} using the techniques from the GIT and showed that $\mathrm{M}_X^{\mathrm{ss}}(n,d)$ is the natural compactification of Mumford's variety $\mathrm{M}_X^{\mathrm{s}}(n,d)$. The same GIT-based construction works over any base field $k$, and it follows that $\text{\rm M}_X^{\mathrm{ss}}(n,d)$ only parametrizes polystable bundles of rank $n$ and degree $d$. In general, the morphism $\pi_X(n,d)$ is not a Coarse moduli morphism since two $k$-points $[E]$ and $[E']$ in $\mathcal{M}_X^{\mathrm{ss}}(n,d)$ have the same image in $\mathrm{M}_X(n,d)$ if and only if the semistable bundles $E$ and $E'$ are $\text{\rm S}$-equivalent; see \cite[Lemma~3.11, Theorem~3.12]{StackModuli}. This leads us to consider a quotient algebraic stack $\mathcal{M}_X^{\mathrm{ss, S}}(n, d)$ of $\mathcal{M}_X^{\mathrm{ss}}(n,d)$ by considering the $\mathrm{S}$-equivalence. The functor $\mathcal{M}_X^{\mathrm{ss, S}}(n, d)$ associates
\begin{equation}
\mathcal{M}_X^{\mathrm{ss, S}}(n, d)(S) = \left\{ 
  \begin{aligned}
  &\text{Families of bundles } E \text{ on } X, \text{ parametrized by } S, \text{ such that for each} \\ 
  &\text{ close point } s \in S, \, E_s \text{ is a semistable bundle of rank } n \text{ and }\\
  &\text{degree } d, \text{ modulo the usual equivalence and the S-equivalence }\sim_{\mathrm{S}} 
  \end{aligned}
\right\},
\end{equation}
to any finite type $k$-scheme $S$. The morphism
\[
\pi \colon \mathcal{M}^{\mathrm{ss, S}}_X(n,d) \longrightarrow \mathrm{M}^{\mathrm{ss}}_X(n,d)
\]
induced by $\pi_X(n,d)$ is the Coarse moduli morphism.

After fixing a line bundle $\Delta \in \mathrm{Pic}^d(X)$ of degree $d$, one can also consider the stack $\mathcal{M}_X(n,\Delta)$ that parametrizes the bundles on $X$ of rank $n$ and determinant isomorphic to $\Delta$ as follows. There is a universal bundle $\mathcal{E}_{\mathrm{univ}}$ on the algebraic stack $X \times \mathcal{M}_X(n,d)$ where recall that $\mathcal{M}_X(n,d)$ is the algebraic stack parametrizing rank $n$ bundles on $X$ of degree $d$. For $n = 1$, the stack $\mathcal{M}_X(1,d)$ is the Picard stack $\mathcal{P}\text{ic}^d_X$, and the universal bundle is the Poincar\'{e} bundle $\mathcal{P}$ on $X \times \mathcal{P}\text{ic}^d_X$. Considering the determinant line bundle associated to the universal bundle $\mathcal{E}_{\rm univ}$, we obtain a canonical morphism
\[
\mathrm{det} \colon \mathcal{M}_X(n,d) \longrightarrow \mathcal{P}\text{ic}^d_X
\]
such that $(\mathrm{id}_X \times \mathrm{det})^*\mathcal{P} \cong \mathrm{det}(\mathcal{E}_{\text{\rm univ}})$. The choice of the line bundle $\Delta$ determines a geometric point $\mathrm{Spec}(k) \overset{[\Delta]}\longrightarrow \mathcal{P}\text{ic}^d_X$, and the stack $\mathcal{M}_X(n,\Delta)$ is the $2$-fiber product stack
\[
\mathcal{M}_X(n,\Delta) = \mathrm{Spec}(k) \times_{[\Delta], \mathcal{P}\text{ic}^d_X, \text{\rm det}} \mathcal{M}_X(n,d).
\]
By \cite[\href{https://stacks.math.columbia.edu/tag/04TF}{Lemma 04TF}]{SP} and base change results, this is an algebraic stack, locally of finite type over $k$, with affine diagonal. It again follows (see \cite{StackModuli}) that the algebraic stack $\mathcal{M}_X(n,\Delta)$ is smooth and irreducible containing an open irreducible quasi-compact algebraic sub-stack $\mathcal{M}^{\mathrm{ss}}_X(n,\Delta)$ which parametrizes semistable bundles of rank $n$ and determinant isomorphic to $\Delta$.There is again a canonical quotient algebraic stack $\mathcal{M}^{\mathrm{ss, S}}_X(n,\Delta)$ of $\mathcal{M}^{\mathrm{ss}}_X(n,\Delta)$ by considering the classes of $\mathrm{S}$-equivalence, and we have a Coarse moduli morphism
\[
\mathcal{M}^{\mathrm{ss, S}}_X(n,\Delta) \longrightarrow \mathrm{M}^{\mathrm{ss}}_X(n, \Delta);
\]
see for \cite{Seshadri} $k = \mathbb{C}$, and \cite{Faltings}, \cite{Hein} for arbitrary base field $k$. More precisely, for any $k$-scheme $S$, $\mathcal{M}^{\mathrm{ss, S}}_X(n,\Delta)(S)$ is the set of equivalence classes of all families $E$ over $X$, parameterized by $S$, such that for each closed point $s \in S$, the bundle $E_s$ is semistable of rank $n$, and $\mathrm{det}(E) \cong p_S^*N \otimes q_S^*(\Delta)$ for some invertible sheaf $N$ on $S$; where the equivalence relation is given by the usual equivalence of families together with the $\mathrm{S}$-equivalence $\sim_{\mathrm{S}}$. It should be mentioned that there are several ways to fix the determinant line bundle; \cite{Hoffmann} gives a comprehensive study of these different moduli problems and the associated moduli spaces, and due to their work, we make the following convention:
\begin{conv}\label{conv_2}
    In whatever follows, we say that a family $\mathcal{F}$ on $X$, parametrized by $S$, has `determinant $\Delta$' to mean that for each closed point $s \in S$, we have $\mathcal{F}_s \cong \Delta$.
\end{conv}

Our objective in this article is to parameterize orbifold semistable bundles of a given rank and whose determinants are isomorphic to a given orbifold line bundle. For this, we rely on a GIT-free construction following \cite{Faltings} and \cite{Hein}. We give a detailed account of this construction of $\mathcal{M}^{\mathrm{ss, S}}_X(n, \Delta)$ in~\Cref{sec_Falting_Construction}.

\subsection{Orbifold Curves}\label{sec_Orbifold_Curves}
As before, let $k$ denote an algebraically closed field of an arbitrary characteristic. The notion of a branch data was introduced and studied in \cite{P} for $k = \mathbb{C}$, further generalized in \cite{KP}. We review the important definitions and results. One new result is Lemma~\ref{lem_choice_tame_cover} which states that an orbifold curve is always dominated by a geometric orbifold curve in such a way that the induced cover is tamely ramified; this important result will be used extensively throughout.

A \textit{branch data} $P$ on a smooth projective connected $k$-curve $X$ is an association of a finite Galois field extension $P(x)$ of the local field $K_{X,x} = \mathrm{QF}(\widehat{\mathcal{O}}_{X,x})$ in a fixed algebraic closure of $K_{X,x}$ for finitely many closed points $x \in X$. The \textit{branch locus} $\text{\rm BL}(P)$ for a branch data $P$ on $X$ is defined to be the finite (possibly empty) set
\[
\mathrm{BL}(P) \coloneqq \{x \in X \, | \, P(x)/K_{X,x} \text{ is a non-trivial extension}\}.
\]
of closed points in $X$. A branch data $P$ with $\mathrm{BL}(P)$ being the empty set is referred to as the \textit{trivial branch data}. 

A \textit{formal orbifold curve} is a pair $(X,P)$, consisting of a smooth projective connected $k$-curve $X$ and a branch data $P$ on $X$. On a curve $X$, the collection of branch data has an ordering $Q \geq P$, which means that for any closed point $x \in X$, we have $Q(x) \supset P(x)$ as extensions of $K_{X,x}$. For any $Q \geq P$, the identity map $\mathrm{id}_X$ induces a cover $(X,Q) \longrightarrow (X,P)$ of formal orbifold curves. In particular, we obtain the induced cover $\iota \colon (X,P) \longrightarrow X$ where we view $X$ as the formal orbifold curve associated to the trivial branch data.

We refer to \cite[Section~3]{Das} for detailed definition, notation, and convention. It was noted in \cite[Appendix, Theorem~A.1]{Das} that a formal orbifold curve $(X,P)$ can be viewed as a smooth proper Deligne Mumford stack that admits the induced cover $\iota \colon (X,P) \longrightarrow X$ as its Coarse moduli map, and such that $\iota$ induces an isomorphism
\[
(X,P) \times_X \left( X - \mathrm{BL}(P) \right) \xrightarrow{\cong} X - \mathrm{BL}(P)
\]
of $k$-curves. Further, the stabilizer group $G_x$ associated to the residual gerbe corresponding to a point $x \in \text{\rm BL}(P)$ is the Galois group $\mathrm{Gal}\left( P(x)/K_{X,x} \right)$.

\begin{conv}
\emph{In whatever follows, we do not distinguish between a formal orbifold curve and the corresponding smooth proper Deligne Mumford stack; the later being referred to as an \textit{orbifold curve}.}    
\end{conv}

Given an orbifold curve $(X,P)$ with branch locus $\mathrm{BL}(P) = \{x_1, \ldots, x_r\}$, an explicit atlas of $(X,P)$ is provided in \cite[Lemma~3.11]{Das}; namely, there is an affine open cover $\{U_i\}_{0 \leq i \leq r}$ of $X$ with $U_0 = X - \text{\rm BL}(P)$, and for $1 \leq i \leq r$, the set $U_i$ contains $x_i$ and no other $x_j$, and the fiber product stack $U_i \times_X (X,P)$ is a quotient stack $[V_i/G_i]$ for some smooth affine irreducible curve $V_i$ admitting a generically faithful action by the stabilizer group $G_i \coloneqq G_{x_i} = \mathrm{Gal}\left( P(x_i)/K_{X,x_i} \right)$. When $(X,P)$ is globally a quotient stack $[Z/\Gamma]$ for some finite group $\Gamma$ and a $\Gamma$-Galois cover $Z \longrightarrow X$ of smooth projective connected curves, we say that $(X,P)$ is a \textit{geometric orbifold curve}, and $P$ is called a \textit{geometric branch data}.

By \cite[Proposition~2.37]{KP}, there is a geometric branch data $Q$ on $X$ such that $Q \geq P$. It should be noted that the choice of the geometric branch data $Q$ is not unique. When $P$ is not a geometric branch data, the induced cover $\tau \colon (X,Q) \longrightarrow (X,P)$ may be \textit{wildly ramified}, i.e. for some closed point $x \in X$, the group $\mathrm{Gal}\left(Q(x)/P(x) \right)$ has order divisible by $p$. In the following, we show that there is a choice of $Q$ as above such that $\tau$ is a \textit{tamely ramified} cover, i.e. $\mathrm{Gal}\left(Q(x)/P(x) \right)$ is a cyclic group of order coprime to $p$.

\begin{lemma}\label{lem_choice_tame_cover}
Let $(X,P)$ be a connected orbifold curve. Then there is a geometric branch data $\tilde{P}$ on $X$ with $\tilde{P} \geq P$ such that the induced cover $\tau \colon (X,\tilde{P}) \longrightarrow (X,P)$ is tamely ramified.
\end{lemma}

\begin{proof}
If $P$ is a geometric branch data, we take $\tilde{P} = P$. Suppose that $P$ is not a geometric branch data. Let $\mathrm{BL}(P)$ be its branch locus. From the theory of finite Galois extension of local fields (\cite[Chapter IV, Corollary 4]{Serre}), it follows that for each $x \in \mathrm{BL}(P)$, the stabilizer group $G_x \coloneqq \mathrm{Gal}(P(x)/K_{X,x})$ is of the form
\[
G_x \cong P_x \rtimes \mathbb{Z}/m_x \mathbb{Z}
\]
for some (possibly trivial) normal $p$-subgroup $P_x$ and a prime-to-$p$ integer $m_x$. Set
\[
B \coloneqq \{x \in \mathrm{BL}(P) \, | \, m_x \neq 1\}.
\]
Since $P$ is not geometric, necessarily $B \neq \emptyset$ (\cite[Corollary~2.33]{KP}). Let $B = \{x_1, \ldots, x_r\}$. Consider a finite set of closed points $\{y_1, \ldots, y_r\}$ in $X$, disjoint from $\mathrm{BL}(P)$. Define a branch data $\tilde{P}$ with branch locus $\mathrm{BL}(P) \sqcup \{y_1, \ldots, y_r\}$ as follows:
\begin{align*}
\tilde{P}(x) \coloneqq P(x) \text{ for } x \in \mathrm{BL}(P), \\
\tilde{P}(y_i) \coloneqq K_{X,y_i}[y_i,t_i]/(t_i^{m_{x_i}}-y_i) \text{ for } 1 \leq i \leq r;
\end{align*}
here we denote the local parameter at the point $y_i$ also by $y_i$, and $\tilde{P}(y_i)$ is the unique $\mathbb{Z}/m_{x_i}\mathbb{Z}$-Galois extension of $K_{X,y_i}$. Note that by \cite[Corollary~2.33]{KP}, the purely wild branch data $P_{\text{\rm wild}}$ on $X$ with branch locus $\mathrm{BL}(P) - B$, and defined by $P_{\mathrm{wild}}(x) = P(x)$ as extensions of $K_{X,x}$ where $x \in \mathrm{BL}(P) - B$, is geometric. For each $1 \leq i \leq r$, the branch data $Q_i$ on $X$ with branch locus $\{x_i, y_i\}$, defined by $Q_i(x_i) = P(x_i)$ and $Q_i(y_i) = \tilde{P}(y_i)$ is also geometric by the existence of Kummer-Harbater-Katz-Gabber covers associated to the local extension $P(x_i)/K_{X,x_i}$ (cf. \cite[Theorem~9.6]{DasIC}). As $\tilde{P}$ is the compositum of the branch data $P_{\mathrm{wild}}$ with all the $Q_i$'s, by \cite[Proposition~2.29]{KP}, $\tilde{P}$ is a geometric branch data on $X$, and $\tilde{P} \geq P$ by the construction. 
\end{proof}

\subsection{Orbifold Sheaves of Modules}
This section is devoted to the generalities of a quasi-coherent sheaf on an orbifold curve. We recall certain important notions used in this article following \cite{Olsson}, \cite{LMB}, and \cite{Das}. The notions and results in this section will be used throught our article, without further mention.

On an algebraic stack $\mathfrak{X}$ over a $k$-scheme $S$, there are several associated ringed topoi: big topoi $\mathfrak{X}_{\mathrm{Zar}}, \, \mathfrak{X}_{\mathrm{fppf}}, \, \mathfrak{X}_{\text{\rm \'{E}T}}$, and two small topoi $\mathfrak{X}_{\text{\rm lis-\'{e}t}}, \, \mathfrak{X}_{\text{\rm \'{e}t}}$; see \cite[Chapter~9]{Olsson} and \cite[Chapter~12]{LMB}. On $\mathfrak{X}$, the structure sheaf $\mathcal{O}_{\mathfrak{X}}$ gives a ringed topos structure on each of these topos. For a Deligne Mumford stack $\mathfrak{X}$, the respective full subcategories of quasi-coherent $\mathcal{O}_{\mathfrak{X}}$-modules in the above mentioned ringed topoi are all equivalent; see \cite[Proposition~9.1.18, page 195]{Olsson}. If $\mathfrak{X}$ is a scheme, this equivalence extends to the small Zariski topos $\mathfrak{X}_{\mathrm{zar}}$ as well. Under these equivalences, we will only consider the small \'{e}tale topos on an orbifold curve $(X,P)$. We write $\mathrm{Coh}_{\mathcal{O}_{(X,P)}}$ and $\mathrm{QCoh}_{\mathcal{O}_{(X,P)}}$ for the full subcategories of the category of $\mathcal{O}_{(X,P)}$-modules $\mathrm{Mod}_{\mathcal{O}_{(X,P)}}$, consisting of the coherent and quasi-coherent sheaves of $\mathcal{O}_{(X,P)}$-modules, respectively. The categories $\mathrm{Coh}_{\mathcal{O}_{(X,P)}}$ and $\mathrm{QCoh}_{\mathcal{O}_{(X,P)}}$ admit tensor structures, induced from the abelian tensor category $\mathrm{Mod}_{\mathcal{O}_{(X,P)}}$. For $E \in \mathrm{Coh}_{\mathcal{O}_{(X,P)}}$, the functor $\sheafhom_{\mathcal{O}_{(X,P)}}(E,-)$ is a right adjoint to the tensor product functor $E \otimes_{\mathcal{O}_{(X,P)}} -$. Our interest is on the locally free sheaves of $\mathcal{O}_{(X,P)}$-modules, which we will refer to as \textit{bundles}. This is a full subcategory of $\mathrm{Coh}_{\mathcal{O}_{(X,P)}}$, denoted by $\mathrm{Vect}(X,P)$. We note that for $E \in \mathrm{Vect}(X,P)$, the coherent sheaf
\[
E^\vee \coloneqq \sheafhom_{\mathcal{O}_{(X,P)}}(E, \mathcal{O}_{(X,P)})
\]
is again a bundle. In fact, the tensor product functor $E \otimes_{\mathcal{O}_{(X,P)}} -$ and the sheaf hom-functor $\sheafhom_{\mathcal{O}_{(X,P)}}(E,-) \cong E^\vee \otimes -$ are exact functors.

Let $f \colon (Y,Q) \longrightarrow (X,P)$ be a cover of orbifold curves. We have the induced functors
\begin{equation*}
f_* \colon \mathrm{Mod}_{\mathcal{O}_{(Y,Q)}} \longrightarrow \mathrm{Mod}_{\mathcal{O}_{(X,P)}}, \, \text{\rm and } \, f^* \colon \mathrm{Mod}_{\mathcal{O}_{(X,P)}} \longrightarrow \mathrm{Mod}_{\mathcal{O}_{(Y,Q)}}.    
\end{equation*}
These functors define an adjoint pair $(f^*,f_*)$; see \cite[Proposition~9.3.6, pg. 205]{Olsson}. Since $f$ is a quasi-separated and quasi-compact morphism, the above adjoint pair restricts to an adjoint pair of functors between quasi-coherent sheaves of modules; since $f_*$ and $f^*$ also preserve coherent and locally free sheaves, we get further restrictions to adjoint pairs of functors between coherent sheaves of modules and between bundles.

Since $f$ is a flat morphism, $f^*$ is an exact functor that is strong symmetric monoidal, which means that we have isomorphisms
\[
f^*\mathcal{O}_{(X,P)} \cong \mathcal{O}_{(Y,Q)}, \text{ and } f^*(E \otimes F) \cong f^*E \otimes f^*F
\]
for any $E, \, F \in \mathrm{QCoh}_{\mathcal{O}_{(X,P)}}$. By \cite[Proposition~1.12]{Adjoint}, for $E \in \mathrm{Vect}(X,P)$, and $F \in \mathrm{QCoh}_{\mathcal{O}_{(X,P)}}$, the natural map
\[
f^* \sheafhom_{\mathcal{O}_{(X,P)}}(E,F) \longrightarrow \sheafhom_{\mathcal{O}_{(Y,Q)}}(f^*E,f^*F)
\]
is an isomorphism, and by \cite[Proposition~A.3]{Das}, the projective formula holds, i.e., for any $E' \in \mathrm{QCoh}_{\mathcal{O}_{(Y,Q)}}$, the natural map
\[
E \otimes f_*E' \longrightarrow f_* (f^*E \otimes E')
\]
is an isomorphism. Another important property of the $(f^*,f_*)$-adjoint pair is that when $f$ is faithfully flat, for any $E \in \mathrm{QCoh}_{\mathcal{O}_{(X,P)}}$, the canonical injective map
\[
f^\# \colon E \longrightarrow f_*f^*E
\]
obtained by evaluating the unit $\mathrm{id} \Rightarrow f_*f^*$ of the $(f^*,f_*)$-adjoint at $E$ is \textit{an $f$-locally split injection} --- this means that the injective map $f^*f^{\#} \colon f^*E \hookrightarrow f^*f_*f^*E$ splits; see \cite[Proposition~3.5.4(i)]{split}.

We also recall from \cite[9.2.5, page 198, and Proposition~9.2.16]{Olsson} that $f_*$ is a left exact functor, and we can form higher direct images $\mathrm{R}^if_*$ for $i \geq 0$, with $\mathrm{R}^0f_* = f_*$: for any $E \in \mathrm{QCoh}_{\mathcal{O}_{(Y,Q)}}$, the sheaf $\mathrm{R}^if_* E \in \mathrm{QCoh}_{\mathcal{O}_{(X,P)}}$ associates $\mathrm{H}^i(V \times_{(X,P)} (Y,Q), \mathrm{pr}^*E)$ to any atlas $V \longrightarrow (X,P)$ where $\mathrm{pr}$ denote the projection $V \times_{(X,P)} (Y,Q) \longrightarrow (Y,Q)$. Moreover, the higher direct image sheaves are coherent if $E$ is coherent. The functor $f_*$ is not right exact, in general. Although, in the following two important cases, $f_*$ is exact.
\begin{enumerate}[leftmargin=*]
    \item The cover $f$ is representable in the sense of a morphism of algebraic stacks; see \cite[Lemma~3.4]{Das}.
    \item The cover $f$ is tamely ramified, i.e. for each point $y \in \mathrm{BL}(Q)$, the Galois field extension $Q(y)/P(f_0(y))$ is a cyclic extension of order co-prime-to $p$ where $f_0 \colon Y \longrightarrow X$ is the cover induced on the Coarse moduli curves. To see this statement, we note that for any cover $f$, the geometric fibers of the relative inertia stack $\mathcal{I}_f \longrightarrow (Y,Q)$ are constant group schemes associated to the groups $\mathrm{Gal}\left( Q(y)/P(f_0(y))\right)$. By \cite[Proposition~2.10]{Dan}, these group schemes are linearly reductive if and only if they are of order prime-to-$p$ (hence, cyclic). So the cover $f$ is tamely ramified if and only if it is a tame morphism in the sense of algebraic stacks, and a relative version of \cite[Theorem~3.2]{Dan} shows that $f_*$ is an exact functor. In fact, this the reason we produced the tamely ramified cover in \Cref{lem_choice_tame_cover}.
\end{enumerate}

The notion of the rank of a bundle (or even a coherent sheaf) is defined generically; so, for $E \in \mathrm{Coh}_{\mathcal{O}_{(X,P)}}$, the rank $\mathrm{rk}(E) = \mathrm{rk}(\iota_*E)$ where $\iota$ is the induced cover $\iota \colon (X,P) \longrightarrow X$. From \cite[Lemma~3.21 and Lemma~3.24]{Das}, we see that the pushforward coherent sheaf $f_*F$ on $(X,P)$ has rank $\mathrm{deg}(f_0) \cdot \mathrm{rk}(E) =\mathrm{rk}(\iota_*f_*F)$ for any $F \in \mathrm{Coh}(Y,Q)$ with $f_0 \colon Y \longrightarrow X$ being the cover induced by $f$ on the Coarse moduli curves. Moreover, $f_*F$ is a bundle if $F$ is so. Let us give a more detailed description of the higher direct images under the Coarse moduli map.

\begin{lemma}\label{lem_R1}
Let $(X,P)$ be an orbifold curve with the Coarse moduli map $\iota \colon (X,P) \longrightarrow X$. For any $E \in \mathrm{Coh}_{\mathcal{O}_{(X,P)}}$, the following hold.
\begin{enumerate}
\item $\iota_*E$ is a coherent sheaf on $X$ of rank $\mathrm{rk}(\iota_*E) = \mathrm{rk}(E)$. For $i \geq 1$, the coherent sheaf $\mathrm{R}^i\iota_*E$ is a torsion sheaf of $\mathcal{O}_X$-modules.\label{high:1}
\item Assume that $E \in \mathrm{Vect}(X,P)$. For any Galois cover $g_0 \colon Z \longrightarrow X$ of smooth projective connected curves that factors as $g_0 \colon Z \overset{g}\longrightarrow (X,P) \overset{\iota}\longrightarrow X$ such that $g$ is tamely ramified, the torsion sheaf $\mathrm{R}^1\iota_*E$ is the cokernel of the natural injection
\[
0 \longrightarrow g_{0,*} g^* E/ \iota_* E \longrightarrow \iota_* \left( g_* g^*E / E \right),
\]
and for all $i \geq 2$, we have the isomorphisms
\begin{center}
  $\mathrm{R}^i\iota_* E \cong \mathrm{R}^{i-1}\iota_*\left( g_*g^*E/E \right)$.  
\end{center}\label{high:2}
\end{enumerate}
\end{lemma}

\begin{proof}
We may assume that $P$ is a non-trivial branch data. We first claim that for any quasi-coherent sheaf $E$, and for each $i \geq 1$, the quasi-coherent sheaf $\mathrm{R}^i\iota_*E$ vanishes away from the branch locus $\mathrm{BL}(P)$. To see this, consider the affine curve $U_0 = X - \mathrm{BL}(P)$. Then the sheaf $\mathrm{R}^i\iota_*E$ on $U_0$ is given by $\mathrm{H}^i(U_0 \times_X (X,P), \mathrm{pr}^* E)$ where $\mathrm{pr} \colon U_0 \times_X (X,P) \longrightarrow (X,P)$ is the projection map. Since $\mathrm{pr}^* E$ is a quasi-coherent sheaf on the affine curve $U_0 \times_X (X,P) \cong U_0$, we have $\mathrm{H}^i(U_0 \times_X (X,P), \mathrm{pr}^* E) = 0$ for all $i \geq 0$. This proves the claim. Statement~\eqref{high:1} follows directly from this and our previous discussion.

To see \eqref{high:2}, consider a $\Gamma$-Galois cover
\[g_0 \colon Z \overset{u}\longrightarrow [Z/\Gamma] = (X,\tilde{P}) \overset{\tau}\longrightarrow (X,P) \overset{\iota}\longrightarrow X\]
where $u$ is the natural atlas, $\tau$ is a tamely ramified cover, $g = \tau \circ u$, and $\tilde{\iota} \coloneqq \iota \circ \tau$ is the Coarse moduli map. Let $E \in \mathrm{Vect}(X,P)$. Set $\tilde{E} \coloneqq \tau^*E \in \mathrm{Vect}(X,\tilde{P})$. Using the canonical injection $\tilde{\iota}^{\, \#} \colon \tilde{E} \hookrightarrow u_*u^*\tilde{E}$, we have a short exact sequence
\[
0 \longrightarrow \tilde{E} \longrightarrow u_*u^*\tilde{E} \longrightarrow Q \coloneqq \left( u_*u^*\tilde{E} \right) /\tilde{E} \longrightarrow 0
\]
of bundles on $(X, \tilde{P})$. Since $g_0$ is a finite map of schemes, $g_{0,*}$ is an exact functor by \cite[\href{https://stacks.math.columbia.edu/tag/03QP}{Proposition 03QP}]{SP}. Also since $u$ is an \'{e}tale cover, the functor $u_*$ is exact. In particular, for all $i \geq 1$, we have
\[
\mathrm{R}^i\tilde{\iota}_*(u_*u^*\tilde{E}) \cong \mathrm{R}^i\iota_*(g_{0,*}g_0^*E) = 0.
\]
So the above exact sequence produces a long exact sequence
\begin{equation}\label{eq_exacttilde1}
0 \longrightarrow \tilde{\iota}_*\tilde{E} \longrightarrow g_{0,*}u^*\tilde{E} \longrightarrow \tilde{\iota}_* Q \longrightarrow \mathrm{R}^1\tilde{\iota}_* \tilde{E} \longrightarrow 0,
\end{equation}
and for all $i \geq 2$, the isomorphisms
\begin{equation}\label{eq_exacttilde2}
\mathrm{R}^i\tilde{\iota}_*\tilde{E} \cong \mathrm{R}^{i-1}\tilde{\iota}_*Q.
\end{equation}

Now the covers $\tau$ and $g$ are tamely ramified; hence, $\tau_*$ and $g_*$ are also exact functors. Further, \cite[Lemma~3.24]{Das} shows that $\tau_*\mathcal{O}_{(X, \tilde{P})} \cong \mathcal{O}_{(X,P)}$. By the flat base change and the projection formula (\cite[Proposition~A.3]{Das}), we obtain the following isomorphisms for $i \geq 0$.
\begin{eqnarray*}
\mathrm{R}^i\tilde{\iota}_*\tilde{E} \cong \mathrm{R}^i\iota_*E, \\
\mathrm{R}^i\tilde{\iota}_*(u_*u^*\tilde{E}) \cong \mathrm{R}^i\iota_*(g_*g^*E), \\
\mathrm{R}^{i}\tilde{\iota}_*Q \cong \mathrm{R}^i\iota_*\left( g_*g^*E/E \right).
\end{eqnarray*}
Using these isomorphisms, the exact sequence~\eqref{eq_exacttilde1} becomes:
\[
0 \longrightarrow \iota_*E \longrightarrow \iota_*(g_*g^*E) \cong g_{0,*}g^*E \longrightarrow \iota_* \left(g_*g^*E/E \right) \longrightarrow \mathrm{R}^1\iota_*E \longrightarrow 0,
\]
and the isomorphisms~\eqref{eq_exacttilde2} become:
\[
\mathrm{R}^i\iota_*E \cong \mathrm{R}^{i-1}\iota_* \left( g_*g^*E/E \right)
\]
for $i \geq 2$. In particular, the coherent torsion sheaf $\mathrm{R}^1\iota_*E$ is the cokernel of the natural inclusion
\[
g_{0,*}g^*E/\iota_*E \cong \mathrm{ker}\left( \iota_*\left( g_*g^*E/E \right) \rightarrow \mathrm{R}^1\iota_*E \right) \hookrightarrow \iota_*\left(g_*g^*E/E \right)
\]
of bundles. Inductively applying the above for $i \geq 1$, \eqref{high:2} follows.
\end{proof}

\section{Stability Conditions on Orbifold Bundles}\label{sec_Orbifolds}
In this section, we pose our moduli problem. To understand the moduli problem, we prove results on the descent of an orbifold bundle related to our context (cf. Proposition~\ref{prop_descent_unique_sub}), recall the definitions and properties of orbifold slope stability, establish important results regarding the cohomologies of orbifold bundles (see Lemma~\ref{lem_cohomologies_using_geometric}, Corollary~\ref{cor_cohomologies}), and provide a comparison of the $\mathrm{S}$-equivalence under certain covers of orbifold curves (see Proposition~\ref{prop_S-equivariance_wrt_set_up}).

\subsection{Descent and cohomologies of Orbifold Bundles}\label{sec_descent}
We start by collecting crucial results on orbifold bundles which are well known for bundles on curves. We also note certain useful properties of the cohomologies of orbifold bundles in our setup. Let us fix the following notation throughout this section.

\begin{notation}\label{not_equivariant_set_up}
Let $P$ be a non-trivial branch data on a smooth projective connected $k$-curve $X$. Fix a geometric branch data $\tilde{P}$ on $X$ satisfying $\tilde{P} \geq P$ such that the induced cover $\tau \colon (X,\tilde{P}) \longrightarrow (X,P)$ of orbifold curves is tamely ramified; this can be done by Lemma~\ref{lem_choice_tame_cover}. Also fix a $\Gamma$-Galois cover $g_0 \colon Z \longrightarrow X$ of smooth projective connected curves (for a finite group $\Gamma$) that factors as a composition
\begin{equation}\label{eq_equivariant_map_factorization}
g_0 \colon Z \overset{u}\longrightarrow [Z/\Gamma] = (X,\tilde{P}) \overset{\tau}\longrightarrow (X,P) \overset{\iota}\longrightarrow X
\end{equation}
where $u$ is the natural atlas, $g \coloneqq \tau \circ u$ is a representable cover, and $\iota$ is the Coarse moduli map.
\end{notation}

Let $\zeta$ denote the generic point of $X$ and $E \in \mathrm{Vect}(X)$. Then $E_\zeta = E \otimes_{\mathcal{O}_X} \mathcal{O}_{X,\zeta}$ is a $k(X) = \mathcal{O}_{X,\zeta}$-vector space. Any coherent sub-sheaf $F$ of $E$ is a coherent torsion-free sheaf; i.e., $F$ is a bundle. Recall that a sub-bundle of $E$ is a sub-sheaf $F \subset E$ such that $\mathrm{coker}\left( F \rightarrow E \right)$ is a bundle. \cite[Proposition~1]{Langton} states that given a $k(X)$-linear subspace $W \subset E_\zeta$, there is a unique sub-bundle $F \subset E$ such that $F_\zeta = W$. We use this fact to obtain similar results for orbifold bundles.

Since the orbifold curve $(X,P)$ is generically isomorphic to $X$, the point $\zeta$ is also the generic point of $(X,P)$. Further, for any bundle $E$ on $(X,P)$, we have the $k(X)$-linear vector space $E_{\zeta}$, any coherent sub-sheaf $F$ of $E$ is a bundle on $(X,P)$, and a sub-bundle of $E$ is a coherent sub-sheaf $F \subset E$ such that $\mathrm{coker}(F \rightarrow E) \in \mathrm{Vect}(X,P)$.

\begin{proposition}\label{prop_descent_unique_sub}
Let $\zeta$ be the generic point of $X$. The following hold.
\begin{enumerate}
\item For $F \in \mathrm{Vect}(X,\tilde{P})$, and a $k(X)$-linear subspace $W \subseteq F_{\zeta}$, there is a unique sub-bundle $F^{(W)} \subseteq F$ such that $F^{(W)}_\zeta = W$.\label{des:1}
\item Let $E \in \mathrm{Vect}(X,P)$ and $F' \subseteq \tau^* E$ be a sub-bundle. Then there exists a unique sub-bundle $E' \subseteq E$ such that $\tau^* E' \cong F'$.\label{des:2}
\item For any $E \in \mathrm{\rm Vect}(X,P)$, and a $k(X)$-linear subspace $V \subseteq E_{\zeta}$, there is a unique sub-bundle $E^{(V)} \subseteq E$ such that $E^{(V)}_\zeta = V$.\label{des:3}
\end{enumerate}
\end{proposition}

\begin{proof}
Let $\eta$ denote the generic point of $Z$. We have a $\Gamma$-invariant $k(Z)$-linear subspace
\[
W \otimes_{k(X)} k(Z) \subseteq F_{\zeta} \otimes_{k(X)} k(Z) \cong (u^*F)_{\eta}.
\]
By \cite[Proposition~1]{Langton}, there is a unique sub-bundle $S \subseteq u^*F$ on $Z$ such that $S_\eta = W \otimes_{k(X)} k(Z)$. Moreover, on any open affine $U \subset Z$, we have
\[
S(U) = \left( u^*F \right) (U) \cap \left( W \otimes_{k(X)} k(Z) \right).
\]
In particular, taking $\Gamma$-invariant affine open subsets $U \subset Z$, we conclude that $S$ is a $\Gamma$-equivariant sub-bundle of $u^*F$ with $S_\eta = W \otimes_{k(X)} k(Z)$. Since $u^* \colon \mathrm{Vect}(X,\tilde{P}) \longrightarrow \mathrm{Vect}^{\Gamma}(Z)$ is an equivalence of categories with a quasi inverse given by the invariant pushforward $u_*^\Gamma$, there exists a unique sub-bundle $F^{(W)} \subseteq F$ such that $u^*F^{(W)} \cong S$ as $\Gamma$-equivariant bundles. In particular, $F^{(W)}_\zeta = W$, proving~\eqref{des:1}.

Now let $E \in \mathrm{Vect}(X,P)$ and $F' \subseteq \tau^* E$ be a sub-bundle. We claim that $F'$ is isomorphic to $\tau^* \tau_* F'$. The inclusion homomorphism $j \colon F' \hookrightarrow \tau^*E$ induces $\tau_* j \colon \tau_* F' \hookrightarrow \tau_* \tau^* E \cong E$, which further induces an inclusion $\tau^* \tau_* j \colon \tau^* \tau_* F' \hookrightarrow \tau^* E$ as coherent sheaves. As $\tau_*$ is an exact functor, we have
\[
\mathrm{coker}(\tau_* F' \overset{\tau_* j}\longrightarrow \tau_* \tau^* E) \cong \tau_* \mathrm{coker}(F' \overset{j}\longrightarrow \tau^*E),
\]
a bundle on $(X,P)$. So $\mathrm{coker}(\tau^*\tau_* F' \overset{\tau^* \tau_* j}\longrightarrow \tau^* \tau_* \tau^* E \cong \tau^* E) \cong \tau^* \tau_* \mathrm{coker}(F' \overset{j}\longrightarrow \tau^*E)$ is a bundle on $(X,\tilde{P})$. Thus $F'$ and $\tau^* \tau_* F'$ are both sub-bundles of $\tau^*E$ of the same rank. Again using the exactness of $\tau_*$, the canonical adjoint morphism
\[
\alpha \colon \tau^* \tau_* F' \longrightarrow F'
\]
is an epimorphism. So $\alpha$ is an isomorphism. Taking $E' \coloneqq \tau_* F'$, and by the uniqueness of the pullback bundles, the statement~\eqref{des:2} follows.

The statement~\eqref{des:3} is the consequence of the first two statements applied to the bundle $F = \tau^*E$.
\end{proof}

Next, we go through the results on the cohomologies of quasi-coherent sheaves on $(X,P)$ --- the first approach is to relate them with the classical cohomologies on $X$, and the second one is to use the equivariant set up. The approach helps us calculate the dimension of the cohomologies in suitable cases, whereas the second theory is more significant for our purpose.

There is a first quadrant Leray spectral sequence (cf. \cite[Theorem~A.1.6.4]{Brochard})
\begin{equation}\label{eq_Leray}
E_2^{p,q} = \mathrm{H}^p( X, R^q \iota_* E) \Rightarrow \mathrm{H}^{p+q}((X,P),E).
\end{equation}
In particular, we have
\begin{eqnarray}\label{eq_cohomology_Coarse}
 & \mathrm{H}^0((X,P),E) \cong \mathrm{H}^0(X, \iota_* E), \, \text{ and } \\
 & \text{ an exact sequence: } 0 \rightarrow \mathrm{H}^1(X, \iota_* E) \rightarrow \mathrm{H}^1 ((X,P) , E) \rightarrow \mathrm{H}^0(X, R^1\iota_* E) \rightarrow 0
\end{eqnarray}
of $k$-vector spaces.

To employ the equivariant approach, we first note that the cohomologies of quasi-coherent sheaves on $(X,P)$ are the same as those of the pullback sheaves on $(X,\tilde{P})$.

\begin{lemma}\label{lem_cohomologies_using_geometric}
Let $E \in \mathrm{QCoh}_{\mathcal{O}_{(X,P)}}$. For any $i \geq 0$, we have a canonical isomorphism
$$\mathrm{H}^i((X, \tilde{P}), \tau^*E) \cong \mathrm{H}^i((X,P), E).$$
\end{lemma}

\begin{proof}
Let $\eta_X \colon X \longrightarrow \mathrm{Spec}(k)$ denote the structure morphism. Then for any $\tilde{E} \in \mathrm{QCoh}_{\mathcal{O}_{(X,\tilde{P})}}$, and for $i \geq 0$, we have the canonical isomorphism
\[
\mathrm{R}^i(\tilde{\iota} \circ \eta_X)_*(\tilde{E}) \cong \mathrm{H}^i((X,\tilde{P}), \tilde{E}) \otimes_k \mathcal{O}_k.
\]
Similarly, for any $E \in \mathrm{QCoh}_{\mathcal{O}_{(X,P)}}$, and for $i \geq 0$, we have the isomorphism
\[
\mathrm{R}^i(\iota_* \circ \eta_X)_*(E) \cong \mathrm{H}^i((X,P),E) \otimes_k \mathcal{O}_k.
\]
By the projection formula (\cite[Proposition~A.3]{Das}), we have
\[
R^q\tau_*(\tau^* E) \cong E \otimes R^q \tau_* (\mathcal{O}_{(X,\tilde{P})})
\]
for all $q \geq 0$. Since $\tau_*$ is an exact functor, $R^q\tau_*(\tau^* E)$ vanishes for all $q \geq 1$. Since $\tau_* \mathcal{O}_{(X,\tilde{P})} \cong \mathcal{O}_{(X,P)}$ (\cite[Lemma~3,24]{Das}), we have $\tau_*(\tau^* E) \cong E$. Thus,
\[
\mathrm{R}^i(\tilde{\iota} \circ \eta_X)_*(\tau^*E) \cong \mathrm{R}^i (\iota \circ \eta_X)_* \left( \tau_* \tau^* E \right) \cong \mathrm{R}^i(\iota_* \circ \eta_X)_*E,
\]
proving the result.
\end{proof}

In view of the above result, we are only interested in computing the cohomologies of the pullback sheaves on the geometric orbifold curve $(X,\tilde{P})$. The cohomologies of coherent sheaves on $(X,\tilde{P})$ can also be computed using a Cartan-Leray spectral sequence associated to the atlas $u \colon Z \longrightarrow (X,\tilde{P})$ which we recall following \cite[\href{https://stacks.math.columbia.edu/tag/06X7}{Section 06X7}]{SP} and \cite[Section~A.1.4]{Brochard}. As before, let $E \in \mathrm{QCoh}_{\mathcal{O}_{(X,P)}}$.

Let $n \geq 0$. Consider the $(n+1)$-fold product $Z_n = Z \times_{(X,\tilde{P})} \cdots \times_{(X,\tilde{P})} Z$. Since $u$ is a $\Gamma$-Galois \'{e}tale cover, we have $Z_0 = Z$, and for $n \geq 1$, the fiber product $Z_n$ is isomorphic to $\Gamma^{(n)} \times Z$ where $\Gamma^{(n)}$ is group scheme over $k$ associated to the product group $\Gamma \times \cdots \times \Gamma$. So $Z_n$ is a disjoint union of copies of $Z$, parametrized by elements of $\Gamma^{(n)}$. Write $\mathrm{pr}_n$ for the projection of $Z_n$ onto $Z$ and $\phi_n \colon Z_n \longrightarrow (X,\tilde{P})$ for the composition of $\mathrm{pr}_n$ followed by $u$. Also set $\mathcal{E}_n \coloneqq \phi_n^*\tau^*E \cong \mathrm{pr}_n^* g^* E$. In particular, $\mathcal{E}_0 = g^*E$. Then we have the exact sequence (the extended relative \v{C}ech complex)
\begin{equation}\label{eq_extended_Cech_cmplx}
0 \longrightarrow \tau^* E \longrightarrow \phi_{0,*}\phi_0^* \tau^*E \longrightarrow \phi_{1,*}\phi_1^* \tau^*E \longrightarrow \phi_{2,*}\phi_2^* \tau^*E \longrightarrow \cdots.
\end{equation}
For each $q \geq 0$, the above exact sequence~\eqref{eq_extended_Cech_cmplx} induces a complex
\begin{equation}\label{sequence}
    \tag*{{$\mathrm{s(H^q)}:$}} \mathrm{H}^q(Z_0,\mathcal{E}_0) \overset{d^0}\longrightarrow \mathrm{H}^q(Z_1, \mathcal{E}_1) \overset{d^1}\longrightarrow \cdots \overset{d^{n-1}}\longrightarrow \mathrm{H}^q(Z_n, \mathcal{E}_n) \overset{d^n}\longrightarrow \cdots
\end{equation}
Since each $Z_n$ is a disjoint union of copies of the a smooth projective connected $k$-curve $Z$, the complex $s(\mathrm{H}^q)$ is identically zero for $q \geq 2$. Set
\[
\check{\mathrm{H}}^p (\mathrm{H}^q (Z_., \mathcal{E}_.)) \coloneqq \mathrm{ker}(d^p)/ \mathrm{im}(d^{p-1}).
\]

The following spectral sequence computes the cohomologies of $\tau^*E$, and hence of $E$.

\begin{proposition}[{\cite[Proposition~A.1.4.1]{Brochard}, \cite[\href{https://stacks.math.columbia.edu/tag/06XG}{Lemma 06XG}]{SP}}]\label{prop_Cartan_Leray_ss}
Under the above notation, there is a first quadrant spectral sequence
\[
E_2^{p,q} = \check{\mathrm{H}}^p (\mathrm{H}^q (Z_., \mathcal{E}_.)) \Rightarrow \mathrm{H}^{p+q}((X,\tilde{P}), \tau^*E).
\]
\end{proposition}

An immediate calculation of the above spectral sequence produces the following.

\begin{corollary}\label{cor_cohomologies}
Under the above notation, the following hold.
\begin{enumerate}
\item $\mathrm{H}^0((X,P), E) = \mathrm{ker} \left( \mathrm{H}^0(Z, g^* E) \rightarrow \mathrm{H}^0(Z \times_{\widetilde{\mathfrak{X}}} Z, \mathrm{ pr}_1^* g* E) \right) = \mathrm{H}^0(Z, g^* E)^{\Gamma}.$\label{i:1}
\item There is an exact sequence
\begin{multline*}
0 \longrightarrow \check{\mathrm{H}}^1 (\mathrm{H}^0 (Z_., \mathcal{E}_.)) \longrightarrow \mathrm{H}^1((X,P), E) \longrightarrow \mathrm{ker} \left( \mathrm{H}^1(Z, g^*E) \rightarrow \mathrm{H}^1(Z_1, \mathcal{E}_1) \right)\\
 \longrightarrow \check{\mathrm{H}}^2 (\mathrm{H}^0 (Z_., \mathcal{E}_.)) \longrightarrow \mathrm{H}^2((X,P), E).    
\end{multline*}
\end{enumerate}\label{i:2}
In particular, if $E \in \mathrm{Vect}(X,P)$ such that $\mathrm{H}^0(Z, g^*E) = 0$, then $\mathrm{H}^1((X,P), E)$ is the sub-space of $\Gamma$-fixed points in $\mathrm{H}^1(Z,g^*E)$.
\end{corollary}

\begin{proof}
The statement follows by Lemma~\ref{lem_cohomologies_using_geometric}, the spectral sequence in Proposition~\ref{prop_Cartan_Leray_ss} together with the observation that
\[
\mathrm{H}^0((X,P),E) = \mathrm{H}^0(X, \iota_* E) = g_{0,*}^\Gamma (g^* E)
\]
from Equation~\eqref{eq_Leray}.

Now suppose that $\mathrm{H}^0(Z, g^*E) = 0$ for $E \in \mathrm{Vect}(X,P)$. Then by the base change theorem (\cite[Proposition~A.3]{Das}), $\mathrm{H}^0(Z_1, \mathcal{E}_1) \cong \mathrm{H}^0(Z, \oplus_{\gamma \in \Gamma} \, \gamma^* g^*E)$ also vanishes. From the complex~\eqref{sequence}, we conclude that $\check{\mathrm{H}}^1 (\mathrm{H}^0 (Z_., \mathcal{E}_.)) = 0$. By~\eqref{i:2}, the group $\mathrm{H}^1((X,P), E)$ is the kernel of the natural map
\[
\mathrm{H}^1(Z, g^*E) \longrightarrow \mathrm{H}^1(Z_1, \mathcal{E}_1) \cong \mathrm{H}^1(Z, \oplus_{\gamma \in \Gamma} \, \gamma^* g^*E),
\]
proving the last statement.
\end{proof}

\subsection{Orbifold Semistable Bundles}
Let us briefly outline the definition and properties of slope stability for bundles on orbifold curves; our main reference is \cite{Das}.

Let $(X,P)$ be an orbifold curve with the associated Coarse moduli map $\iota \colon (X,P) \longrightarrow X$. Recall that the rank $\mathrm{rk}(E)$ of a bundle $E$ on $(X,P)$ is defined to be the rank $\mathrm{rk}(\iota_* E)$ of the bundle $\iota_*E$ on $X$. Associated to a bundle $E$ of rank $n$, we have the determinant line bundle $\mathrm{det}(E) = \wedge^n E \in \mathrm{Pic}(X,P)$. Every line bundle on $(X,P)$ is of the form $\mathcal{O}_{(X,P)}(D)$ for a Weil divisor $D$ on $(X,P)$, up to a linear equivalence; cf. \cite[Lemma~5.4.5]{VZB}. So the $P$-\textit{degree} of $E$ is defined to be
\[
\mathrm{deg}_P(E) \coloneqq \mathrm{deg}_P(\mathrm{det}(E)) = \sum_{ 1 \leq i \leq t} \, n_i \, \mathrm{deg}_P(x_i) \in \mathbb{Q}
\]
where $\mathrm{det}(E) \cong \mathcal{O}_{(X,P)} \left( \sum_{1 \leq i \leq t} n_i x_i \right)$ for some residual gerbes $x_i$'s, $n_i \in \mathbb{Z}$, and the $P$-degree of any residual gerbe $x$ is defined to be the rational number $\frac{1}{|\mathrm{Gal}\left( P(\iota(x))/K_{X,\iota(x)} \right)|}$. For $E \in \mathrm{Vect}(X,P)$, its $P$-\textit{slope} is the rational number
\[
\mu_P(E) \coloneqq \frac{\mathrm{deg}_P(E)}{\mathrm{rk}(E)}.
\]
A bundle $E$ on $(X,P)$ is said to be $P$-\textit{semistable} if for every proper sub-bundle $F \subset E$, we have
\[
\mu_P(F) \, \leq \, \mu_P(E).
\]
The bundle $E$ in the above definition is $P$-\textit{stable} if the inequality is strict for every proper sub-bundle $F \subset E$, and is $P$-\textit{polystable} if $E$ is a finite direct sum $\oplus_{1 \leq j \leq r} E_i$ of $P$-stable bundles $E_j$ satisfying $\mu_P(E_j) = \mu_P(E)$ for all $j$.

The existence of the unique Harder-Narasimhan filtration of a bundle on $(X,P)$ and a Jordan H\"{o}lder filtration for a $P$-semistable bundle are established in \cite[Proposition~4.8]{Das}. More precisely, any $E \in \mathrm{Vect}(X,P)$ has a maximal $P$-semistable sub-bundle $\mathrm{HN}_1(E)$, called the maximal destabilizing sub-bundle, determined uniquely by the following property: for any sub-bundle $F$ of $E$, we have $\mu_P(\mathrm{HN}_1(E)) \geq \mu_P(F)$, and if equality holds, $F \subset \mathrm{HN}_1(E)$. Using the existence of the maximal destabilizing sub-bundles, any $E \in \mathrm{Vect}(X,P)$ admits a unique Harder-Narasimhan filtration
\[ 0 = \mathrm{HN}_0(E) \subset \mathrm{HN}_1(E) \subset \mathrm{HN}_2(E) \subset \cdots \subset \mathrm{HN}_r(E) = E\]
where each subsequent quotient bundles $\mathrm{HN}_{i+1}(E)/\mathrm{HN}_i(E)$ are $P$-semistable, satisfying
\[
\mu_{\mathrm{max},P}(E) \coloneqq \mu_P(\mathrm{HN}_1(E)) > \mu_P(\mathrm{HN}_2(E)/\mathrm{HN}_1(E)) > \cdots > \mu_P(E/\mathrm{HN}_{r-1}(E)).
\]
Moreover, each $P$-semistable bundle $E$ on $(X,P)$ admits a (non-canonical) Jordan-H\"{o}lder filtration
\[ 0 = E_0 \subset E_1 \subset \cdots \subset E_l = E \]
by sub-bundles such that each subsequent quotient bundle $E_{i+1}/E_i$ are $P$-stable of $P$-slope $\mu_P(E)$, \, $0 \leq i \leq l-1$. The associated graded bundle $\mathrm{gr}(E) \coloneqq \oplus_{\substack{0 \leq i \leq l-1}} \, E_{i+1}/E_i$ does not depend on the choice of a Jordan-H\"{o}lder filtration, up to a canonical isomorphism. So we have a well defined notion of $\mathrm{S}$-\textit{equivalence}: two $P$-semistable bundles $E_1$ and $E_2$ are said to be $\mathrm{S}$-equivalent and written as $E_1 \sim_{\mathrm{S}} E_2$, if $\mathrm{gr}(E_1) \cong \mathrm{gr}(E_2)$. We summarize the following observations from \cite{Das} which will be used frequently in our paper.

\begin{proposition}[{\cite[Proposition~4.8, Proposition~4.9, and Theorem 1.4]{Das}}]\label{prop_easy}
Let $(X,P)$ be an orbifold curve. For any $E \in \mathrm{Vect}(X,P)$, the following hold.
\begin{enumerate}[leftmargin=*]
\item If $F$ is another bundle on $(X,P)$, then $\mu_P(E \otimes F) = \mu_P(E) + \mu_P(F)$.\label{j:1}
\item $E$ is $P$-(semi)stable if and only if for all surjective morphism $E \twoheadrightarrow E''$ of bundles, we have
\begin{center}
$\mu_P(E) \, (\leq) \, \mu_P(E'').$
\end{center}\label{j:2}
\item $E$ is $P$-(semi)stable if and only if the dual bundle $E^\vee$ is $P$-(semi)stable.\label{j:3}
\item If $E$ is $P$-(semi)stable of slope $\mu_P(E) < 0$, then
\begin{center}
$\mathrm{H}^0((X,P), E) = 0.$
\end{center}
More generally, if $\mu_{P, \mathrm{max}}(E) = \mu_P(\mathrm{HN}_1(E)) < 0$, the group $\mathrm{H}^0((X,P), E)$ vanishes.\label{j:4}
\item If $L$ is a line bundle on $(X,P)$, then $E$ is $P$-(semi)stable if and only if $E \otimes L \in \mathrm{Vect}(X,P)$ is $P$-(semi)stable.\label{j:5}
\item For any finite group $\Gamma$, let $g \colon Z \longrightarrow (X,P)$ be a $\Gamma$-Galois cover where $Z$ is a smooth projective connected $k$-curve. Then $g$ factors as a composition
\[ g \colon Z \overset{u}\longrightarrow [Z/\Gamma] = (X, P') \overset{h}\longrightarrow (X,P) \]
where $u$ is the natural atlas, and $h$ is the induced cover of orbifold curves. We have the following properties.
\begin{enumerate}
\item $E$ is $P$-semistable (respectively, $P$-polystable) if and only if the bundle $h^*E$ on $(X,P')$ is $P'$-semistable (respectively, $P'$-polystable) if and only if the bundle $u^*h^*E$ on $Z$ is semistable (respectively, polystable). The $\Gamma$-semistability (respectively, $\Gamma$-polystability) of any $\Gamma$-equivariant bundle is the same as the usual semistability (respectively, polystability).\label{6.a}
\item $E$ is $P$-stable if and only if $h^*E$ is $P'$-stable if and only if the $\Gamma$-equivariant bundle $u^*h^*E$ on $Z$ is $\Gamma$-stable.\label{6.b}
\end{enumerate}\label{j:0}
\item Given a cover $f \colon (Y,Q) \longrightarrow (X,P)$ of orbifold curves, we have the following slope stability relations between $E$ and the pullback bundle $f^*E \in \mathrm{Vect}(Y,Q)$.
\begin{enumerate}[leftmargin=*]
    \item $f^*E$ is $Q$-semistable if and only if $E$ is $P$-semistable;
    \item if $f^*E$ is $Q$-stable, then $E$ is $P$-stable;
    \item for any $P$-stable bundle $E'$, the pullback bundle $f^*E'$ is $Q$-stable if and only if $\mathrm{HN}_1(f_* \mathcal{O}_{(Y,Q)}) = \mathcal{O}_{(X,P)}$;
    \item if the induced cover $f_0 \colon Y \longrightarrow X$ is Galois and $E$ is $P$-stable or $P$-polystable, then $f^*E$ is $Q$-polystable. Additionally, if $f$ is \'{e}tale, then $f^*E$ is $Q$-polystable if and only if $E$ is $P$-polystable.\label{j:6}
\end{enumerate} 
\end{enumerate}
\end{proposition}

Our moduli problem requires a comparison of the $\mathrm{S}$-equivalence classes of bundles relative to our setup in Notation~\ref{not_equivariant_set_up}. Before proceeding, let us see that there is a canonical filtration of a $P$-semistable bundle $E$ that calculates the associated graded bundle $\mathrm{gr}_E$.

\begin{remark}[{The canonical filtration and the associated graded bundle}]\label{rmk_canonical_filtration}
Let $E \in \mathrm{Vect}(X,P)$ be $P$-semistable. \cite[Proposition~4.8~(8)]{Das} gives the existence of the unique maximal $P$-polystable sub-bundle $\mathcal{S}(E)$ of $E$, called the \textit{socle} of $E$, such that $\mu_P(\mathcal{S}(E)) = \mu_p(E)$. By an induction on the rank $\mathrm{rk}(E)$, we obtain a canonical filtration
\begin{equation}\label{eq_polystable_filtration}
0 = \mathcal{S}_0(E) \subset \mathcal{S}_1(E) \subset \cdots \subset \mathcal{S}_l(E) = E    
\end{equation}
such that each subsequent quotient $\mathcal{S}_{i+1}(E)/\mathcal{S}_i(E)$ is the socle of the bundle $E/\mathcal{S}_i(E)$, \, $0 \leq i \leq l-1$. In particular, each $\mathcal{S}_{i+1}(E)/\mathcal{S}_i(E)$ is $P$-polystable of $P$-slope $\mu_P(E)$. Since a maximal refinement of the above filtration produces a Jordan-H\"{o}lder filtration of $E$, we obtain
\[
\mathrm{gr}_E \cong \oplus_{\substack{0 \leq i \leq l-1}} \, \mathcal{S}_{i+1}(E)/\mathcal{S}_i(E).\hfill\qed
\]
\end{remark}

The following result gives a comparison of the associated graded bundles in our setup. We will prove a stronger version of it later (in Proposition~\ref{prop_S_equivariance}).

\begin{proposition}\label{prop_S-equivariance_wrt_set_up}
Suppose that Notation~\ref{not_equivariant_set_up} hold. We have the composite cover
\[
g \colon Z \overset{u}\longrightarrow [Z/\Gamma] = (X,\tilde{P}) \overset{\tau}\longrightarrow (X,P).
\]
Let $E, \, E_1$ and $E_2$ be $P$-semistable bundles on $(X,P)$. Then the following hold.
\begin{enumerate}[leftmargin=*]
\item If $ 0 = E_0 \subset E_1 \subset \cdots \subset E_l = E $ is a Jordan-H\"{o}lder filtration for $E$, then
\[
0 = \tau^*E_0 \subset \tau^*E_1 \subset \cdots \subset \tau^*E_l = \tau^*E
\]
is a Jordan-H\"{o}lder filtration for $\tau^*E$. Conversely, every Jordan-H\"{o}lder filtration for $\tau^*E$ is of the form
\[
0 = \tau^*E_0 \subset \tau^*E_1 \subset \cdots \subset \tau^*E_l = \tau^*E
\]
for a Jordan-H\"{o}lder filtration $ 0 = E_0 \subset E_1 \subset \cdots \subset E_l = E $ of $E$. In particular, $\mathrm{gr}_{\tau^*E} \cong \tau^* \mathrm{gr}_E$.

The bundles $E_1$ and $E_2$ are $\mathrm{S}$-equivalent if and only if $\tau^*E_1 \sim_{\mathrm{S}} \tau^*E_2$.\label{jh:1}
\item The associated graded bundle $\mathrm{gr}_{g^*E}$ is isomorphic to $g^*\mathrm{gr}_E$ as $\Gamma$-equivariant bundles. In particular, if the bundles $E_1$ and $E_2$ are $\mathrm{S}$-equivalent, then the $\Gamma$-equivariant bundles $g^*E_1$ and $g^*E_2$ are $\mathrm{S}$-equivalent as bundles on $Z$. However, the converse need not hold.\label{jh:2}
\end{enumerate}
\end{proposition}

\begin{proof}
The bundle $\tau^*E$ is $\tilde{P}$-semistable by \Cref{prop_easy}~\eqref{j:0}. Consider a Jordan-H\"{o}lder filtration
\[
0 = E_0 \subset E_1 \subset \cdots \subset E_l = E
\]
of $E$. Then each quotient $E_{i+1}/E_i$ is $P$-stable of $P$-slope $\mu_P(E)$. As $\tau^*E$ is an exact functor, we obtain a filtration
\begin{equation}\label{eq_JH_1}
0 = \tau^*E_0 \subset \tau^*E_1 \subset \cdots \subset \tau^*E_l = \tau^*E   
\end{equation}
of $\tau^*E$ such that $\tau^*(E_{i+1}/E_i) \cong \tau^*E_{i+1}/\tau^*E_i$ is $\tilde{P}$-stable of $\tilde{P}$-slope
\[
\mu_{\tilde{P}}(\tau^*(E_{i+1}/E_i)) = \mu_P(E_{i+1}/E_i) = \mu_P(E) = \mu_{\tilde{P}}(\tau^*E).
\]
Thus, \eqref{eq_JH_1} is a Jordan-H\"{o}lder filtration of $\tau^*E$. Conversely, suppose that
\[
0 = F_0 \subset F_1 \subset \cdots \subset F_l = \tau^*E
\]
is a Jordan H\"{o}lder filtration of $\tau^*E$. Using \Cref{prop_descent_unique_sub}~\eqref{des:2}, we obtain a unique filtration
\[
0 = E_0 \subset E_1 \subset \cdots \subset E_l = E
\]
of $E$ such that $F_i \cong \tau^*E_i$. Since each $F_{i+1}/F_i \cong \tau^*(E_{i+1}/E_i)$ is $\tilde{P}$-stable of $\tilde{P}$-slope $\mu_{\tilde{P}}(\tau^*E) = \mu_P(E)$, \Cref{prop_easy}~\eqref{6.b} says that each $E_{i+1}/E_i$ is $P$-stable, and it is of $P$-slope $\mu_P(E)$. This proves the converse statement.

Using the above notation, we immediately have:
\[
\mathrm{gr}_{\tau^*E} = \oplus_{\substack{0 \leq i \leq l-1}}\,  F_{i+1}/F_i \cong \oplus_{\substack{0 \leq i \leq l-1}}\, \tau^*(E_{i+1}/E_i) \cong \tau^* \left( \oplus_{\substack{0 \leq i \leq l-1}}\, E_{i+1}/E_i \right) \cong \tau^* \mathrm{gr}_E.
\]
Finally, $E_1 \sim_{\mathrm{S}} E_2$ if and only if $\mathrm{gr}_{E_1} \cong \mathrm{gr}_{E_2}$ if and only if $\tau^*\mathrm{gr}_{E_1} \cong \tau^*\mathrm{gr}_{E_2}$ as $\tau^*$ and $\tau_*$ are both exact functors. The last statement if equivalent to: $\mathrm{gr}_{\tau^*E_1} \cong \mathrm{gr}_{\tau^*E_1}$, that is: $\tau^*E_1 \sim_{\mathrm{S}} \tau^*E_2$, completing the proof of~\eqref{jh:1}.

Using the first statement, it is enough to prove~\eqref{jh:2} when $g = u \colon Z \longrightarrow [Z/\Gamma] = (X,P)$. Consider the canonical filtration~\eqref{eq_polystable_filtration}
\[
0 = \mathcal{S}_0(E) \subset \mathcal{S}_1(E) \subset \cdots \subset \mathcal{S}_l(E) = E 
\]
such that each subsequent quotient bundle $\mathcal{S}_{i+1}(E)/\mathcal{S}_i(E)$ is the socle $\mathcal{S}\left( E/\mathcal{S}_i(E) \right)$, which is $P$-polystable of $P$-slope $\mu_P(E)$. This gives:
\[\text{\rm gr}_E \cong \oplus_{\substack{0 \leq i \leq l-1}} \, \mathcal{S}_{i+1}(E)/\mathcal{S}_i(E).\]
Now consider the $\Gamma$-equivariant semistable bundle $u^*E$ on $Z$. Once again, we have the canonical filtration
\[
0 = \mathcal{S}_0(u^*E) \subset \mathcal{S}_1(u^*E) \subset \cdots \subset \mathcal{S}_l(u^*E) = u^*E 
\]
such that each quotient bundle $\mathcal{S}_{i+1}(u^*E)/\mathcal{S}_i(u^*E)$ is the socle of the bundle $u^*E/\mathcal{S}_i(u^*E)$, which is polystable of slope $\mu(u^*E) = |\Gamma|\mu_P(E)$. By our construction and the uniqueness of the socle, we conclude that each sub-bundle $\mathcal{S}_i(u^*E)$ of $u^*E$ is a $\Gamma$-equivariant sub-bundle. Since the cover $u \colon Z \longrightarrow (X,P)$ is a representable Galois cover, \cite[Lemma~3.4(2)]{Das} says that there is a unique sub-bundle $E_i \subset E$ such that $u^*E_i \cong \mathcal{S}_i(u^*E)$ as $\Gamma$-equivariant bundles. By Proposition~\ref{prop_easy}~\eqref{j:6} and the maximality of the socle, we conclude that $E_1 \cong \mathcal{S}(E) = \mathcal{S}_1(E)$. An induction on $i$ shows that for all $0 \leq i \leq l-1$, we have
\[
E_i \cong \mathcal{S}_i(E) \hspace{.2cm} \text{and} \hspace{.2cm} \mathcal{S}_i(u^*E) \cong u^*\mathcal{S}_i(E).
\]
We conclude that
\begin{eqnarray*}
\mathrm{gr}_{u^*E} \cong \oplus_{\substack{0 \leq i \leq l-1}} \, \mathcal{S}_{i+1}(u^*E)/\mathcal{S}_i(u^*E) \cong \oplus_{\substack{0 \leq i \leq l-1}} \, u^*\mathcal{S}_{i+1}(E)/u^*\mathcal{S}_i(E) \\
\cong \oplus_{\substack{0 \leq i \leq l-1}} \, u^* \left( \mathcal{S}_{i+1}(E)/\mathcal{S}_i(E) \right) \cong u^* \mathrm{gr}_E.
\end{eqnarray*}

Now, an isomorphism $\mathrm{gr}_{E_1} \cong \mathrm{gr}_{E_2}$ gives rise to a $\Gamma$-equivariant isomorphism $\mathrm{gr}_{E_1} \cong u^*\mathrm{gr}_{E_2}$. In particular, $E_1 \sim_{\mathrm{S}} E_2$ implies that $u^*E_1$ and $u^*E_2$ are $\mathrm{S}$-equivalent. To see that the converse does not necessarily hold, we note that for any non-trivial character $\chi$ of $k[\Gamma]$, the line bundle $\mathcal{O}_Z \otimes_k \chi$ is isomorphic to $\mathcal{O}_Z$, but such an isomorphism cannot be $\Gamma$-equivariant. Consequently, the line bundle $M$ on $(X,P)$ determined by the $\Gamma$-equivariant isomorphism $u^*M \cong \mathcal{O}_Z \otimes_k \chi$ is not isomorphic to $\mathcal{O}_{(X,P)}$.
\end{proof}

\subsection{The Moduli Problem}\label{sec_Moduli_Problem_Orbifold_Case}
Our goal is to study the moduli problem of $P$-semistable bundles on an orbifold curve $(X,P)$ of a given rank and determinant. Fix an integer $n \geq 2$, and a line bundle $\Delta \in \mathrm{Pic}(X,P)$. As in the classical case, the notion of families of bundles, and the usual and the $\mathrm{S}$-equivalence relation on the families are well defined (we continue with the convention analogous to Convention~\ref{conv_1} and Convention~\ref{conv_2}). In particular, for any finite type $k$-scheme $S$, and a bundle $E$ on $S \times (X,P)$, we see that the bundles $E_s \cong q_{S,*} \left( E \otimes p_S^* k(s) \right)$ on $(X,P)$ are isomorphic for different choices of the closed points $s \in S$; here $k(s)$ denote the skyscraper sheaf on $S$ that is supported on the closed point $s$, and $p_S \colon S \times (X,P) \longrightarrow S, \, q_S \colon S \times (X,P) \longrightarrow (X,P)$ are the usual projections. Moreover, two families of bundles $E$ and $E'$ on $S \times (X,P)$ are equivalent under the usual equivalence if $E \cong E' \otimes p_S^*N$ for some line bundle $N$ on $S$.

A priori, we have no reason to believe that there is a $P$-semistable bundle on $(X,P)$ of a given rank $n \geq 2$ and determinant $\Delta \in \mathrm{Pic}(X,P)$. Nevertheless, we show that being $P$-semistable is an open condition.

\begin{lemma}\label{lem_openness_P-semistability}
Let $S$ be a $k$-scheme of finite type. Let $E$ be a family of bundles on $(X,P)$ of rank $n$ and determinant $\Delta$, parameterized by $S$. Then the set of points $s \in S$ such that $E_s$ is $P$-semistable forms an open subscheme of $S$.
\end{lemma}

\begin{proof}
Fix a $\Gamma$-Galois cover $g \colon Z \longrightarrow (X,P)$ as in Lemma~\ref{lem_choice_tame_cover}. Consider the family $\left( \mathrm{id}_S \times g \right)^*E$ on $S \times Z$ of of $\Gamma$-equivariant bundles of rank $n$ and determinant $g^*\Delta$. By a classical result (see \cite[Proposition~2.11]{StackModuli}), the set
\[
U = \{ s \in S \, | \,  \left( \left( \mathrm{id}_S \times g \right)^*E \right)_s \text{ is semistable}\}
\]
is an open subscheme of $S$. We have the following commutative diagram with obvious maps.
\begin{equation*}
\begin{tikzcd}
S \arrow[dr, phantom, "\#"] & S \times Z \arrow[dr, phantom, "\square"] \arrow[l, swap, "\bar{p}_S"] \arrow[r, "\bar{q}_S"] \arrow[d, "\mathrm{id}_S \times g"] & Z \arrow[d, "g"] \\
S \arrow[u, equal, swap, "\mathrm{id}_S"] & S \times (X,P) \arrow[l, "p_S"] \arrow[r, swap, "q_S"] & (X,P)
\end{tikzcd}
\end{equation*}
For any closed point $s \in S$, the natural isomorphism $\bar{p}_S^* \cong \left( \mathrm{id}_S \times g \right)^* \circ p_S^*$ and the projection formula (cf. \cite[Proposition~A.3.]{Das}) produces the isomorphism:
\[
\left( \left( \mathrm{id}_S \times g \right)^*E \right)_s \cong g^* E_s.
\]
Since the $\Gamma$-equivariant bundle $g^*E_s$ is semistable if and only if $E_s$ is $P$-semistable, the set $U$ becomes
\[
U = \{ s \in S \, | \,  E_s \text{ is } P\text{-semistable}\}
\]
which is an open subscheme of $S$, proving the result.
\end{proof}

We want to construct a Coarse moduli space for the functor $\mathcal{M}_{(X,P)}^{\mathrm{ss, S}}(n, \Delta)$ that to a $k$-scheme $S$ of finite type associates the following set
\begin{equation}\label{eq_functor}
\mathcal{M}_{(X,P)}^{\mathrm{ss, S}}(n, \Delta)(S) = \left\{ 
  \begin{aligned}
  &\text{Families of bundles } E \text{ on } (X,P), \text{ parametrized by } S, \text{ such that for} \\ 
  &\text{ each closed point } s \in S, \, E_s \text{ is a } P \text{- semistable bundle of rank } n \\
  &\text{and } \mathrm{det}(E) \cong q_S^*\Delta \otimes p_S^*N, \text{ modulo the usual and the } \mathrm{S}\text{-equivalence}
  \end{aligned}
\right\}.
\end{equation}

We claim that the functor $\mathcal{M}_{(X,P)}^{\mathrm{ss, S}}(n, \Delta)$ is an algebraic stack -- the rest of the section is devoted in proving this claim.

Fix a $\Gamma$-Galois cover $g \colon Z \longrightarrow (X,P)$ as in Lemma~\ref{lem_choice_tame_cover}. The functor $\mathrm{Coh}_{(X,P)}$ parameterizing flat families of coherent sheaves of modules on $(X,P)$ is an algebraic stack, locally of finite type by \cite[Theorem~2.1.1]{Lie}. As in~\eqref{eq_stack_F}, we have the algebraic stack $\mathcal{M}_{Z}^{\mathrm{ss}}(n, g^*\Delta)$ that parametrizes the families of the usual equivalence classes of bundles on $Z$ of rank $n$ and determinant $g^*\Delta$. Consider the $2$-fiber product stack
\[
\mathcal{M}_{(X,P)}^{\mathrm{ss}}(n, \Delta) \coloneqq \mathcal{M}_{Z}^{\mathrm{ss}}(n, g^*\Delta) \times_{\mathrm{Coh}_Z, \, g^* \Delta} \mathrm{Coh}_{(X,P)}.
\]
Since $\mathrm{Coh}_Z$ and $\mathrm{Coh}_{(X,P)}$ are algebraic stacks, $\mathcal{M}_{(X,P)}^{\mathrm{ss}}(n, \Delta)$ is also algebraic stack by \cite[\href{https://stacks.math.columbia.edu/tag/04TF}{Lemma 04TF}]{SP}, and it parametrizes $P$-semistable bundles of rank $n$ and determinant $\Delta$, up to the usual equivalence relation on families. The arguments of \cite[Lemma~3.11]{StackModuli} hold in our setup, producing the following result.

\begin{lemma}\label{lem_closed_points}
Let $[E]$ be a $k$-point of the stack $\mathcal{M}_{(X,P)}^{\mathrm{ss}}(n, \Delta)$. Then the following hold.
\begin{enumerate}
    \item The point $[\mathrm{gr}_E]$ is contained in the closure of the point $[E]$ where $\mathrm{gr}_E$ is the associated graded bundle.\label{c:1}
    \item The point $[E]$ is a closed point if and only if $E$ is $P$-polystable.\label{c:2}
\end{enumerate}
\end{lemma}

The idea behind the proof of Lemma~\ref{lem_closed_points}~\eqref{c:1} is that any $P$-semistable bundle $E$ that is not $P$-stable, the bundle $E$ sits in an exact sequence
\[
0 \longrightarrow E' \longrightarrow E \longrightarrow E'' \longrightarrow 0
\]
of $P$-semistable bundles of the same $P$-slope. Consider the universal family $\mathcal{E}$ over the affine line spanned by the above extension in $\mathrm{Ext}^1_{(X,P)}(E'',E')$. Then $\mathcal{E}$ is a family of $P$-semistable bundles on $(X,P)$, parameterized by $\mathbb{A}^1_t$, such that $\mathcal{E}_t \cong E$ if $t \neq 0$, and $\mathcal{E}_0 \cong E' \oplus E''$. This defines a morphism
\[
\mathbb{A}^1_t \overset{[\mathcal{E}]}\longrightarrow \mathcal{M}_{(X,P)}^{\mathrm{ss}}(n, \Delta)
\]
such that $t \mapsto [E]$ if $t \neq 0$, and $[0] \mapsto [E' \oplus E'']$. Iterating the above for the $P$-semistable bundles $E'$ and $E''$, we see that $[\mathrm{gr}_E]$ lies in the closure of the point $[E]$. The above argument also gives an action (in the sense of \cite[Definition~2.1]{Romagny}) of the smooth group scheme $\mathbb{G}_a$ on the algebraic stack $\mathcal{M}_{(X,P)}^{\mathrm{ss}}(n, \Delta)$ such that the orbit of a $k$-point $[E]$ of $\mathcal{M}_{(X,P)}^{\mathrm{ss}}(n, \Delta)$ is the set of all $k$-points $[F]$ lying in the closure of $[E]$. In other words, two $k$-points $[E]$ and $[E']$ belong to the same $\mathbb{G}_a$-orbit if and only if $E \sim_{\mathrm{S}} E'$.

By our definition and construction, the quotient stack $\mathcal{M}_{(X,P)}^{\mathrm{ss}}(n, \Delta)/\mathbb{G}_a$ is $2$-isomorphic to the stack $\mathcal{M}_{(X,P)}^{\mathrm{ss, S}}(n, \Delta)$. By \cite[Theorem~4.1]{Romagny}, $\mathcal{M}_{(X,P)}^{\mathrm{ss, S}}(n, \Delta)$ is an algebraic stack.

\section{Construction of the Moduli Space}\label{sec_construction}
Throughout this section, we fix the following notation.

\begin{notation}\label{not_fixed}
Suppose that Notation~\ref{not_equivariant_set_up} hold. In particular, we have an orbifold curve $(X,P)$ together with its Coarse moduli map $\iota \colon (X,P) \longrightarrow X$, a tamely ramified cover $\tau \colon (X,\tilde{P}) \longrightarrow (X,P)$, and a $\Gamma$-Galois cover $g_0 \colon Z \longrightarrow X$ of smooth projective connected $k$-curves that is a composition
\[
g_0 \colon Z \overset{u}\longrightarrow [Z/\Gamma] = (X,\tilde{P}) \overset{\tau}\longrightarrow (X,P) \overset{\iota}\longrightarrow X
\]
where $u$ is the canonical atlas. Then the representable cover $g = \tau \circ u \colon Z \longrightarrow (X,P)$ is $\Gamma$-Galois and is tamely ramified. Let $g_Z$ denote the genus of the curve $Z$.

Additionally, let $n \geq 2$ be an integer. Fix a line bundle $\Delta \in \mathrm{Pic}(X,P)$. Finally, fix a line bundle $M \in \mathrm{Pic}(X)$ satisfying
\begin{equation}\label{eq_ass_degree_M}
|\Gamma| \left( \mathrm{deg}(M) + \mathrm{deg}_P(\Delta)/n \right) \geq 2g_Z +1.
\end{equation}
Set $L \coloneqq \iota^*M \in \mathrm{Pic}(X,P)$.
\hfill\qedsymbol
\end{notation}

In this section, we construct the Coarse moduli space of the algebraic stack $\mathcal{M}_{(X,P)}^{\mathrm{ss}}(n, \Delta)$; see~\eqref{eq_functor} for the definition. Our construction uses results from the construction of the moduli space $\mathrm{M}_Z^{\mathrm{ss}}(n, g^*\Delta)$ from \Cref{sec_Falting_Construction}. The construction is a lengthy process which we break down into sections for the convenience, and every subsequent section retains the notion and construction of the previous one.

\subsection{Bounded Families}\label{sec_bdd}
Our first objective is to show that any family of all $P$-semistable bundles on $(X,P)$ of rank $n$ and determinant $\Delta$ is bounded, i.e., there is a $k$-scheme $B$ of finite type together with a bundle $\mathcal{V}$ on $B \times (X,P)$ such that every $P$-semistable bundle of rank $n$ and determinant $\Delta$ is represented by $\mathcal{V}_b$ for some closed point $b \in B$. This is equivalent to the algebraic stack $\mathcal{M}_{(X,P)}^{\mathrm{ss}}(n, \Delta)$ being quasi-compact.

We recall the following well-known result.

\begin{proposition}\label{prop_classical}
Suppose that Notation~\ref{not_fixed} hold. Let $E \in \mathrm{Vect}(X,P)$ be $P$-semistable of rank $n$ and $\mathrm{det}(E) \cong \Delta$. Then the following hold.
\begin{enumerate}
\item The evaluation map
\[
\mathrm{ev} \colon \mathrm{H}^0(Z, g^*(E \otimes L)) \otimes_k \mathcal{O}_Z \longrightarrow g^*(E \otimes L)
\]
is a $\Gamma$-equivariant surjective map.\label{P1:1}
\item For $1 \leq i \leq n-1$, there are $i$-dimensional $k$-subspaces $U_i \subset \mathrm{H}^0(Z, g^*(E \otimes L))$ satisfying $U_1 \subset \cdots \subset U_{n-1}$ such that each restriction of the evaluation map
\[
\mathrm{ev}_{i} \colon U_{i} \otimes_k \mathcal{O}_Z \longrightarrow g^*(E \otimes L)
\]
is an injective map, and the cokernel $\mathrm{coker}\left( \mathrm{ev}_{i} \right)$ is a bundle on $Z$.\label{P1:2}
\end{enumerate}
\end{proposition}

\begin{proof}
Using Proposition~\ref{prop_easy} and the defining properties of rank and determinant, we see that the $\Gamma$-equivariant bundle $g^*(E \otimes L)$ on $Z$ is semistable of rank $n$ and $\mathrm{det}(g^*(E \otimes L)) \cong g^*\Delta \otimes g^*L^{\otimes n}$. So $g^*(E \otimes L)$ is a bundle on $Z$ satisfying $\mu(g^*(E \otimes L)) > 2g_Z -1$ by~\eqref{eq_ass_degree_M}, and hence it is globally generated. In other words, the evaluation map $\mathrm{ev}$ is a surjection. Since $g^*(E \otimes L)$ is also $\Gamma$-equivariant, so is the bundle $\mathrm{H}^0(Z, g^*(E \otimes L)) \otimes_k \mathcal{O}_Z$. Since the map $\mathrm{ev}$ is the evaluation of the transformation $\eta_Z^* \eta_{Z,*} \Rightarrow \mathrm{id}$ at the $\Gamma$-equivariant bundle $g^*(E \otimes L)$, by naturality, $\mathrm{ev}$ is a $\Gamma$-equivariant map -- this proves~\eqref{P1:1}.

To prove~\eqref{P1:2}, it is enough to exhibit $U_{n-1}$ satisfying the stated property as any subspace $U_i$ of $U_{n-1}$ of dimension $i$ will also have the stated property. This is \cite[Theorem~2]{Atiyah}.
\end{proof}

Next, we use the above result to establish that any family of $P$-semistable bundles of rank $n$ and determinant $\Delta$ is bounded.

\begin{proposition}\label{prop_main_bounded}
Suppose that Notation~\ref{not_fixed} hold. Let $E \in \mathrm{Vect}(X,P)$ be $P$-semistable of rank $n$ and determinant $\mathrm{det}(E) \cong \Delta$. Then the following hold.
\begin{enumerate}[leftmargin=*]
\item For $1 \leq i \leq n-1$, there are sub-bundles $F_i \subset E \otimes L$ of rank $i$ satisfying the following properties.\label{P2:1}
\begin{enumerate}
\item For $1 \leq i < j \leq n-1$, $F_i$ is a sub-bundle of $F_j$.\label{P2:1:1}
\item For each $1 \leq i \leq n-1$, we have a $\Gamma$-equivariant isomorphism $g^* F_i \cong V_i \otimes_k \mathcal{O}_Z$ for some $k[\Gamma]$-sub-module $V_i \subset \mathrm{H}^0(Z, u^*u_*(U_i \otimes_k \mathcal{O}_Z))$ where $U_i$ is an in Proposition~\ref{prop_classical}.\label{P2:1:2}
\item For each $1 \leq i \leq n-1$, \, $F_i \cong \oplus_{1 \leq j \leq i} L_j$ where $L_j \in \text{\rm Pic}(X,P)$ with $g^*L_j \cong \chi_j \otimes_k \mathcal{O}_Z$ as $\Gamma$-equivariant line bundles on $Z$ for some character $\chi_j$ of $\Gamma$ over $k$.\label{P2:1:3}
\end{enumerate}
\item There is a short exact sequence
\begin{eqnarray*}
0 \longrightarrow F_{n-1} \otimes L^{-1} \longrightarrow E \longrightarrow \Delta \otimes L^{\otimes (n-1)} \otimes \mathrm{det}(F_{n-1})^{-1} \longrightarrow 0
\end{eqnarray*}
of bundles on $(X,P)$.\label{P2:2}
\end{enumerate}
\end{proposition}

\begin{proof}
First, note that if there are bundles $\tilde{F}_i$ on $(X, \tilde{P})$ satisfying the corresponding required statements over $(X, \tilde{P})$ for the bundle $\tau^*E$, then the statements hold (with $F_i$ such that $\tau^*F_i \cong \tilde{F}_i$) over $(X,P)$ for the bundle $E$ by applying Proposition~\ref{prop_descent_unique_sub}~\eqref{des:2} and the projection formula. So, without loss of generality, we may assume that $(X,P) = (X,\tilde{P})$, and $g = u \colon Z \longrightarrow (X,P)$ is the canonical atlas.

By Proposition~\ref{prop_classical}, the $\Gamma$-equivariant evaluation map
\[
\mathrm{ev} \colon \mathrm{H}^0(Z, u^*(E \otimes L)) \otimes_k \mathcal{O}_Z \longrightarrow u^*(E \otimes L)
\]
is a surjection, and for each $1 \leq i \leq n-1$, there exists an $i$-dimensional subspace $U_i$ of $\mathrm{H}^0(Z, u^*(E \otimes L))$ with $U_1 \subset \cdots \subset U_{n-1}$ such that the restrictions
\[
\mathrm{ev}_{i} \colon U_{i} \otimes_k \mathcal{O}_Z \longrightarrow u^*(E \otimes L)
\]
are injective maps and each $\mathrm{coker}(\mathrm{ev}_{i})$ is a bundle on $Z$. The canonical inclusion $u^\# \colon E \otimes L \hookrightarrow u_* u^* \left( E \otimes L \right)$ is a $u$-locally split monomorphism by \cite[Proposition~3.5.4(i)]{split}, i.e., $u^*(E \otimes L)$ is a direct summand of $u^*u_*u^*(E \otimes L)$.

For each $1 \leq i \leq n-1$, consider the Cartesian diagram
\begin{equation}\label{D_i}
\begin{tikzcd}
F_i \arrow[dr, phantom, "\square"] \arrow[r, "p_{2,i}"] \arrow[d, "p_{1,i}"] & E \otimes L \arrow[d, hookrightarrow, "u^\#"] \\
u_*\left( U_i \otimes_k \mathcal{O}_Z \right) \arrow[r, hookrightarrow, "u_*\mathrm{ev}_i"] & u_*u^*(E \otimes L) \cong E \otimes L \otimes u_* \mathcal{O}_Z
\end{tikzcd}
\end{equation}
in $\mathrm{Coh}_{\mathcal{O}_{(X, P)}}$ where
\[ F_i = \mathrm{ker}\left( \left( E \otimes L \right) \oplus u_*\left( U_i \otimes_k \mathcal{O}_Z \right) \xlongrightarrow{u^\# - u_* \mathrm{ev}_i} u_*u^*(E \otimes L) \right). \]
Since the above diagram is Cartesian, we see that $F_i$ is a coherent sub-sheaf of $u_*(U_i \otimes_k \mathcal{O}_Z)$ and of $E \otimes L$. As $u$ is flat, the diagram
\begin{equation}\label{gDi}
\begin{tikzcd}
u^*F_i \arrow[r, hookrightarrow, "u^*p_{2,i}"] \arrow[d, hookrightarrow, "u^*p_{1,i}"] & u^*(E \otimes L) \arrow[d, hookrightarrow, "u^*u^\#"] \\
u^*u_*\left( U_i \otimes_k \mathcal{O}_Z \right) \arrow[r, hookrightarrow, "u^*u_*\mathrm{ev}_i"] & u^*u_*u^*(E \otimes L)
\end{tikzcd}
\end{equation}
obtained by applying $u^*$ is also a Cartesian diagram of $\Gamma$-equivariant bundles. In particular, $F_i$ is sub-bundle of $u_*(U_i \otimes_k \mathcal{O}_Z)$ and of $E \otimes L$.

We claim that $F_i$ has rank $i$ for $1 \leq i \leq n-1$. To see this, let $\zeta$ denote the generic point of $X$. Then the map $u_* \mathrm{ev}_{i}$ is generically given by
\[
\left( u_* \mathrm{ev}_{i} \right)_\zeta \colon U_i \otimes_k k(Z) \longrightarrow \left( E \otimes L \right)_\zeta \otimes_{k(X)} k(Z).
\]
Since this map is the base change of
\[
U_i \otimes_k k(X) \longrightarrow \left( E \otimes L \right)_\zeta
\]
by $- \otimes_{k(X)} k(Z)$, and the canonical map $u^\#_\zeta$ is the natural inclusion $(E \otimes L)_\zeta \otimes_{k(X)} k(X) \hookrightarrow (E \otimes L)_\zeta \otimes_{k(X)} k(Z)$, we conclude that the map $p_{2,i}$ is generically given by $U_i \otimes_k k(X) \longrightarrow \left( E \otimes L \right)_\zeta$. This is a $k(X)$-linear monomorphism, and by Proposition~\ref{prop_descent_unique_sub}~\eqref{des:3}, $F_i$ is the unique sub-bundle of $E \otimes L$ such that $(F_i)_\zeta = U_i \otimes_k k(X)$. Thus $\mathrm{rk}(F_i) = i$.

From the universal property of the Cartesian squares, we see that the diagram
\begin{equation}\label{D_ij}
\begin{tikzcd}
F_i \arrow[r] \arrow[d, hookrightarrow, "p_{1,i}"] & F_j \arrow[d, hookrightarrow, "p_{1,j}"] \\
u_*\left( U_i \otimes_k \mathcal{O}_Z \right) \, \, \,  \arrow[r, hookrightarrow, "u_*(\mathrm{inclusion})"] & \, \, u_*\left( U_j \otimes_k \mathcal{O}_Z \right)
\end{tikzcd}
\end{equation}
is also Cartesian for any $1 \leq i < j \leq n-1$. This shows that $F_i$ is a sub-bundle of $F_j$ for $1 \leq i < j \leq n-1$. We conclude that the bundles $F_i$'s satisfies property~\eqref{P2:1:1}.

Now we show that $F_i$'s satisfy property~\eqref{P2:1:2}. In the Cartesian square~\eqref{gDi}, since $u^*u^\#$ is a split inclusion, so is $u^*p_{1,i}$. So, $u^*F_i$ is a direct summand of the trivial bundle $u^*u_*(U_i \otimes_k \mathcal{O}_Z) \cong \oplus_{\gamma \in \Gamma} \, \gamma^* \mathcal{O}_Z^{\oplus \, i}$. Thus, the $\Gamma$-equivariant bundle $u^*F_i$ is a trivial bundle on $Z$, and it is of the form $V_i \otimes_k \mathcal{O}_Z$ for a $k[\Gamma]$-sub-module of $\mathrm{H}^0(Z, u^*u_*(U_i \otimes_k \mathcal{O}_Z)) \cong \oplus_{1 \leq j \leq i} \, \mathrm{H}^0(Z, u^*u_*\mathcal{O}_Z)$.

From the above, $F_1$ is a line bundle on $(X,P)$ such that $u^*F_1 \cong V_1 \otimes_k \mathcal{O}_Z$ as $\Gamma$-equivariant bundles on $Z$, and $\chi_1 \coloneqq V_1$ is a character of $\Gamma$ over $k$. Now for any $1 \leq i < j \leq n-1$, the subspace $U_i$ of $U_j$ defines the trivial bundle $\mathcal{O}_Z^{\oplus i} \cong U_i \otimes_k \mathcal{O}_Z$ as a direct summand of the trivial bundle $U_j \otimes_k \mathcal{O}_Z \cong \mathcal{O}_Z^{\oplus j}$. As $u_*$ is an exact functor, and from the Cartesian square~\ref{D_ij} we conclude that the sub-bundle $F_i$ of $F_j$ is a direct summand. Inductively, we see that for each $1 \leq i \leq n-1$, we have
\[
F_i \cong \oplus_{1 \leq j \leq i} \, L_j
\]
for line bundles $L_j \in \mathrm{Pic}(X,P)$ with each $u^*L_j \cong \chi_j \otimes_k \mathcal{O}_Z$ as $\Gamma$-equivariant bundles on $Z$ for some character $\chi_j$ of $\Gamma$ over $k$, proving~\eqref{P2:1:3}.

Finally, since $F_{n-1}$ is a rank $(n-1)$ sub-bundle of $E \otimes L$, considering the determinant, we obtain a short exact sequence
\[
0 \longrightarrow F_{n-1} \longrightarrow E \otimes L \longrightarrow \Delta \otimes L^{\otimes n} \otimes \text{\rm det}(F_{n-1})^{-1} \longrightarrow 0.
\]
Taking the tensor product with $L^{-1}$, we obtain the short exact sequence
\[
0 \longrightarrow F_{n-1} \otimes L^{-1} \longrightarrow E \longrightarrow \Delta \otimes L^{\otimes (n-1)} \otimes \text{\rm det}(F_{n-1})^{-1} \longrightarrow 0
\]
which is~\eqref{P2:2}.
\end{proof}

Now, we proceed to the construction of the over-parameterizing family for our moduli problem.

\begin{definition}[{The Indexing Set}]\label{def_F}
Define the \emph{indexing set for the moduli problem} to be the set $\mathcal{F}$ of mutually non-isomorphic bundles $F$ on $(X,P)$ of rank $n-1$ such that
\[
F \cong \oplus_{\substack{1 \leq j \leq n-1}} L_j
\]
where $L_j \in \mathrm{Pic}(X,P)$ with $g^*L_j \cong \chi_j \otimes_k \mathcal{O}_Z$ as $\Gamma$-equivariant bundles on $Z$ for some characters $\chi_j$ of $\Gamma$ over $k$.
\end{definition}

For $F \in \mathcal{F}$, consider the $k$-vector space $\mathcal{V}_F$ whose dual is defined by
\[
\mathcal{V}_F^* \coloneqq \mathrm{Ext}^1_{(X,P)}\left( \Delta \otimes L^{\otimes (n-1)} \otimes \mathrm{det}(F)^{-1}, F \otimes L^{-1} \right) \cong \mathrm{H}^1\left( (X,P), F \otimes \mathrm{det}(F) \otimes \Delta^{-1} \otimes L^{-n} \right).
\]
Since $F \otimes \mathrm{det}(F) \otimes \Delta^{-1} \otimes L^{-n}$ is a direct sum of line bundles $L_j \otimes (\otimes_{1 \leq i \leq n-1}\, L_j) \otimes \Delta^{-1} \otimes L^{-n}$, each of $P$-slope
\[
- n (\mu_P(\Delta)/n + \mathrm{deg}(M)) < 0
\]
by our assumption~\eqref{eq_ass_degree_M}, the $P$-polystable bundle $F \otimes \mathrm{det}(F) \otimes \Delta^{-1} \otimes L^{-n}$ is also of negative $P$-slope. So,
\[
\mathrm{H}^0\left( (X,P), F \otimes \mathrm{det}(F) \otimes \Delta^{-1} \otimes L^{-n} \right) = 0
\]
by Proposition~\ref{prop_easy}~\eqref{j:4}. Setting
\begin{equation}\label{eq_P_F}
\mathbb{P}_F \coloneqq \mathbb{P}(\mathcal{V}_F)    
\end{equation}
for the Grassmannian variety of linear hyperplanes in $\mathcal{V}_F$, and writing $p_F \colon \mathbb{P}_F \times (X,P) \longrightarrow \mathbb{P}_F$ and $q_F \colon \mathbb{P}_F \times (X,P) \longrightarrow (X,P)$ for the natural projections, Proposition~\ref{prop_universal_general} states the following.

\begin{proposition}\label{prop_extn_classes_F}
Under the above notation, there is an extension
\begin{equation}\label{eq_extn_F}
0 \rightarrow q_F^* \left( F \otimes L^{-1} \right) \otimes p_F^* \mathcal{O}_{\mathbb{P}_F}(1) \longrightarrow \mathcal{E}_F \longrightarrow q_F^* \left( \Delta \otimes L^{\otimes (n-1)} \otimes \mathrm{det}(F)^{-1} \right) \rightarrow 0
\end{equation}
that is universal in the category of Noetherian $k$-schemes $S$ for the equivalent classes of families of extensions of $q_{S}^*( \Delta \otimes L^{\otimes (n-1)} \otimes \mathrm{det}(F)^{-1} )$ by $q_{S}^*(F \otimes L^{-1}) \otimes p_{S}^* N$ which are nowhere splitting on $S$, for arbitrary line bundle $N$ on $S$, and modulo the canonical action of $\mathrm{H}^0\left( S, \mathcal{O}_S^{\times} \right)$; here $q_{S}$ and $p_{S}$ denote the usual projection morphisms.
\end{proposition}

By Proposition~\ref{prop_main_bounded} and Proposition~\ref{prop_extn_classes_F}, each $P$-semistable bundle $E$ on $(X,P)$ of rank $n$ and of determinant $\Delta$ corresponds to the $k$-point of the space $\mathbb{P}_F$ for some $F \in \mathcal{F}$. For any $F \in \mathcal{F}$, the direct sum $\left( F \otimes L^{-1} \right) \oplus \left( \Delta \otimes L^{\otimes (n-1)} \otimes \mathrm{det}(F)^{-1} \right)$ is not a $P$-semistable bundle. To identify the unstable locus in each $\mathbb{P}_F$, we embed $\mathbb{P}_F$ in the over-parameterizing space of semistable bundles on $Z$ of rank $n$ and determinant $g^*\Delta$. Using Proposition~\ref{prop_classical}, for any semistable bundle $A$ on $Z$ of rank $n$ and determinant $g^*\Delta$, the bundle $A \otimes g^*L$ contains a trivial bundle $\mathcal{O}_Z^{\oplus (n-1)}$ as a sub-bundle, resulting in an exact sequence
\[
0 \longrightarrow \left( g^* L^{-1} \right)^{\oplus (n-1)} \longrightarrow A \longrightarrow g^*(\Delta \otimes L^{\otimes (n-1)}) \longrightarrow 0.
\]
Any bundle $A$ as above on $Z$ corresponds to a $k$-points in the projective space $\mathbb{P}_Z \coloneqq \mathbb{P}(\mathcal{W}_Z)$ where $\mathcal{W}_Z$ is dual of the vector space
\[
\mathcal{W}_Z^* \coloneqq \mathrm{H}^1\left( Z, \oplus_{\substack{1 \leq i \leq n-1}} g^*\Delta^{-1} \otimes g^*L^{-n} \right).
\]
Further, since $\mathrm{H}^0(Z, \oplus_{\substack{1 \leq i \leq n-1}} g^*\Delta \otimes g^*L^{-n}) = 0$ from our choice~\eqref{eq_ass_degree_M}, there is a universal extension class (cf.~\eqref{eq_universal_family_on_X}) $\mathcal{E}_Z$ on $\mathbb{P}_Z \times Z$:
\[
0 \rightarrow q_{\mathbb{P}_Z}^* \left( g^*L^{-1} \right)^{\oplus (n-1)} \otimes p_{\mathbb{P}_Z}^* \mathcal{O}_{\mathbb{P}_Z}(1) \longrightarrow \mathcal{E}_Z \longrightarrow q_{\mathbb{P}_Z}^* \left( g^*\Delta \otimes (g^*L)^{\otimes (n-1)} \right) \rightarrow 0.
\]
The following result identifies the projective spaces $\mathbb{P}_F$ as disjoint linear closed sub-spaces of $\mathbb{P}_Z$, together with an important relation between the universal families $\mathcal{E}_F$ and $\mathcal{E}_Z$.

\begin{proposition}\label{prop_disjoint_images}
Under the above notation, for any $F \in \mathcal{F}$, the vector space $\mathcal{V}_F^*$ is the linear subspace of the $\Gamma$-fixed points of $\mathcal{W}_Z^* \coloneqq \mathrm{H}^1\left( Z, \oplus_{\substack{1 \leq i \leq n-1}} \, g^*\Delta \otimes g^*L^{-n} \right)$. The images of the induced closed embeddings
\[
\iota_F \colon \mathbb{P}_F \hookrightarrow \mathbb{P}_Z
\]
for $F \in \mathcal{F}$ are disjoint. Further, there is a natural isomorphism
\begin{center}
$\left( \mathrm{id}_{\mathbb{P}_F} \times g \right)^* \mathcal{E}_F \cong \left( \iota_F \times \mathrm{id}_Z \right)^* \mathcal{E}_Z.$
\end{center}
\end{proposition}

\begin{proof}
Proposition~\ref{prop_relation_universal} states that $\mathcal{V}_F^*$ is a linear subspace of the $\Gamma$-fixed points of $\mathcal{W}_Z^*$, hence $\mathbb{P}_F$ is realized as a projective subspace of $\mathbb{P}_Z$ via a closed immersion
\[
\iota_F \colon \mathbb{P}_F \hookrightarrow \mathbb{P}_Z;
\]
further, there is a natural isomorphism
\[
\left( \mathrm{id}_{\mathbb{P}_F} \times g \right)^* \mathcal{E}_F \cong \left( \iota_F \times \mathrm{id}_Z \right)^* \mathcal{E}_Z.
\]

We need to show that the images of $\iota_F$ are mutually disjoint. By our construction, it is enough to show the following: if $E$ is a bundle on $(X,P)$ that simultaneously sit in the non-split extensions
\[
0 \longrightarrow F \otimes L^{-1} \longrightarrow E \longrightarrow \Delta \otimes L^{\otimes (n-1)} \otimes \mathrm{det}(F)^{-1} \longrightarrow 0, \, \text{and}
\]
\[
0 \longrightarrow F' \otimes L^{-1} \longrightarrow E \longrightarrow \Delta \otimes L^{\otimes (n-1)} \otimes \mathrm{det}(F')^{-1} \longrightarrow 0\]
for $F, \, F' \in \mathcal{F}$, then $F = F'$. If the above hold for $F \neq F'$, then $E \otimes L$ is generated by $F$ and $F'$, and hence, $F + F' = E \otimes L$. Consequently, we have an exact sequence of bundles on $(X,P)$:
\[
0 \longrightarrow F \cap F' \longrightarrow F \oplus F' \longrightarrow F + F' = E \otimes L \longrightarrow 0.
\]
Then necessarily, $E \otimes L$ is a bundle of $P$-degree $0$, contradicting the fact that $\mu_P(E \otimes L) = \mathrm{deg}_P(\Delta)/n + \mathrm{deg}_P(L) = \frac{2g_Z +1}{|\Gamma|} > 0$.
\end{proof}

\begin{remark}\label{rmk_overparameterizing_P}
In view of the above results, we see that each $P$-semistable bundle $E$ on $(X,P)$ of rank $n$ and of determinant $\Delta$ corresponds to the $k$-point of the space
\[
\mathbb{P}_{(X,P)} \coloneqq \sqcup_{F \in \mathcal{F}} \, \mathbb{P}_F.
\]
We also have a universal extension $\mathcal{E}_{(X,P)}$ on $\mathbb{P}_{(X,P)} \times (X,P)$ which restricts to the universal extension $\mathcal{E}_F$ on $\mathbb{P}_F \times (X,P)$; more precisely, if $l_F \colon \mathbb{P}_F \hookrightarrow \mathbb{P}_{(X,P)}$ is the inclusion map, then $\mathcal{E}_{(X,P)} = \oplus_{F \in \mathcal{F}} \, \left( l_F \times \mathrm{id}_{(X,P)} \right)_* \mathcal{E}_F$. Via the closed immersion $\iota_P \coloneqq \sqcup \, \iota_F$, we realize $\mathbb{P}_{(X,P)}$ as a closed projective subspace of $\mathbb{P}_Z$. It can be easily seen that
\begin{center}
$\left( \mathrm{id}_{\mathbb{P}_{(X,P)}} \times g \right)^* \mathcal{E}_{(X,P)} \cong \left( \iota_P \times \mathrm{id}_Z \right)^* \mathcal{E}_Z.$    
\end{center}
\end{remark}

We have the following important consequence of our construction.

\begin{lemma}\label{lem_implication_of_F-uniqueness}
Let $F \in \mathcal{F}$. Let $E$ and $E'$ be $P$-semistable bundles on $(X,P)$, representing two $k$-points in $\mathbb{P}_F$. Any isomorphism $\alpha \colon g^*E \rightarrow g^*E'$ as bundles on $Z$ is $\Gamma$-equivariant, and there exists a unique isomorphism $\beta \colon E \longrightarrow E'$ of bundles such that $\alpha = g^*\beta$.
\end{lemma}

\begin{proof}
First note that if $R \in \mathrm{Pic}(X,P)$, then any automorphism of $g^*R$ is $\Gamma$-equivariant; this is because any such automorphism is given by a non-zero scalar multiple of the identity map.

Let $F \cong \oplus_{1 \leq j \leq n-1} \, L_j$ where each $L_j \in \mathrm{Pic}(X,P)$ with $g^*L_j \cong \chi_j \otimes \mathcal{O}_Z$ as $\Gamma$-equivariant line bundle for some character $\chi_j$ of $\Gamma$ over $k$. Let $\alpha \colon g^*E \longrightarrow g^*E'$ be an isomorphism. By our construction in \Cref{prop_main_bounded} (see the Cartesian diagram~\eqref{gDi}), the automorphism $\alpha$ restricts to an automorphism $\alpha|_{\text{\rm rest}}$ of $g^*(F \otimes L^{-1})$, and this induces an automorphism $\alpha_q$ of the quotient $\Gamma$-equivariant line bundle $g^*\left( \Delta \otimes L^{\otimes (n-1)} \otimes \mathrm{det}(F)^{-1} \right)$. By the observation in the first paragraph, $\alpha_q$ is $\Gamma$-equivariant. Also by the structure of $F$, the automorphism $\alpha|_{\text{\rm rest}}$ uniquely corresponds to a an element $\sigma$ of the Symmetric group $S_{n-1}$ together with isomorphisms $\alpha_{\sigma} \colon g^*L_j \longrightarrow g^*L_{\sigma(j)}$ for each $1 \leq j \leq n-1$. Once again, the isomorphisms $\alpha_{\sigma}$ are $\Gamma$-equivariant, and hence, the automorphism $\alpha|_{\text{\rm rest}}$ is also $\Gamma$-equivariant. We obtain a commutative diagram with exact rows:
\begin{equation*}
\begin{tikzcd}
0 \arrow[r] & g^*(F \otimes L^{-1}) \arrow[d, "\alpha|_{\text{\rm rest}}"] \arrow[r] & g^*E \arrow[d, "\alpha"] \arrow[r] & g^*\left( \Delta \otimes L^{\otimes (n-1)} \otimes \mathrm{det}(F)^{-1} \right) \arrow[d, "\alpha_q"] \arrow[r] & 0 \\
0 \arrow[r] & g^*(F \otimes L^{-1}) \arrow[r] & g^*E' \arrow[r] & g^*\left( \Delta \otimes L^{\otimes (n-1)} \otimes \mathrm{det}(F)^{-1} \right) \arrow[r] & 0
\end{tikzcd}
\end{equation*}
where the terminal vertical arrows are $\Gamma$-equivariant isomorphisms. This shows that $g^*\alpha$ is $\Gamma$-equivariant.




By the equivalence of categories $u^* \colon \mathrm{Vect}(X,\tilde{P}) \longrightarrow \mathrm{Vect}^{\Gamma}(Z)$, the $\Gamma$-equivariant isomorphism $\alpha$ descents to an isomorphism of the bundles $\tau^*E \xrightarrow{\cong} \tau^*E'$ on $(X,\tilde{P})$. Finally, using the projection formula for the cover $\tau$, we see that the above isomorphism descends to an isomorphism $\beta \colon E \longrightarrow E'$, proving the result.
\end{proof}

Let us see that the family $\mathcal{E}_{(X,P)}$ locally induces any other semistable family.

\begin{theorem}\label{thm_local_families}
Let $S$ be a $k$-scheme of finite type. Let $E$ be a family of $P$-semistable bundles of rank $n$ and determinant $\Delta$ on $(X,P)$, parameterized by $S$. Then, under the above notation, the following hold.
\begin{enumerate}
\item There is a Zariski open covering $S = \cup_i S_i$ of $S$ together with morphisms $\phi_i \colon S_i \longrightarrow \mathbb{P}_Z$ such that \label{zar:1}
\begin{center}
$\left( \mathrm{id}_{S_i} \times g \right)^* \left( E|_{S_i \times (X,P)} \right) \cong \left( \phi_i \times \mathrm{id}_Z \right)^* \mathcal{E}_Z.$
\end{center}
\item For any connected component $T$ of $S$, with induced Zariski open covering $T = \cup_i T_i$ from~\eqref{zar:1}, there is a unique bundle $F \in \mathcal{F}$ such that each restriction morphism $\phi_i \colon T_i \longrightarrow \mathbb{P}_Z$ factors uniquely as a composition
\[
\phi_i \colon T_i \overset{\psi_i}\longrightarrow \mathbb{P}_F \overset{\iota_F}\hookrightarrow \mathbb{P}_Z
\]
for a morphism $\psi_i \colon T_i \longrightarrow \mathbb{P}_F$, and we have an isomorphism
\[
E|_{T_i \times (X,P)} \cong \left( \psi_i \times \mathrm{id}_{(X,P)} \right)^* \mathcal{E}_F.
\]
The bundle $F$ corresponding to $T$ is uniquely determined by any closed point of $T$.\label{zar:2}
\end{enumerate}
The above open coverings, the maps $\phi_i$'s and $\psi_i$'s, and the above isomorphism are independent of any family in an equivalence class of $E$ under the usual equivalence relation on families.
\end{theorem}

\begin{proof}
The first statement~\eqref{zar:1} is classical; see \cite[Proposition~3.4]{Hein}.

Let us prove~\eqref{zar:2}. Let $T$ be as in the statement. For any closed point $t \in T$, the bundle $E_t \otimes L$ contains a uniquely determined $F \in \mathcal{F}$ as its sub-bundle by Proposition~\ref{prop_main_bounded} and Proposition~\ref{prop_disjoint_images}. Since $T$ is also connected, the image $\phi_i(T)$ is contained in $\mathbb{P}_F$ for each $i$. Moreover, this $F$ is uniquely determined by any closed point $t \in T$. We want to show that for each $i$,
\[
E|_{T_i \times (X,P)} \cong \left( \psi_i \times \mathrm{id}_{(X,P)} \right)^* \mathcal{E}_F.
\]
For this, we may replace $T_i$ by $T$. We have a factorization
$\phi \colon T \overset{\psi}\longrightarrow \mathbb{P}_F \overset{\iota_F}\hookrightarrow \mathbb{P}_Z$.

Since $\phi = \iota_F \circ \psi$, using Proposition~\ref{prop_disjoint_images} and the first statement~\eqref{zar:1}, we have the isomorphisms
\begin{multline}\label{multi:isom}
    \left( \mathrm{id}_T \times g \right)^* \left( E|_{T \times (X,P)} \right) \cong \left( \phi \times \mathrm{id}_Z \right)^*\mathcal{E_Z} \cong \left( \psi \times \mathrm{id}_Z \right)^* \left( \iota_F \times \mathrm{id}_Z \right)^* \mathcal{E}_Z\\
    \cong \left( \psi \times \mathrm{id}_Z \right)^*\left( \mathrm{id}_{\mathbb{P}_F} \times g \right)^* \mathcal{E}_F \cong \left( \mathrm{id}_T \times g \right)^* \left( \psi \times \mathrm{id}_{(X,P)} \right)^* \mathcal{E}_F.
\end{multline}
For any closed point $t \in T$, the bundles $E_t$ and $\left( \left( \psi \times \mathrm{id}_{(X,P)} \right)^*\mathcal{E}_F \right)_t$ on $(X,P)$ are $P$-semistable bundles of rank $n$ and determinant $\Delta$, representing two points in $\mathbb{P}_F$. Lemma~\ref{lem_implication_of_F-uniqueness} says that the above isomorphism~\eqref{multi:isom} is $\Gamma$-equivariant where we endow $T \times Z$ with the induced $\Gamma$-action of $Z$.

As $[(T \times Z)/\Gamma] = T \times (X,\tilde{P})$ by our construction and Notation~\ref{not_fixed}, we have the canonical isomorphism
\[
\left( \mathrm{id}_T \times \tau \right)^*\left( E|_{T \times (X,P)} \right) \cong \left( \mathrm{id}_T \times \tau \right)^* \left( \psi \times \mathrm{id}_{(X,P)} \right)^* \mathcal{E}_F
\]
on the stack $T \times (X,\tilde{P})$. Since $\left( \mathrm{id}_T \times \tau \right)$ is an exact functor, and $\left( \mathrm{id}_T \times \tau \right)_* \mathcal{O}_{T \times (X, \tilde{P})} \cong \mathcal{O}_{T \times (X,P)}$ (which follows from our construction and the universal property of the Coarse moduli morphism), we conclude that
\[
E|_{T \times (X,P)} \cong \left( \psi \times \mathrm{id}_{(X,P)} \right)^* \mathcal{E}_F
\]
on $T \times (X,P)$.

Finally, the construction of any $S_i \subset S$ containing a closed point $s \in S$ in \cite[Proposition~3.4]{Hein} is such that $S_i$ is the maximal affine open subset on which the trivial sub-bundle $\mathcal{O}_Z^{\oplus (n-1)}$ of $g^*(E_s \otimes L)$ lifts to a sub-bundle $\mathcal{O}_{S \times Z}^{\oplus (n-1)}$ of $\left( \mathrm{id}_S \times g \right)^* E \otimes q_S^*g^*L$ under the canonical surjection
\[
\beta \colon p_{S,*}\left( \left( \mathrm{id}_S \times g \right)^* E \otimes q_S^* g^*L \right) \longrightarrow \left( \left( \mathrm{id}_S \times g \right)^* E \otimes q_S^* g^*L \right) \otimes k(s) \cong \mathrm{H}^0\left( Z, g^*\left( E_s \otimes L \right) \right).
\]
Then the universal property of the family $\mathcal{E}_Z$ canonically determines $\phi_i$. If $E' \cong E \otimes p_S^*N$ for a line bundle $N$ on $S$, none of the above discussion changes. So, the induced cover of $T$, the maps $\psi_i$'s remain unchanged, and we get the isomorphisms $E'|_{T_i \times (X,P)} \cong \left( \psi_i \times \mathrm{id}_{(X,P)} \right)^* \mathcal{E}_F.$
\end{proof}


We conclude this section with results on the associated graded bundle $\mathrm{gr}_E$ of a $P$-semistable bundle $E$. It is easy to see that the $P$-polystable bundle $\mathrm{gr}_E$ has rank $n$ and $ \mathrm{det}(\mathrm{gr}_E)\cong \mathrm{det}(E) \cong \Delta$. We show that the closed points in $\mathbb{P}_{(X,P)}$ corresponding to the bundles $E$ and $\mathrm{gr}_E$ both lie in the same connected component $\mathbb{P}_F$. Further, we establish that the $\mathrm{S}$-equivalence relation of bundles in a component $\mathbb{P}_F$ is the same as the $\mathrm{S}$-equivalence relation of the corresponding $\Gamma$-equivariant bundles; compare with Proposition~\ref{prop_S-equivariance_wrt_set_up}.

\begin{proposition}\label{prop_S_equivariance}
Let $E \in \mathrm{Vect}(X,P)$ has rank $n$ and determinant $\Delta$. Then both the closed points in $\mathbb{P}_{(X,P)}$ corresponding to $E$ and $\mathrm{gr}_E$ lie in a uniquely determined (by $E$) connected component $\mathbb{P}_F$. Further, for two bundles $E_1$ and $E_2$ whose points lie in a $\mathbb{P}_F$, the bundles $E_1$ and $E_2$ are $\mathrm{S}$-equivalent if and only if $g^*E_1$ and $g^*E_2$ are $\mathrm{S}$-equivalent.
\end{proposition}

\begin{proof}
From \Cref{rmk_canonical_filtration}, recall that we have a canonical filtration
\[
0 = \mathcal{S}_0(E) \subset \mathcal{S}_1(E) \subset \cdots \subset \mathcal{S}_l(E) = E
\]
such that for each $i$, the quotient $\mathcal{S}_{i+1}(E)/\mathcal{S}_i(E) \cong \mathcal{S}\left( E/\mathcal{S}_i(E) \right)$ is the socle of $E/\mathcal{S}_i(E)$ which is $P$-polystable of $P$-slope $\mu_P(E)$. Then $\mathrm{gr}_E \cong \oplus_{\substack{0 \leq l-1}} \, \mathcal{S}_{i+1}(E)/\mathcal{S}_i(E)$. Since taking a tensor product with a line bundle is an exact functor, we have $\mathcal{S}_i(E \otimes L) \cong \mathcal{S}_i(E) \otimes L$ for each $i$, and $\mathrm{gr}_{E \otimes L} \cong \mathrm{gr}_E \otimes L$. By Proposition~\ref{prop_disjoint_images}, the bundle $E$ uniquely determines an $F = \oplus_{1 \leq j \leq n-1} \, L_j \in \mathcal{F}$ such that the closed point $[E]$ determined by $E$ is in $\mathbb{P}_F$. Define a filtration $\{F_i\}_{1 \leq i \leq l}$ of $F$ as follows.

For $0 \leq i \leq l$, set
\[
I_i \coloneqq \{ 1 \leq j \leq n-1 \, | \, \text{\rm the restriction } L_j \hookrightarrow E \otimes L \, \text{\rm induces } L_j \hookrightarrow \mathcal{S}_i(E) \otimes L\}.
\]
Then $I_0 = \emptyset$, and $I_l = \{1, \ldots, n-1\}$. Consider the bundle $F_i \coloneqq \oplus_{j \in I_i} \, L_j$. Then we have
\[
F_{i+1}/F_i \cong \oplus_{j \in I_{i+1}\setminus I_i} \, L_j \hookrightarrow \left( \mathcal{S}_{i+1}(E)/\mathcal{S}_i(E) \right) \otimes L
\]
for each $i$. Taking the direct sum, $F \cong \oplus_{1 \leq i\leq l-1} \, F_{i+1}/F_i$ is a sub-bundle of $\mathrm{gr}_E \otimes L \cong \mathrm{gr}_{E \otimes L}$. This produces a short exact sequence
\[
0 \longrightarrow F \longrightarrow \mathrm{gr}_E \otimes L \longrightarrow \mathrm{det}(\mathrm{gr}_E \otimes L) \otimes \mathrm{det}(F)^{-1} \longrightarrow 0.
\]
Since $\mathrm{det}(\mathrm{gr}_E \otimes L) \cong \Delta \otimes L^n$, we see that the closed point $[\mathrm{gr}_E] \in \mathbb{P}_{(X,P)}$ also lies in $\mathbb{P}_F$.

On the other hand, the closed point $[\mathrm{gr}_E] \in \mathbb{P}_{(X,P)}$ lies in a unique $\mathbb{P}_F$ by Proposition~\ref{prop_disjoint_images}. Then by the above paragraph and loc. cit., the closed point $[E]$ also lies in the same $\mathbb{P}_F$.

The forward direction of the last statement is by Proposition~\ref{prop_S-equivariance_wrt_set_up}. So assume that $g^*E_1$ and $g^*E_2$ are $\mathrm{S}$-equivalent. By the first statement, the associated graded bundles $\mathrm{gr}_{E_1}$ and $\mathrm{gr}_{E_2}$ represent points in $\mathbb{P}_F$. Since $\mathrm{gr}_{g^*E_i} \cong g^* \mathrm{gr}_{E_i}$ as $\Gamma$-equivariant bundles on $Z$ for $i = 1, \, 2$ by Proposition~\ref{prop_S-equivariance_wrt_set_up}, Lemma~\ref{lem_implication_of_F-uniqueness} shows that $E_1 \sim_{\mathrm{S}} E_2$ on $(X,P)$.
\end{proof}

\subsection{Results on the non-emptiness}\label{sec_Non-Emptiness}
In this section, we list some results on the existence of $P$-semistable bundles in individual $\mathbb{P}_F$'s. The main result says that if there is a $\Gamma$-equivariant semistable bundle of rank $n$ and determinant $g^*\Delta$, then each over-parameterizing space $\mathbb{P}_F$ contains a non-empty open subscheme $U_F$ whose points are represented by $P$-semistable bundles; moreover, the points in $U_F$ are in bijective correspondence with those in $U_{F'}$ for any $F, \, F' \in \mathcal{F}$.

\begin{proposition}\label{prop_geometric_F-correspondence}
Let $(X,P) = (X,\tilde{P}) = [Z/\Gamma]$, and $u \colon Z \longrightarrow (X,P)$ be the natural atlas in Notation~\ref{not_fixed}. Suppose that $F \neq F' \in \mathcal{F}$. Then there is a natural bijective correspondence between the $P$-semistable bundles representing closed points in $\mathbb{P}_F$ with those representing closed points in $\mathbb{P}_{F'}$.
\end{proposition}

\begin{proof}
Without loss of generality, we may assume that there is a $P$-semistable bundle $E$ of rank $n$ and determinant $\Delta$ that represents a $k$-point in $\mathbb{P}_F$. Then the $\Gamma$-equivariant semistable bundle $u^*E$ contains $\Gamma$-equivariant sub-bundle $h_1 \colon u^*F \otimes u^*L^{-1} \hookrightarrow u^*E$. We have an isomorphism $l_1 \colon u^*F \otimes u^*L^{-1} \xlongrightarrow{\cong} u^*F' \otimes u^*L^{-1}$ which is not $\Gamma$-equivariant. Since the category of coherent sheaves constitutes an abelian category, we consider the push-out square
\begin{equation*}
\begin{tikzcd}
u^*F' \otimes u^*L^{-1} \arrow[r, "h_2"] \arrow[dr, phantom, "\text{\rm co}-\square"] & \tilde{E} \\
u^*F \otimes u^*L^{-1} \arrow[u, "l_1"] \arrow[r, hookrightarrow, "h_1"] & u^*E \arrow[u, "l_2"]
\end{tikzcd}
\end{equation*}
It follows that the above diagram is also Cartesian, $l_2$ is an isomorphism that is not $\Gamma$-equivariant, $\tilde{E}$ is a $\Gamma$-equivariant bundle, and $u^*F' \otimes u^*L^{-1}$ is a $\Gamma$-equivariant sub-bundle of $\tilde{E}$. So there is a unique (up to a canonical isomorphism) bundle $E' \in \mathrm{Vect}(X,P)$ such that $u^*E' \cong \tilde{E}$ as $\Gamma$-equivariant bundles. Since $u^*E$ is semistable, so is $u^*E'$, and hence, $E'$ is $P$-semistable. By \Cref{prop_disjoint_images}, $E'$ represents a $k$-point in $\mathbb{P}_{F'}$. It is easy to check that the above argument gives the required bijective correspondence.
\end{proof}

\begin{proposition}\label{prop_F-correspondence_descent_tau}
Let $F \in \mathcal{F}$. The functor $\tau^*$ defines a bijective correspondence between the $P$-semistable bundles representing points in $\mathbb{P}_F$ and the $\tilde{P}$-semistable bundles on $(X,\tilde{P})$ containing $\tau^*F \otimes \tau^*L^{-1}$ as a sub-bundle and which are of rank $n$ and determinant $\tau^*\Delta$ where the inverse is induced by $\tau_*$. The bijective correspondence restricts to stable and polystable bundles.
\end{proposition}

\begin{proof}
It is clear that if $E \in \mathrm{Vect}(X,P)$ is $P$-semistable, such that $[E] \in \mathbb{P}_F$, then $\tau^*E$ is a $\tilde{P}$-semistable bundle on $(X,\tilde{P})$ containing $\tau^*F \otimes \tau^*L^{-1}$ as a sub-bundle and which is of rank $n$ and determinant $\tau^*\Delta$.

Suppose that $\tilde{E}$ is a $\tilde{P}$-semistable bundle containing $\tau^*F \otimes \tau^*L^{-1}$ as a sub-bundle and which is of rank $n$ and determinant $\tau^*\Delta$. We obtain a short exact sequence
\[
0 \longrightarrow \tau^*F \otimes \tau^*L^{-1} \longrightarrow \tilde{E} \longrightarrow \tau^*\Delta \otimes \tau^*L^{\otimes (n-1)} \otimes \tau^* \mathrm{det}(F)^{-1} \longrightarrow 0.
\]
The evaluation map $\tau^*\tau_* \Rightarrow \mathrm{id}$ defines the following commutative diagram.
\begin{equation*}
\begin{tikzcd}
0 \arrow[r] & \tau^*F \otimes \tau^*L^{-1} \arrow[r] \arrow[d, "\cong"] & \tau^*\tau_*\tilde{E} \arrow[d, "h"] \arrow[r] & \tau^*\Delta \otimes \tau^*L^{\otimes (n-1)} \otimes \tau^* \mathrm{det}(F)^{-1} \arrow[r] \arrow[d, "\cong"] & 0 \\
0 \arrow[r] & \tau^*F \otimes \tau^*L^{-1} \arrow[r] & \tilde{E} \arrow[r] & \tau^*\Delta \otimes \tau^*L^{\otimes (n-1)} \otimes \tau^* \mathrm{det}(F)^{-1} \arrow[r] & 0
\end{tikzcd}
\end{equation*}
where we have used the projection formula and the isomorphism $\tau_* \mathcal{O}_{(X,\tilde{P})} \cong \mathcal{O}_{(X,P)}$. This shows that $h \colon \tau^*\tau_*\tilde{E} \longrightarrow \tilde{E}$ is also an isomorphism. In particular, $\tau_*\tilde{E}$ is $P$-semistable of rank $n$ and of determinant $\Delta$. Moreover, $\tilde{E}$ is $\tilde{P}$-stable (respectively, $\tilde{P}$-polystable) if and only if $\tau_* \tilde{E}$ is $P$-stable (respectively, $P$-polystable).
\end{proof}

Let us summarize the above observations.

\begin{corollary}\label{cor_root_of_Assumption}
Suppose that there is a $\Gamma$-equivariant semistable bundle on $Z$ of rank $n$ and determinant $g^*\Delta$. Then the following hold.
\begin{enumerate}[leftmargin=*]
    \item For each $F \in \mathcal{F}$, there is a $P$-polystable bundle in $\mathbb{P}_F$. In particular, the sub-scheme of $\mathbb{P}_F$ representing $P$-semistable bundles is open and non-empty.
    \item For any $F, \, F' \in \mathcal{F}$, the $P$-semistable (respectively, $P$-polystable) bundles in $\mathbb{P}_F$ are in bijective correspondence with the $P$-semistable (respectively, $P$-polystable) bundles in $\mathbb{P}_{F'}$.
\end{enumerate}
\end{corollary}

\subsection{The Unstable Locus}\label{sec_unstable}
Set $\delta' \coloneqq (n, |\Gamma| \mathrm{deg}_P(\Delta))$ as in Notation~\ref{not_delta}. We noted the following observations from~\Cref{sec_A_determinantal_line_bundle}. On the curve $Z$ of genus $g_Z$, for any positive integer $R$, we have the generalized $\Theta$-line bundle $\mathcal{O}_{\mathbb{P}_Z}(R \cdot \Theta) = \lambda_{\mathcal{E}_Z}(W)^{-1}$ where $W$ is any bundle on $Z$ of rank $R \cdot n/\delta'$ and degree $- R \cdot \frac{ \chi'}{\delta'}$, and $\chi' = |\Gamma| \mathrm{deg}_P(\Delta) - n (g_Z - 1)$; cf. Definition~\ref{def_Theta_line_bundle}. For each $W$ as above, we have a distinguished sections $s_{Z, W} \in \Gamma(\mathbb{P}_Z, \mathcal{O}_{\mathbb{P}_Z}(R \cdot \Theta))$ whose vanishing locus is
\[
\theta_{Z, W} = \{ s \in \mathbb{P}_Z(k) \, | \, \mathrm{H}^*(Z, \mathcal{E}_{Z,s} \otimes W) \neq 0 \}.
\]
Taking a fixed $R > \frac{n^2}{4}(g_Z - 1)$, Le Potier's result \cite[Theorem~2.4]{Potier} implies that the closed sub-scheme $B_Z = \{[E] \in \mathbb{P}_Z \, | \, E \, \text{\rm is not semistable}\}$ is the base locus of the finite-dimensional complete linear system associated to $\mathcal{O}_{\mathbb{P}_Z}(R \cdot \Theta)$; see \Cref{prop_bl}. Moreover, the sections $S_{Z,W}$ generate $\mathcal{O}_{\mathbb{P}_Z}(R \cdot \Theta)$ on the complement $\mathbb{P}_Z - B_Z$.

Our aim is to show that for each $F \in \mathcal{F}$, the unstable locus in $\mathbb{P}_F$, namely, the sub-scheme
\[
B_F \coloneqq \{ [E] \in \mathbb{P}_F \, | \, E \, \text{\rm is not } P\text{-}\mathrm{ semistable}\}
\]
is a closed sub-scheme that is also the base locus of the complete linear system associated to $\iota_F^* \mathcal{O}_{\mathbb{P}_Z}(R \cdot \Theta)$ for an $R > \frac{n^2}{4}(g_Z - 1)$ where $\iota_F^* \colon \mathrm{Pic}(\mathbb{P}_Z) \longrightarrow \mathrm{Pic}(\mathbb{P}_F)$ is the induced homomorphism (see \Cref{prop_disjoint_images} and \Cref{def_F} for notation). Let us fix $F \in \mathcal{F}$.

First, note that we have a commutative diagram
\begin{equation}\label{eq_big_diagram}
\begin{tikzcd}
\mathbb{P}_Z \arrow[dr, phantom, "\square"] & \mathbb{P}_Z \times Z \arrow[l, "p_Z"] \arrow[r, swap, "q_Z"] & Z \\
\mathbb{P}_F \arrow[u, hook, "\iota_F"] \arrow[d, equal, "\mathrm{id}"] & \mathbb{P}_F \times Z \arrow[dr, phantom, "\square"] \arrow[u, hook, "\iota_F \times \mathrm{id}_Z"] \arrow[r, swap, "\tilde{q}_Z"] \arrow[l, "\tilde{p}_Z"] \arrow[d, swap, "\text{\rm id} \times g"] & Z \arrow[d, swap, "g"] \arrow[u, equal, swap, "\mathrm{id}_Z"] \\
\mathbb{P}_F & \mathbb{P}_F \times (X,P) \arrow[l, "p_F"] \arrow[r, swap, "q_F"] & (X,P)
\end{tikzcd}
\end{equation}
where the bottom right and the top left squares are Cartesian with obvious arrows. Also note that in the top left Cartesian square, the projection $p_Z$ is a flat map of varieties, and the closed inclusion $\iota_F$ is a proper map --- so we can use the flat and the proper base change theorems with respect to this maps, respectively; on the other hand, in the right bottom Cartesian square, the map $g$ is a representable finite flat map and the projection $q_F$, itself being a base change of the structure morphism of the projective variety $\mathbb{P}_F$, is a flat and proper morphism --- so again, the appropriate base change theorems are applicable; cf. \cite[Proposition~A.3.]{Das}. We also saw in Proposition~\ref{prop_disjoint_images} that we have the isomorphism
\[
\left( \mathrm{id} \times g \right)^* \mathcal{E}_F \cong \left( \iota_F \times \mathrm{id}_Z \right)^*\mathcal{E}_Z.
\]
Using the base change theorems and the projection formula (\cite[Proposition~A.3.]{Das}), the commutativity of the above diagram, and the fact that $\left(\mathrm{id} \times g\right)_*$ is an exact functor, we obtain the following: for any $W \in \mathrm{Vect}(Z)$, and for any $i \geq 0$, there are canonical isomorphisms
\begin{multline}\label{multi:1}
\iota_F^* \mathrm{R}^ip_{Z,*} \left( \mathcal{E}_Z \otimes q_Z^*W \right) \, \cong \mathrm{R}^i\tilde{p}_{Z,*} \left( \iota_F \times \mathrm{id}_Z \right)^* \left( \mathcal{E}_Z \otimes q_Z^*W \right) \\
\cong \mathrm{R}^i\tilde{p}_{Z,*} \left( \left( \iota_F \times \mathrm{id}_Z \right)^* \mathcal{E}_Z \otimes \left( \iota_F \times \mathrm{id}_Z \right)^* q_Z^*W \right)
\cong \mathrm{R}^i\tilde{p}_{Z,*} \left(\left( \mathrm{id} \times g \right)^* \mathcal{E}_F \otimes \tilde{q}_Z^*W \right) \\
\cong \mathrm{R}^i p_{F,*} \circ \left( \mathrm{id} \times g \right)_* \left(\left( \mathrm{id} \times g \right)^* \mathcal{E}_F \otimes \tilde{q}_Z^*W \right)
\cong \mathrm{R}^i p_{F,*} \left( \mathcal{E}_F \otimes \left( \mathrm{id} \times g \right)_* \tilde{q}_Z^*W \right) \\
\cong \mathrm{R}^i p_{F,*} \left( \mathcal{E}_F \otimes q_F^*g_*W \right).    
\end{multline}
The above isomorphisms show that we have a well-defined homomorphism
\[
\lambda_F \colon \mathrm{H}(n, |\Gamma| \mathrm{deg}_P(\Delta)) \longrightarrow \mathrm{Pic}(\mathbb{P}_F)
\]
given by
\[
\lambda_F(W) \coloneqq \otimes \, \mathrm{det} \, \mathrm{R}^i p_{F, *} \left( \mathcal{E}_F \otimes q_F^* g_* W \right)^{(-1)^i} \cong \iota_F^*\lambda_Z(W)
\]
where $\mathrm{H}(n, |\Gamma| \mathrm{deg}_P(\Delta))$ is the subgroup of the Grothendieck group $K(Z)$, generated by the classes of the coherent sheaves of rank $l \cdot n/\delta'$ and degree $-l \cdot \frac{|\Gamma| \mathrm{deg}_P(\Delta) - n (g_Z - 1)}{\delta'}$ for $l \geq 1$, $W$ is a bundle on $Z$ whose class is in $\mathrm{H}(n, |\Gamma| \mathrm{deg}_P(\Delta))$, $\lambda_Z = \lambda_{\mathcal{E}_Z}$ is defined as in~\eqref{eq_lambda_definition}, and $\iota_F^* \colon \mathrm{Pic}(\mathbb{P}_Z) \longrightarrow \mathrm{Pic}(\mathbb{P}_F)$ is the induced homomorphism. Moreover, the line bundle $\lambda_{F}(W)$ only depends of the degree and the rank of $W$ since the same holds for the line bundle $\lambda_Z(W)$.

Next, consider any bundle $W$ on $Z$ of rank $R \cdot n/\delta'$ and degree $-R \cdot \chi'/\delta'$ for some integer $R \geq 1$. We are going to construct a distinguished section $s_{F,W}$ of $\lambda_F(W)^{-1}$. Consider any locally free resolution
\[
0 \longrightarrow V_1 \overset{\alpha}\longrightarrow V_0 \longrightarrow \mathcal{E}_Z \otimes q_Z^*W \longrightarrow 0
\]
such that $p_{Z,*}V_i = 0$ for $i = 1, \, 2$. Applying the proper map $\left( \iota_F \times \mathrm{id}_Z \right)^*$, we obtain the exact sequence:
\[
\left( \iota_F \times \mathrm{id}_Z \right)^* V_1 \xrightarrow{\left( \iota_F \times \mathrm{id}_Z \right)^* \alpha} \left( \iota_F \times \mathrm{id}_Z \right)^* V_0 \longrightarrow \left( \iota_F \times \mathrm{id}_Z \right)^* \left( \mathcal{E}_Z \otimes q_Z^*W \right) \longrightarrow 0.
\]
Let $\tilde{V}_1$ be the kernel of the above surjection. Then $\tilde{V}_1$ is a sub-bundle of $\left( \iota_F \times \mathrm{id}_Z \right)^* V_0$ that is the image of $\left( \iota_F \times \mathrm{id}_Z \right)^* V_1$ under the map $\left( \iota_F \times \mathrm{id}_Z \right)^* \alpha$. We obtain the short exact sequence:
\[
0 \longrightarrow \tilde{V}_1 \xrightarrow{\tilde{\alpha}} \left( \iota_F \times \mathrm{id}_Z \right)^* V_0 \longrightarrow \left( \iota_F \times \mathrm{id}_Z \right)^* \left( \mathcal{E}_Z \otimes q_Z^*W \right) \longrightarrow 0.
\]
Since $p_{Z,*}V_0 = 0$, by the base change theorem, we have $\tilde{p}_{Z,*} \left( \iota_F \times \mathrm{id}_Z \right)^* V_0 \cong \iota_F^* p_{Z,*}V_0 = 0$ from~\eqref{eq_big_diagram}. So, the above exact sequence produced the long exact sequence
\begin{multline}\label{multi:2}
0 \longrightarrow \tilde{p}_{Z,*}\left(\iota_F \times \mathrm{id}_Z\right)^*\left( \mathcal{E}_Z \otimes q_Z^*W \right) \longrightarrow \mathrm{R}^1\tilde{p}_{Z,*} \tilde{V}_1 \xrightarrow{\mathrm{R}^1\tilde{\alpha}} \mathrm{R}^1\tilde{p}_{Z,*}\left(\iota_F \times \mathrm{id}_Z\right)^* V_0 \longrightarrow \\ \mathrm{R}^1\tilde{p}_{Z,*}\left(\iota_F \times \mathrm{id}_Z\right)^*\left( \mathcal{E}_Z \otimes q_Z^*W \right) \longrightarrow 0    
\end{multline}
of coherent sheaves on $\mathbb{P}_F$. From~\eqref{multi:1}, we have
\begin{align*}
\iota_F^*\mathrm{R}^ip_{Z,*}\left( \mathcal{E}_Z \otimes q_Z^*W \right) \cong \mathrm{R}^i\tilde{p}_{Z,*}\left(\iota_F \times \mathrm{id}_Z\right)^*\left( \mathcal{E}_Z \otimes q_Z^*W \right) \cong \mathrm{R}^ip_{F,*} \left( \mathcal{E}_F \otimes q_F^*g_*W \right)
\end{align*}
for $i = 0, \, 1$. Similarly, we also have
\[
\iota_F^* \mathrm{R}^1p_{Z,*}V_0 \cong \mathrm{R}^1\tilde{p}_{Z,*} \left( \iota_F \times \mathrm{id}_Z \right)^* V_0.
\]

We claim that the Euler characteristic of the bundle $\tilde{p}_{Z,*}\left( \iota_F \times \mathrm{id}_Z \right)^*\left( \mathcal{E}_Z \otimes q_Z^*W \right)$ on $\mathbb{P}_F$ is zero. To see this, we first prove a small lemma.

\begin{lemma}\label{lem_short:1}
Under the above notation, for any closed point $s \in \mathbb{P}_F$, we have an isomorphism
\[
g^* \mathcal{E}_{F,s} \cong \mathcal{E}_{Z, \iota_F(s)}.
\]
\end{lemma}

\begin{proof}
From Proposition~\ref{prop_disjoint_images}, we know that $(\mathrm{id}_{\mathbb{P}_F} \times g)^* \mathcal{E}_F \cong (\iota_F \times \mathrm{id}_Z)^* \mathcal{E}_Z$. Using this isomorphism together with the projection formula and the base change theorems, we obtain the following isomorphisms:
\begin{multline*}
g^* \mathcal{E}_{F,s} \cong g^* q_{F,*} \left( \mathcal{E}_F \otimes p_F^*k(s) \right) \cong \tilde{q}_{Z,*} \left( \mathrm{id}_{\mathbb{P}_F} \times g \right)^* \left( \mathcal{E}_F \otimes p_F^*k(s) \right) \\
\cong \tilde{q}_{Z,*} \left( \left( \iota_F \times \mathrm{id}_Z \right)^* \mathcal{E}_Z \otimes \tilde{p}_Z^* k(s) \right) \cong q_{Z,*} \left( \iota_F \times \mathrm{id}_Z \right)_* \left[ \left( \iota_F \times \mathrm{id}_Z \right)^* \mathcal{E}_Z \otimes \tilde{p}_Z^*k(s) \right] \\
\cong q_{Z,*} \left( \mathcal{E}_Z \otimes \left( \iota_F \times \mathrm{id}_Z \right)_* \tilde{p}_Z^* k(s) \right) \cong q_{Z,*}(\mathcal{E}_Z \otimes p_Z^* k(\iota_F(s))) \cong \mathcal{E}_{Z, \iota_F(s)}.    
\end{multline*}
\end{proof}

As a consequence of the above result, for any closed point $s \in \mathbb{P}_F$, for $i = 0, \, 1$, we obtain
\begin{equation}\label{eq_corollary_consequence}
\mathrm{H}^i((X,P), \mathcal{E}_{F, s} \otimes g_*W) \cong \mathrm{H}^i(Z, g^*\mathcal{E}_{F, s} \otimes W) \cong \mathrm{H}^i(Z, \mathcal{E}_{Z, \iota_P(s)} \otimes W).
\end{equation}
By our assumption on $W$ and the flat base change theorem, the Euler characteristic of the bundle $\tilde{p}_{Z,*}\left(\iota_F \times \mathrm{id}_Z\right)^*\left( \mathcal{E}_Z \otimes q_Z^*W \right)$ on $\mathbb{P}_F$ is zero, proving our claim.

Hence, for any closed point $s \in \mathbb{P}_F$, the $k$-vector spaces $\mathrm{R}^1\tilde{p}_{Z,*} \tilde{V}_1 \otimes k(s)$ and $\iota_F^*p_{Z,*}V_0 \otimes k(s)$ are of the same dimension. Again by using isomorphisms as in Lemma~\ref{lem_short:1}, it follows that this dimension is given by $ r = \mathrm{dim}_k \, \mathrm{H}^1(Z, V_{0, k(\iota_F(s))})$. We obtain the distinguished section
\begin{equation}\label{eq_def_s_P,W}
s_{F,W} \coloneqq \mathrm{det} \, \mathrm{R}^1 \tilde{\alpha} = \mathrm{Hom}\left( \wedge^r \mathrm{R}^1\tilde{p}_{Z,*} \tilde{V}_1, \wedge^r \iota_F^*p_{Z,*}V_0 \right) \in \Gamma(\mathbb{P}_F, \iota_F^*\mathcal{O}_{\mathbb{P}_Z}(R \cdot \Theta))
\end{equation}
associated to $W$, up to a non-zero scalar in $k$. From the above construction and the definition of $s_{Z,W}$, it is not hard to see that the above section is independent of the choice of the resolution, and $s_{F,W}$ is induced from $s_{Z,W}$. More precisely, the natural $k$-linear map
\begin{equation}\label{eq_surjective_k-linear}
\mathrm{H}^0(\mathbb{P}_Z, \mathcal{O}_{\mathbb{P}_Z}(R \cdot \Theta)) \xlongrightarrow{\iota_F^*} \mathrm{H}^0(\mathbb{P}_F, \iota_F^*\mathcal{O}_{\mathbb{P}_Z}(R \cdot \Theta))    
\end{equation}
is surjective, and $\iota_F^*s_{Z,W} = s_{F,W}$.

Next, we establish the relation between the vanishing locus of $s_{F,W}$ and that of $s_{Z,W}$, in terms of vanishing of certain cohomology groups.

\begin{lemma}\label{lem_cohomology_vanishing_unstable_locus}
Under the above notation, the vanishing locus of $s_{F,W}$ is given by
\[
\theta_{F,W} \, = \, \{s \in \mathbb{P}_F(k) \, | \, \mathrm{H}^*(Z , \mathcal{E}_{Z, \iota_F(s)} \otimes W) \neq 0\}
\]
which is the scheme theoretic intersection $\theta_{Z,W} \cap \mathbb{P}_F = \iota_F^* \theta_{Z,W}$. In particular, the ideal sheaf of $\theta_{F,W}$ is the inverse image ideal sheaf of $\theta_{Z, W}$ under $\iota_F$.
\end{lemma}

\begin{proof}
Let $s \in \mathbb{P}_F$ be a closed point. Since the locally free sheaves $\mathrm{R}^1\tilde{p}_{Z,*} \tilde{V}_1$ and $\iota_F^*p_{Z,*}V_0$ are of the same rank $r$, we see that the following conditions are equivalent.
\begin{enumerate}
\item $\mathrm{det} \, \mathrm{R}^1\tilde{\alpha} \neq 0$ at $s$;
\item $\mathrm{R}^1\tilde{\alpha}$ is surjective at $s$;
\item $\mathrm{R}^1p_{F,*} \left( \mathcal{E}_F \otimes q_F^*g_*W \right)$ vanishes at $s$;
\item $\mathrm{H}^1((X,P), \mathcal{E}_{F,s} \otimes g_* W) = 0$.\label{zero}
\end{enumerate}
By~\eqref{eq_corollary_consequence}, the last condition~\eqref{zero} is equivalent to: $\mathrm{H}^1(Z, \mathcal{E}_{Z, \iota_F(s)} \otimes W) = 0$. We obtain:
\begin{itemize}
\item the section $s_{F,W}$ vanishes at a closed point $s \in \mathbb{P}_F$ if and only if
\item $\mathrm{H}^1(Z, \mathcal{E}_{Z, \iota_F(s)} \otimes W) = 0$.
\end{itemize}
Finally, by the Riemann Roch Theorem, the choice of the rank and determinant of $W$ tells us that the last condition is equivalent to:
\[
\mathrm{H}^0(Z, \mathcal{E}_{Z, \iota_F(s)} \otimes W) = \mathrm{H}^1(Z, \mathcal{E}_{Z, \iota_F(s)} \otimes W) = 0.
\]
This shows that
\[
\theta_{F,W} \, = \, \{s \in \mathbb{P}_F(k) \, | \, \mathrm{H}^*(Z , \mathcal{E}_{Z, \iota_F(s)} \otimes W) \neq 0\}.
\]
Since the condition $\mathrm{H}^*(Z, \mathcal{E}_{Z, \iota_F(s)} \otimes W) = 0$ is equivalent to: $\iota_F(s) \in \theta_{Z,W}$, we have:
\[
\theta_{F,W} = \{\iota_F(s) \in \mathbb{P}_Z \, | \, s_{Z,W} \, \text{\rm vanishes at } \iota_F(s)\} = \theta_{Z,W} \cap \mathbb{P}_F.
\]
The particular statement is immediate.
\end{proof}

Let us summarize the above discussion together with a proof of the fact that the unstable locus in $\mathbb{P}_{(X,P)}$ is the base locus of a line bundle.

\begin{theorem}\label{thm_unstable_locus}
Suppose that Notation~\ref{not_fixed} hold. Set
\[
\delta' = (n, |\Gamma|\text{\rm deg}_P(\Delta)) \, \text{\rm and} \, \chi' = |\Gamma|\text{\rm deg}_P(\Delta) - n (g_Z - 1).
\]
Fix an integer $R > \frac{n^2}{4}(g_Z - 1)$. Let $F \in \mathcal{F}$. Consider the closed projective subspace $\iota_F \colon \mathbb{P}_F \hookrightarrow \mathbb{P}_Z$. Then the following hold.
\begin{enumerate}
\item For any bundle $W$ on $Z$ of rank $R \cdot n/\delta'$ and of degree $- R \cdot \chi'/\delta'$, the distinguished section $s_{Z, W}$ of the generalized $\Theta$-line bundle on $\mathbb{P}_Z$ (cf. Definition~\ref{def_Theta_line_bundle}) associated $W$ induces a distinguished section $s_{F,W} = \iota_F^* s_{Z,W}$ of the line bundle $\iota_F^*\mathcal{O}_{\mathbb{P}_Z}(R \cdot \Theta)$. In particular, the natural $k$-linear map
\[
\mathrm{H}^0(\mathbb{P}_Z, \mathcal{O}_{\mathbb{P}_Z}(R \cdot \Theta)) \xlongrightarrow{\iota_F^*} \mathrm{H}^0(\mathbb{P}_F, \iota_F^*\mathcal{O}_{\mathbb{P}_Z}(R \cdot \Theta))
\]
is surjective.
\item The vanishing locus of the section $s_{F,W}$ is given by
\[
\theta_{F,W} \, = \, \{s \in \mathbb{P}_F(k) \, | \, \mathrm{H}^*(Z , \mathcal{E}_{Z, \iota_F(s)} \otimes W) = 0\} \, = \, \theta_{Z,W} \cap \mathbb{P}_F
\]
where $\theta_{Z,W}$ is the vanishing locus of the section $s_{Z,W}$, as described in~\eqref{eq_classical_vanishing_locus}. If $\mathcal{I}_W$ denote the ideal sheaf of $\theta_{Z,F}$, then $\theta_{F,W}$ is the closed sub-scheme with ideal sheaf $\iota_F^{-1}\mathcal{I}_W \cdot \mathcal{O}_{\mathbb{P}_F}$ for any $W$ as above.
\item The unstable locus in $\mathbb{P}_F$, i.e. the closed sub-scheme
\[
B_F = \{ [E] \in \mathbb{P}_F(k) \, | \, E \, \text{\rm is not }P\text{-}\text{\rm semistable}\}
\]
is the base locus
\[
\cap_{\substack{W \text{ semistable}\\ \mathrm{rk}(W) = R \cdot n/\delta',\\ \mathrm{deg}(W) = - R \cdot \chi'/\delta'}} \, \theta_{F,W} = B_Z \cap \mathbb{P}_F
\]
of the complete linear system of the line bundle $\iota_F^*\mathcal{O}_{\mathbb{P}_Z}(R \cdot \Theta)$; here $B_Z$ is the base locus of the linear system of the generalized $\Theta$-line bundle $\mathcal{O}_{\mathbb{P}_Z}(R \cdot \Theta)$. In particular, the complete linear system of the line bundle $\iota_F^*\mathcal{O}_{\mathbb{P}_Z}(R \cdot \Theta)$ is of finite dimension and is generated by the pullback of the generalized theta divisors $\theta_{Z,W}$ on $Z$ under $\iota_F$.\label{un}
\end{enumerate}
\end{theorem}

\begin{proof}
Only~\eqref{un} is new. It is enough to show that $B_F = B_Z \cap \mathbb{P}_F$.

Given a bundle $E$ with $[E] \in \mathbb{P}_F$, we know that $E$ is $P$-semistable if and only if the $\Gamma$-equivariant bundle $g^*E$ on $Z$ is semistable. If $g^*E$ on is semistable, by Le Potier's result (\cite[Theorem~2.4]{Potier}), there exists a necessarily semistable bundle $W$ on $Z$ of rank $R \cdot n/\delta'$ and of degree $- R \cdot \chi'/\delta'$ such that $\mathrm{H}^*(Z, g^*E \otimes W) = 0$. So $E$ is not $P$-semistable if and only if for each semistable bundle $W$ on $Z$ of rank $R \cdot n/\delta'$ and of degree $- R \cdot \chi'/\delta'$, we have
\[
\mathrm{H}^*(Z, g^*E \otimes W) \neq 0.
\]
As in the proof of Lemma~\ref{lem_cohomology_vanishing_unstable_locus}, we conclude that
\[
B_F = \cap_{\substack{W \text{ semistable}\\ \mathrm{rk}(W) = R \cdot n/\delta',\\ \mathrm{deg}(W) = - R \cdot \chi'/\delta'}} \, \theta_{F,W}
\]
which is the scheme theoretic intersection of $B_Z$ with $\mathbb{P}_F$.
\end{proof}

\subsection{Main Theorem}\label{sec_Main}
We now have all the ingredients to construct a Coarse moduli space for the algebraic stack $\mathcal{M}_{(X,P)}^{\mathrm{ss, S}}(n, \Delta)$. Throughout this section, we work with the following assumption.
\begin{ass}\label{ass_non-empty}
There exists a $P$-semistable bundle $E \in \mathrm{Vect}(X,P)$ of rank $n$ and of determinant $\Delta$.
\end{ass}

We have seen that the $P$-semistable bundles $E$ on $(X,P)$ of rank $n$ and of determinant $\Delta$ are over-parameterized by the projective space $\mathbb{P}_{(X,P)}$ which is the disjoint union
\[
\mathbb{P}_{(X,P)} = \sqcup_F \, \mathbb{P}_F = \mathrm{Grass}(1, \mathrm{H}^1((X,P), F \otimes \mathrm{det}(F) \otimes \Delta^{-1} \otimes L^{-n}))
\]
with $F$ varying over the indexing set $\mathcal{F}$ (cf. \Cref{def_F}). Further, $\mathbb{P}_F$'s are closed subspaces of $\mathbb{P}_Z$ via the closed embeddings $\mathbb{P}_F \xhookrightarrow{\iota_F} \mathbb{P}_Z$ with disjoint images where
\[
\mathbb{P}_Z = \mathrm{Grass}(1, \mathrm{H}^1(Z, \oplus_{1 \leq i \leq n-1} \, g^*\Delta^{-1} \otimes g^*L^{-n}))
\]
over-parametrizes semistable bundles on $Z$ of rank $n$ and determinant $g^*\Delta$. Fix an integer $R > \frac{n^2}{4}(g_Z - 1)$. By the classical theory (Proposition~\ref{prop_bl}), the closed subscheme $B_Z$ of $\mathbb{P}_Z$ of points corresponding to the bundles on $Z$ which are not semistable is the base locus of the complete linear system corresponding to the generalized $\Theta$-line bundle $\mathcal{L}_Z \coloneqq \mathcal{O}_{\mathbb{P}_Z}(R \cdot \Theta)$. With \Cref{ass_non-empty}, \Cref{cor_root_of_Assumption} assures that for each $F \in \mathcal{F}$, the sub-scheme of $\mathbb{P}_F$ whose closed points represents $P$-semistable bundles is open and non-empty. By \Cref{thm_unstable_locus}, the closed subscheme $B_F \subset \mathbb{P}_F$ of points corresponding to the bundles which are not $P$-semistable is precisely the base locus of the complete linear system associated to $\iota_F^*\mathcal{L}_Z$, and this coincides with the closed subscheme $\iota_F^*B_Z = B_Z \cap \mathbb{P}_F$, the scheme-theoretic intersection of $B_Z$ and $\mathbb{P}_F$ in $\mathbb{P}_Z$. So the sections $s_{F,W}$ generate $\iota_F^*\mathcal{L}_Z$ on the open subscheme $Q_F = \mathbb{P}_F - B_F$.

We have the following commutative diagram (\eqref{diag_classical_linear_system}) determined by the complete linear system associated to $\mathcal{L}_Z$:
\begin{equation*}\label{diag_classical}
\begin{tikzcd}
\mathbb{P}_Z \arrow[r, hookleftarrow] & Q_Z = \mathbb{P}_Z - B_Z \arrow[d, hookrightarrow] \arrow[r, "\psi_Z"] & \mathbb{P}^{N_Z} \\
 & \tilde{Q}_Z = \mathrm{Bl}_{B_Z}{\mathbb{P}_Z} \arrow[lu, "\pi_Z"] \arrow[ru, swap, "\tilde{\psi}_Z"] & 
\end{tikzcd}
\end{equation*}
By Langton's Theorem~\cite[Theorem~6.4]{Hein}, the restriction map $\psi_Z \colon Q_Z \longrightarrow \mathbb{P}^{N_Z}$ is a proper map whose image in $\mathbb{P}^{N_Z}$ coincides with that of the map $\widetilde{\psi}_Z$.

For each $F \in \mathcal{F}$, we now do the following construction. Consider the blowing-up $\pi_F \colon \tilde{Q}_F \longrightarrow \mathbb{P}_F$ of $\mathbb{P}_F$ with respect to the closed subscheme $B_F$. Since $\mathbb{P}_F$ is an integral projective variety and $B_F$ does not contain its generic point, $\tilde{Q}_F$ is an integral projective variety as well. Since the ideal sheaf of $B_F$ is the inverse image ideal sheaf of $B_Z$ by \Cref{thm_unstable_locus}, the closed immersion $\iota_F$ restricts to a closed immersion $\iota_F \colon Q_F \hookrightarrow Q_Z$, and by \cite[II, Corollary~7.15, page 165]{Ha}, there is a canonical closed immersion $\tilde{\iota}_F \colon \tilde{Q}_F \hookrightarrow \tilde{Q}_Z$ that realizes $\tilde{Q}_F$ as the strict transform of $\mathbb{P}_F$ under $\pi_Z$. The above construction has the following consequences.

\begin{theorem}\label{thm_Main_Construction}
Under the above notation and construction, we have the following commutative diagram.
\begin{equation}\label{huge_diag}
\begin{tikzcd}[column sep=small, text height=1ex, text depth=0.25ex]
 & \mathbb{P_Z} \arrow[rrdd, leftarrow, shift right, bend right=30, "\pi_Z" near end, dashed] \arrow[rr, hookleftarrow, dashed] & & Q_Z = \mathbb{P}_Z - B_Z \arrow[rr, rightarrow, "\psi_Z", dashed] \arrow[dd, hookrightarrow, dashed]  & & \mathbb{P}^{N_Z} \\
\mathbb{P}_F \arrow[rrdd, leftarrow, shift right, bend right=30, "\pi_Z" near start] \arrow[rr, hookleftarrow] \arrow[ur, hookrightarrow, "\iota_F"] & & Q_F = \mathbb{P}_F - B_F \arrow[dd, hookrightarrow] \arrow[ur, hookrightarrow, "\iota_F"] \arrow[rr, "\psi_F" near start] & & \mathbb{P}^{N_F} \arrow[ur, hookrightarrow, "j_F"]& \\
 & & & \tilde{Q}_Z = \mathrm{BL}_{B_Z}(\mathbb{P}_Z) \arrow[uurr, rightarrow, swap, "\tilde{\psi}_Z", dashed, bend right=30]& & & \\
 & & \tilde{Q}_F = \mathrm{BL}_{B_F}(\mathbb{P}_F) \arrow[ur, hookrightarrow, "\tilde{\iota}_F"] \arrow[uurr, rightarrow, swap, "\tilde{\psi}_Z", bend right=30]& & &
\end{tikzcd}
\end{equation}
Here, $\mathbb{P}^{N_Z} = \mathrm{Grass}\left(1, \mathrm{H}^0(\mathbb{P}_Z,\mathcal{L}_Z) \right), \, \mathbb{P}^{N_F} = \mathrm{Grass}\left(1,  \mathrm{H}^0(\mathbb{P}_F,\iota_F^*\mathcal{L}_Z) \right)$, the closed immersion $j_F$ is induced by the surjective $k$-linear map
\[
\iota_F^* \colon \mathrm{H}^0(\mathbb{P}_Z,\mathcal{L}_Z) \longrightarrow \mathrm{H}^0(\mathbb{P}_F, \iota_F^*\mathcal{L}_Z),
\]
and $\psi_F$, and $\tilde{\psi}_F$ are the canonical map induced by the complete linear system associated to $\iota_F^*\mathcal{L}_Z$. The following statements hold.
\begin{enumerate}[leftmargin=*]
\item The images of the maps $\psi_F$ and $\tilde{\psi}_F$ in $\mathbb{P}^{N_F}$ coincide.\label{T:1}
\item Considering the respective Stein factorization
\begin{eqnarray*}
\tilde{\psi}_Z \colon \tilde{Q}_Z \xlongrightarrow{g_Z} M_Z \xlongrightarrow{f_Z} \mathbb{P}^{N_Z} \text{ and } \\
\tilde{\psi}_F \colon \tilde{Q}_F \xlongrightarrow{g_F} M_F \xlongrightarrow{f_F} \mathbb{P}^{N_F},
\end{eqnarray*}
the canonical induced map $M_F \longrightarrow M_Z$ of projective integral varieties, which we again denote by $\tilde{\iota}_F$ by abuse of notation, is a closed immersion. Note that $M_Z$ is the Coarse moduli space of the algebraic stack $\mathcal{M}_Z^{\mathrm{ss, S}}(n, g^*\Delta)$ by \Cref{thm_moduli_classical}.\label{T:2}
\item The images of the maps $\tilde{\iota}_F \colon M_F \hookrightarrow M_Z$ are disjoint for distinct $F \in \mathcal{F}$.\label{T:3}
\item Each closed point of $M_F$ is the image of a closed point of $Q_F$, defining a surjective morphism
\[
h_F \colon Q_F \longrightarrow M_F.
\]
If two distinct closed points $q, \, q' \in Q_F$ map to the same point in $M_F$ under $h_F$, then the $P$-semistable bundles representing the points $q$ and $q'$ are $\mathrm{S}$-equivalent.\label{T:4}
\end{enumerate}
\end{theorem}

\begin{proof}
The commutative diagram exists by our construction and \Cref{thm_unstable_locus}. Again by the construction, the map $\psi_F$ is the restriction of the proper map $\psi_Z$ to the closed subspace $Q_F$ of $Q_Z$. So $\psi_F$ is a proper map, and hence the images of the maps $\psi_F$ and $\tilde{\psi}_F$ in $\mathbb{P}_F$ coincide, proving~\eqref{T:1}.

Since each of the Stein factorization is a relative normalization, we obtain a canonical induced map $\tilde{\iota}_F \colon M_F \longrightarrow M_Z$ by \cite[\href{https://stacks.math.columbia.edu/tag/035J}{Lemma 035J}]{SP}. Since $\tilde{\iota}_F \colon \tilde{Q}_F \hookrightarrow \tilde{Q}_Z$ is a closed immersion, by the universal property of the Stein factorization, statement~\eqref{T:2} is settled.

The statement~\eqref{T:3} is a consequence of~\eqref{T:4}, \Cref{lem_iamge_of_Q}, and Proposition~\ref{prop_disjoint_images}. Let us prove~\eqref{T:4}. An application of Mumford's Rigidity Lemma~\cite[Proposition~6.1, Chapter 6, page 115]{MumfordGIT} as in Lemma~\ref{lem_iamge_of_Q} shows that the finite map $f_F$ locally admits sections. As $\psi_F(Q_F) = \tilde{\psi}_F(\tilde{Q}_F)$ by~\eqref{T:1}, we see that each closed point of $M_F$ is an image of a closed point in $Q_F$. This defines a surjective morphism $h_F$. Now suppose that $h_F(q) = h_F(q') = m$ for two distinct closed points $q, \, q' \in Q$. Let $E$ and $E'$ be two $P$-semistable bundles on $(X,P)$ corresponding to the points $q$ and $q'$, respectively. By the commutativity of the diagram~\eqref{huge_diag}, the closed points $\iota_F(q)$ and $\iota_F(q')$ in $Q_Z$ map to the same point $\tilde{\iota}_F(m) \in M_Z$ under the surjective map $Q_Z \longrightarrow M_Z$ (cf. Lemma~\ref{lem_iamge_of_Q}). Since $g^*E$ and $g^*E'$ are the bundles on $Z$ corresponding to the points $\iota_F(q)$ and $\iota_F(q')$, respectively, by the first part of the proof of \Cref{thm_moduli_classical}, the semistable bundles $g^*E$ and $g^*E'$ are $\mathrm{S}$-equivalent. By \Cref{prop_S_equivariance}, we conclude that $E \sim_{\mathrm{S}} E'$.
\end{proof}

We are now ready to present the main theorem.

\begin{theorem}\label{thm_main_orbifold_moduli}
Under Notation~\ref{not_fixed}, and notation of \Cref{thm_Main_Construction}, suppose that \Cref{ass_non-empty} holds. Then the algebraic stack $\mathcal{M}_{(X,P)}^{\mathrm{ss, S}}(n, \Delta)$, parameterizing $P$-semistable bundles of rank $n$ and determinant $\Delta$, admits a Coarse moduli space $\mathrm{M}_{(X,P)}^{\mathrm{ss}}(n, \Delta)$ that is a disjoint union of the projective integral varieties $M_F$ constructed as the Stein factorization in \Cref{thm_Main_Construction}~\eqref{T:2}, and where $F$ varies over the set $\mathcal{F}$ of mutually non-isomorphic bundles of the form $\oplus_{1 \leq i \leq n-1} L_i$ where each $L_i$ is a line bundle on $(X,P)$ such that $g^*L_i \cong \chi_i \otimes \mathcal{O}_Z$ as $\Gamma$-equivariant line bundles for some character $\chi_i$ of $\Gamma$ over $k$.
\end{theorem}

\begin{proof}
By \Cref{thm_Main_Construction}~\eqref{T:4}, we have a surjective map
\[
\sqcup_{F \in \mathcal{F}} \, h_F \colon \sqcup_{F \in \mathcal{F}} \, Q_F \longrightarrow \mathrm{M}_{(X,P)}^{\mathrm{ss}}(n, \Delta) = \sqcup_{F \in \mathcal{F}} \, M_F,
\]
component wise mapping the closed points of $Q_F$ to $M_F$ such that if two distinct closed points $q, \, q' \in Q_F$ map to the same point in $M_F$, the points represent bundles which are $\mathrm{S}$-equivalent.

We need a morphism $\mathcal{M}_{(X,P)}^{\mathrm{ss, S}}(n,\Delta) \longrightarrow \mathrm{Hom}(-, \mathrm{M}^{\mathrm{ss}}_{(X,P)}(n, \Delta))$ that is initial among the morphisms to $k$-schemes and that induces a bijection on the geometric points. The algebraic stack $\mathcal{M}_{(X,P)}^{\mathrm{ss, S}}(n,\Delta)$ is defined as the quotient of the stack $\mathcal{M}_{(X,P)}^{\mathrm{ss}}(n,\Delta)$, parameterizing $P$-semistable bundles of rank $n$ and determinant $\Delta$, modulo the $\mathrm{S}$-equivalence. So it is enough to construct a morphism $\mathcal{M}_{(X,P)}^{\mathrm{ss}}(n,\Delta) \longrightarrow \mathrm{Hom}(-, \mathrm{M}^{\mathrm{ss}}_{(X,P)}(n, \Delta))$ such that any two $k$-points of $\mathcal{M}_{(X,P)}^{\mathrm{ss}}(n,\Delta)$ map to the same point if and only if the corresponding bundles (seen as families over $\mathrm{\Spec}(k)$) are $\mathrm{S}$-equivalent. This morphism on any $k$-scheme $T$ is constructed as follows. Let $E$ be a usual equivalence class of family of $P$-semistable bundles on $(X,P)$, parameterized by $T$, corresponding to a $T$-point in $\mathcal{M}_{(X,P)}^{\mathrm{ss}}(n,\Delta)$. By the standard approximation techniques, we may assume that $T$ is connected and of finite type. By \Cref{thm_local_families}~\eqref{zar:2}, there is a unique $F \in \mathcal{F}$, a Zariski open covering $T = \cup_i \, T_i$ together with morphisms $\colon T_i \xlongrightarrow{\phi_i} \mathbb{P}_F \xhookrightarrow{\iota_F} \mathbb{P}_Z$ such that
\[
\left( \mathrm{id}_T \times g \right)^*E|_{T_i \times Z} \cong \left( \iota_F \circ \phi_i \times \mathrm{id}_Z \right)^*\mathcal{E}_Z, \text{\rm inducing } E|_{T_i \times (X,P)} \cong \left( \phi_i \times \mathrm{id}_{(X,P)} \right)^* \mathcal{E}_F.
\]
The image if each $\phi_i$ is contained in $Q_F$. Post-composing with the natural map $Q_F \hookrightarrow \sqcup_{F \in \mathcal{F}} \, Q_F \xrightarrow{\sqcup_{F \in \mathcal{F}} \, h_F} \mathrm{M}_{(X,P)}^{\mathrm{ss}}(n, \Delta)$, we get maps $T_i \longrightarrow \mathrm{M}_{(X,P)}^{\mathrm{ss}}(n, \Delta)$. On the other hand, the images of $\iota_F \circ \phi_i$ are contained in $Q_Z$, defining maps $T_i \longrightarrow \mathrm{M}_{Z}^{\mathrm{ss}}(n, g^*\Delta)$. By \cite[Section~8.1, Functoriality]{Hein}, these later maps glue together in the intersection of the $T_i$'s. Thus, the maps $T_i \longrightarrow \mathrm{M}_{(X,P)}^{\mathrm{ss}}(n, \Delta)$ also glue together in the intersections, producing a $T$-point of the scheme $\mathrm{M}^{\mathrm{ss}}_{(X,P)}(n, \Delta))$. This produces a morphism $\mathcal{M}_{(X,P)}^{\mathrm{ss}}(n,\Delta)(T) \longrightarrow \mathrm{M}^{\mathrm{ss}}_{(X,P)}(n, \Delta)(T)$ that is functorial in $T$ and with the required property. We conclude that $\mathrm{M}_{(X,P)}^{\mathrm{ss}}(n, \Delta)$ is a Coarse moduli scheme for the algebraic stack $\mathcal{M}_{(X,P)}^{\mathrm{ss, S}}(n, \Delta)$.
\end{proof}

\begin{remark}
From the above theorem and the universal property of the Coarse moduli space of an algebraic stack, it follows that the construction of the space $\mathrm{M}_{(X,P)}^{\mathrm{ss}}(n, \Delta)$ is independent of the choice of the tamely ramified cover $\tau$ as in \Cref{lem_choice_tame_cover} or the choice of the curve $Z$.
\end{remark}

\appendix

\section{Construction of the Moduli Space in the Classical Context}\label{sec_Falting_Construction}
\subsection{Bounded Families}
First, we establish that the algebraic stack $\mathcal{M}^{\mathrm{ss, S}}_X(n, \Delta)$ is quasi-compact or, equivalently, any family of all semistable bundles on $X$ of rank $n$ and determinant $\Delta$ is bounded. For this, we show that there is a $k$-variety $T$ together with a bundle $\mathcal{F}$ on $T \times X$ such that each semistable bundle of rank $n$ and determinant $\Delta$ is represented by $\mathcal{F}_t$ for some closed point $t \in T$. It will turn out that we can take $T$ to be a certain projective variety ($\mathbb{P}_X$, constructed later) together with a natural universal family $\mathcal{F} = \mathcal{E}_X$.

Fix a line bundle $M$ on $X$ satisfying
\begin{equation}\label{eq_deg_M}
  \mathrm{deg}(M) + \mathrm{deg}(\Delta)/n \geq 2g.  
\end{equation}
Then any semistable bundle $E$ of rank $n$ satisfying $\mathrm{det}(E) \cong \Delta$ sits in an exact sequence
\begin{equation}\label{eq_classical_extn}
  0 \longrightarrow \mathcal{O}_X^{\oplus (n-1)} \otimes M^{-1} \longrightarrow E \longrightarrow \Delta \otimes M^{\otimes (n-1)} \longrightarrow 0  
\end{equation}
of bundles on $X$ --- this is \cite[Proposition~3.1]{Hein} when $n = 2$, and $\Delta$ is the canonical line bundle $\omega_X$, but the same argument extends to general $n$ and $\Delta$; see Proposition~\ref{prop_classical}. Since $\Delta^{-1} \otimes M^{-n}$ is a line bundle of negative degree by the assumption~\eqref{eq_deg_M}, we have
\[
\mathrm{H}^0\left( X, \oplus_{1 \leq i \leq n-1} \Delta^{-1} \otimes M^{-n} \right) = 0.
\]
So \cite[Corollary~4.5]{Lange} is applicable, and hence the functor parameterizing the equivalence classes of non-split extensions of $\Delta \otimes M^{\otimes (n-1)}$ by $\mathcal{O}_X^{\oplus (n-1)}$ is representable by the Grassmannian variety
\begin{equation}\label{eq_parameter_space}
 \mathbb{P}_X \coloneqq \mathbb{P}(V_X) = \mathrm{Grass}(1, V_X^*)
\end{equation}
of one-dimensional subspaces of the dual space $V_X^*$, where
\[
V_X^* \coloneqq \mathrm{Ext}^1\left( \Delta \otimes M^{\otimes (n-1)}, \mathcal{O}_X^{\oplus (n-1)} \otimes M^{-1} \right) \cong \mathrm{H}^1\left(X, \oplus_{1 \leq i \leq n-1} \Delta^{-1} \otimes M^{-n} \right).
\]
Moreover, there is a family $\mathcal{E}_X$ of extensions
\begin{equation}\label{eq_universal_family_on_X}
    0 \rightarrow q_{\mathbb{P}_X}^* \left( M^{-1} \right)^{\oplus (n-1)} \longrightarrow \mathcal{E}_X \longrightarrow p_{\mathbb{P}_X}^* \mathcal{O}_{\mathbb{P}_X}(-1) \otimes  q_{\mathbb{P}_X}^* \left( \Delta \otimes M^{\otimes (n-1)} \right) \rightarrow 0
\end{equation}
on $X \times \mathbb{P}_X$ that is universal for the families of non-split extension, with $p_{\mathbb{P}_X}$ and $q_{\mathbb{P}_X}$ being the usual projection maps. By \cite[Proposition~3.4]{Hein}, this family locally induces any other family of semistable bundles, i.e. for any $k$-scheme $S$ of finite type and a family $\mathcal{F}$ over $X$ of semistable bundles of rank $n$ and determinant $\Delta$, parameterized by $S$, there is a Zariski open covering $S = \cup_i S_i$ together with maps $\phi_i \colon S_i \longrightarrow \mathbb{P}_X$ such that
\[
\mathcal{F}|_{S_i} \cong \left(\phi_i \times \mathrm{id}_X \right)^{*} \mathcal{E}_X.
\]
By~\eqref{eq_classical_extn}, the semistable bundles of rank $n$ and determinant  $\Delta$ are parametrized by $k$-points in $\mathbb{P}_X$. Since not all points of $\mathbb{P}_X$ correspond to semistable bundles, we need to identify the closed subscheme of $\mathbb{P}_X$ corresponding to unstable bundles. This is done by observing that the aforementioned closed subscheme is the base locus of the complete linear system associated with a line bundle on $\mathbb{P}_X$, produced using the determinant of cohomologies and a result of Le Potier. Before proceeding, let us mention Falting's cohomological criteria of semistability.

\begin{proposition}[{\cite[Proposition~2.7]{Hein}}]\label{Prop_Falting_ss}
For two bundles $E, \, F \in \text{\rm Vect}(X)$, if
\[
\mathrm{H}^*(X, E \otimes F) = 0 \hspace{.3cm} \text{\rm i.e.} \hspace{.3cm} \mathrm{H}^0(X, E \otimes F) = \mathrm{H}^1(X, E \otimes F) = 0,
\]
then $E$ and $F$ are semistable. 
\end{proposition}

\subsection{The Determinantal Line Bundle}\label{sec_A_determinantal_line_bundle}
We fix the following notation.
\begin{notation}\label{not_delta}
Set $d \coloneqq \mathrm{deg}(\Delta)$, and $\delta \coloneqq (n, d)$ with the convention that $\delta = n$ if $d = 0$, and $\delta = \mathrm{g.c.d.}(n,d)$, otherwise. Also, set $\rchi \coloneqq d - n (g - 1)$.
\end{notation}

Let us now recall the construction of the determinant of cohomologies; cf. \cite[Section~2]{Seshadri93}. As before, let $S$ be any integral $k$-scheme, and we have the projection maps
\[
S \xlongleftarrow{p_S} X \times S \overset{q_S}\longrightarrow X.
\]
Consider a bundle $\mathcal{E}$ on $S \times X$ of rank $n$ and determinant $\Delta$. For any coherent sheaf $W$ on $X$, the determinant of cohomology is the invertible sheaf
\[
\lambda_{\mathcal{E}}(W) \coloneqq \otimes_{i \geq 0} \, \mathrm{det}\left( \mathrm{R}^ip_{S,*}\left( \mathcal{F} \otimes q_{S}^* W \right) \right)^{(-1)^i}
\]
on $S$. Then \cite[Section~1.2]{Potier} and \cite[Lemma~2.5, Lemma~2.7]{Seshadri93} shows that up to isomorphism, the invertible sheaf $\lambda_{\mathcal{E}}(W)$ on $S$ is independent of the class of $W$ in the Grothendieck group $K(X)$ of coherent sheaves on $X$, and we obtain a group homomorphism
\[
\lambda_{\mathcal{E}} \colon K(X) \longrightarrow \mathrm{Pic}(S)
\]
that is functorial in $S$. By \cite[Lemma~1.2]{Potier}, for any invertible sheaf $N$ on $S$, we have the following isomorphisms
\begin{equation}\label{eq_invariance_in_family}
\lambda_{\mathcal{E}\otimes p_S^*N}(W) \cong \lambda_{\mathcal{E}}(W) \otimes N^{\otimes \chi_X( \mathcal{E}_s \otimes W)}  
\end{equation}
for any closed point $s \in S$. Since $\mathcal{E}_s$ is a bundle of rank $n$ and determinant $\Delta$ for any $s \in S$, the Riemann Roch Theorem shows that the Euler characteristic
\[
\chi( \mathcal{E}_s \otimes W) = n \cdot \mathrm{rk}(W) \left( \mu(\mathcal{E}_s) + \mu(W) - (g - 1) \right)
\]
is independent of the choice of $s$. So, if $\chi( \mathcal{E}_s \otimes W) = 0$ or equivalently, $\mu(\mathcal{E}_s \otimes W) = g - 1$, the determinant of cohomology $\lambda_{\mathcal{E}}(W)$ is well defined in the usual equivalence class of family containing $\mathcal{E}$.

Recall that $d = \mathrm{deg}(\Delta)$. Consider the subgroup $\mathrm{H}(n, d)$ of $K(X)$ generated by the coherent sheaves $W$ satisfying $\chi( \mathcal{E}_s \otimes W) = 0$ for any point $s \in S$. It is noted in \cite{DN} that $\mathrm{H}(n,d)$ is the kernel of the group homomorphism
\[
K(X) \longrightarrow \mathbb{Z}, \, \alpha \mapsto \chi(\alpha \otimes [\mathcal{E}_s])
\]
that is independent of the point $s \in S$, and only depends on the integers $n$ and $d$. By the Riemann Roch Theorem, the class of a coherent sheaf $W$ lies in $\mathrm{H}(n,d)$ precisely when $\mathrm{rk}(W) = l \cdot n/\delta$ and $\mathrm{deg}(W)= -l \cdot \rchi/\delta = -l \cdot \frac{d - n(g-1)}{\delta}$ for some $l \geq 1$. The isomorphism~\eqref{eq_invariance_in_family} shows that the group homomorphism
\begin{equation}\label{eq_lambda_definition}
\lambda_{\mathcal{E}} \colon \mathrm{H}(n,d) \longrightarrow \mathrm{Pic}(S)
\end{equation}
is independent of the choice of the family $\mathcal{E}$ in its usual equivalence class, and is functorial in $S$. Moreover, \cite[Lemma~2.6]{Seshadri93} shows that for any line bundle $L_0$ on $X$ of degree $0$, and a coherent sheaf $W$ on $X$ whose class in $K(X)$ is in $\mathrm{H}(n,d)$, we have an isomorphism
\[
\lambda_{\mathcal{E}}(W \otimes L_0) \cong \lambda_{\mathcal{E}}(W).
\]

\begin{definition}\label{def_Theta_line_bundle}
Let $S$ be an integral $k$-scheme. Let $\mathcal{E}$ be a family of bundles on $X$ of rank $n$ and determinant $\Delta$, parametrized by $S$. Let $\delta$ be as before, and $R$ be a positive integer. The \textit{generalized} $\Theta$-\textit{line bundle} $\mathcal{O}_S(R \cdot \Theta)$ is defined to be the invertible sheaf
\[
\mathcal{O}_S(R \cdot \Theta) \coloneqq \lambda_{\mathcal{E}}(W)^{-1}
\]
for any bundle $W$ on $X$ of rank $R \cdot n/\delta$ and of degree $- R \cdot \rchi/\delta$.
\end{definition}

It is well known that the generalized $\Theta$-line bundles have a multiplicative structure, namely $\mathcal{O}_S(R \cdot \Theta) \cong \mathcal{O}_S(\Theta)^{\otimes R}$ where we write $\mathcal{O}_S(1 \cdot \Theta)$ as $\mathcal{O}_S(\Theta)$; see \cite[Section~4.3]{Hein}. We will show that the unstable locus in $\mathbb{P}_X$ is the base locus of the generalized $\Theta$-line bundle $\mathcal{O}_{\mathbb{P}_X}(R \cdot \Theta)$ for a large $R$.

Next, we recall that for any positive integer $R$ and any bundle $W$ on $X$ of rank $R \cdot n/\delta$ and of degree $-R \cdot \rchi/\delta$, there is a distinguished section $s_{\mathcal{E},W}$ of the generalized $\Theta$-line bundle $\mathcal{O}_S(R \cdot \Theta) \cong \lambda_{\mathcal{E}}(W)^{-1}$, defined up to a non-zero scalar in $k$, and the vanishing locus $\theta_{\mathcal{E},W}$ of $s_{\mathcal{E}, W}$ has a description in terms of non-vanishing of certain cohomologies. Consider a locally free resolution
\[
0 \longrightarrow V_1 \overset{\alpha}\longrightarrow V_0 \longrightarrow \mathcal{E} \otimes q_S^* W \longrightarrow 0
\]
such that $p_{S,*}V_i = 0$; this can be done by taking $V_0 = p_S^* p_{S_*}\left( ( \mathcal{E} \otimes q_S^* W )(l)\right) \otimes \mathcal{O}_{X_S}(-l)$ for $l > > 0$, and then taking $V_1$ as the kernel of the natural surjection $V_0 \longrightarrow \mathcal{E} \otimes q_S^* W$. The above exact complex produces the following long exact sequence of higher direct image sheaves on S:
\[
0 \rightarrow p_{S,*}(\mathcal{E} \otimes q_S^* W) \longrightarrow \mathrm{R}^1p_{S,*} (V_1) \overset{\mathrm{R}^1\alpha}\longrightarrow \mathrm{R}^1p_{S,*}(V_0) \longrightarrow \mathrm{R}^1p_{S,*}(\mathcal{E} \otimes q_S^* W) \rightarrow 0.
\]
Since $\chi(\mathcal{E}_s \otimes W) = 0$ for any closed point $s \in S$, the sheaves $\mathrm{R}^1p_{S,*}(V_i)$ are locally free sheaves on $S$ of the same rank, $r$ say. Then the element in the Grothendieck group $K_0(S)$ of bundles on $S$, defined by $\mathrm{R}p_{S,*}(\mathcal{E} \otimes q_S^*W)$ is represented by the complex
\[
0 \rightarrow \mathrm{R}^1p_{S,*} (V_1) \overset{\mathrm{R}^1\alpha}\longrightarrow \mathrm{R}^1p_{S,*}(V_0) \rightarrow 0.
\]
The determinant
\[
\mathrm{det} \, \mathrm{R}^1\alpha \coloneqq \wedge^{r} \, \mathrm{R}^1 \alpha \in \mathrm{Hom}(\wedge^r \mathrm{R}^1p_{S,*}(V_1), \wedge^r \mathrm{R}^1p_{S,*}(V_0)) = \Gamma(S, \mathcal{O}_S(R \cdot \Theta))
\]
defines a global section of $\lambda_{\mathcal{E}}(W)^{-1} \cong \mathcal{O}_S(R \cdot \Theta)$, independent of the chosen resolution, determined up to a non-zero element of $\mathrm{H}^0(X, \mathcal{O}_X) = k$; cf. \cite[\href{https://stacks.math.columbia.edu/tag/0FJI}{Section 0FJI}]{SP}. We set our distinguished section associated to $W$ to be
\[
s_{\mathcal{E},W} \coloneqq \mathrm{det} \, \mathrm{R}^1\alpha.
\]
Note that $\mathrm{R}^1\alpha$ is surjective at $s$ if and only if $\mathrm{R}^1p_{S,*}(\mathcal{E} \otimes q_S^* W)$ vanishes at $s$ if and only if
\[
\mathrm{H}^0(X, \mathcal{E}_s \otimes W) = \mathrm{H}^1(X, \mathcal{E}_s \otimes W) = 0.
\]
Since $\mathrm{R}^1\alpha \otimes k(s)$ is a morphism of vector spaces of the same dimension, $\mathrm{R}^1\alpha$ is surjective at $s$ if and only if $\mathrm{det} \, \mathrm{R}^1\alpha \neq 0$ at $s$. This shows that the vanishing locus of the section $s_{\mathcal{E},W}$ is
\begin{equation}\label{eq_classical_vanishing_locus}
\theta_{\mathcal{E},W} = \{s \in S(k) \, | \, \mathrm{H}^*(X, \mathcal{E}_s \otimes W) \neq 0\}.
\end{equation}
One immediate consequence of Falting's criteria (Proposition~\ref{Prop_Falting_ss}) is that if $W$ is unstable, $\theta_{\mathcal{F},W}$ is the whole space $S$.

\subsection{Construction}\label{sec_ap_classical_contruction}
Let us apply the above general results in our context with $S = \mathbb{P}_X$ and the universal family $\mathcal{E} = \mathcal{E}_X$; see~\eqref{eq_parameter_space} and~\eqref{eq_universal_family_on_X}. Let $\delta$ and $\chi$ be as in Notation~\ref{not_delta}. For any positive integer $R$, and a bundle $W$ of rank $R \cdot n/\delta$ and of degree $-R \cdot \chi/\delta$, we have the generalized $\Theta$-line bundle $\mathcal{O}_{\mathbb{P}_X}(R \cdot \Theta)$. To each such bundle $W$, there is an associated global section $s_W \coloneqq s_{\mathcal{E}_X,W} \in \Gamma(\mathbb{P}, \mathcal{O}_{\mathbb{P}_X}(R \cdot \Theta))$ with vanishing locus
$$\theta_{W} \coloneqq \theta_{\mathcal{E}_X,W} = \{[E] \in \mathbb{P}_X \, | \, \mathrm{H}^0(X, E \otimes W) \neq 0 \}.$$
If $W$ as above is not a semistable bundle, by \cite[Proposition~2.7]{Hein}, the section $s_W$ is the zero section. On the other hand, if $W$ above is semistable, the vanishing locus $\theta_W$ is a Cartier divisor on $\mathbb{P}_X$; cf. \cite[Lemma~4.2]{Hein}.

\emph{For the rest of this section, fix an integer $R > \frac{n^2}{4}(g-1)$.}

Le Potier's result \cite[Theorem~2.4]{Potier} shows that for any semistable bundle $E$ on $X$ of rank $n$ and determinant $\Delta$, there is a (necessarily semistable) bundle $W$ of rank $R \cdot n/\delta$ and determinant $\mathcal{O}_X(-R \cdot \rchi/\delta)$ such that
\[
\mathrm{H}^*(X, E \otimes W) = 0.
\]
This shows that for any bundle $E$ corresponding to a closed point $e \in \mathbb{P}_X$, the following statements are equivalent.
\begin{enumerate}
\item $E$ is not semistable;
\item for all $W$ of rank $R \cdot n/\delta$ and of degree $-R \cdot \rchi/\delta$, the cohomology groups $\mathrm{H}^*(X, E \otimes W)$ are non-zero;
\item for all $W$ of rank $R \cdot n/\delta$ and of degree $-R \cdot \rchi/\delta$, the point $s$ is in the vanishing locus $\theta_W$.
\end{enumerate}
We conclude the following.
\begin{proposition}\label{prop_bl}
Under the above notation, the closed subscheme
\[
B = \{ [E] \in \mathbb{P}_X \, | \, E \text{ is not semistable} \}
\]
of $\mathbb{P}_X$ is the base locus of the complete linear system of the generalized $\Theta$-line bundle $\mathcal{O}_{\mathbb{P}_X}(R \cdot \Theta)$, spanned by the theta divisors. As this linear system is finite-dimensional, we can write $B$ as a finite intersection
\[
B = \cap_{0 \leq i \leq N} \theta_{W_i}
\]
for a uniquely determined $N$ and such that the sections $s_{W_i}$ generate the line bundle $\mathcal{O}_{\mathbb{P}_X}(R \cdot \Theta)$ on the complement of the base locus $B$.
\end{proposition}

Set $Q$ as the open subscheme of $\mathbb{P}_X$, which is the complement of $B$. Then the sections $s_{W_i}$ generate $\mathcal{O}_{\mathbb{P}_X}(R \cdot \Theta)$ over $Q$. Let $\pi \colon \tilde{Q} \longrightarrow \mathbb{P}_X$ be the blow up of $\mathbb{P}_X$ at the closed subscheme $B$. Since $\mathbb{P}_X$ is reduced and $B$ is a proper closed subscheme of $\mathbb{P}_X$, the projective variety $\tilde{Q}$ is integral. Also, $Q$ is an open subscheme of $\tilde{Q}$. Using the above proposition, we obtain the following commutative diagram determined by the complete linear system associated to $\mathcal{O}_{\mathbb{P}_X}(R \cdot \Theta)$.
\begin{equation}\label{diag_classical_linear_system}
\begin{tikzcd}
\mathbb{P}_X \arrow[r, hookleftarrow] & Q = \mathbb{P}_X - B \arrow[d, hookrightarrow] \arrow[r, "\psi"] & \mathbb{P}^N \\
 & \tilde{Q} = \text{\rm Bl}_{B}{\mathbb{P}_X} \arrow[lu, "\pi"] \arrow[ru, swap, "\widetilde{\psi}"] & 
\end{tikzcd}
\end{equation}
By Langton's Theorem~\cite[Theorem~6.4]{Hein}, the map $\psi \colon Q \longrightarrow \mathbb{P}^N$ is a proper map whose image in $\mathbb{P}^N$ coincides with that of the map $\widetilde{\psi}$. Consider the Stein factorization (cf. \cite[\href{https://stacks.math.columbia.edu/tag/03H0}{Theorem 03H0}]{SP}) of the map $\widetilde{\psi}$:
\[
\widetilde{\psi} \colon \tilde{Q} \overset{g_X}\longrightarrow \mathrm{M}^{\mathrm{ss}}_X(n,\Delta) \overset{f_X}\longrightarrow \mathbb{P}^N.
\]
where $f_X$ is a finite map of integral projective varieties, and $g_X$ has connected fibers. Summarizing the above, we have the following commutative diagram.
\begin{equation}
\begin{tikzcd}
Q \arrow[r, "\psi"] \arrow[d, hookrightarrow] & \mathbb{P}^N \\
\tilde{Q} \arrow[ur, "\tilde{\psi}"] \arrow[r, "g_X"] & \mathrm{M}^{\mathrm{ss}}_X(n,\Delta) \arrow[u, "f_X"]
\end{tikzcd}
\end{equation}
In the following, we observe that the closed points of $\mathrm{M}^{\mathrm{ss}}_X(n,\Delta)$ are the images of closed points of $Q$.

\begin{lemma}\label{lem_iamge_of_Q}
Under the above notation, each closed point of $\mathrm{M}^{\mathrm{ss}}_X(n,\Delta)$ is the image of a closed point of $Q$ under the map $g_X$.
\end{lemma}

\begin{proof}
Let $m$ be a closed point of $\mathrm{M}^{\text{\rm ss}}_X(n, \Delta)$. Applying Mumford's Rigidity Lemma~\cite[Proposition~6.1, Chapter 6, page 115]{MumfordGIT} to the following commutative diagram
\begin{equation*}
\begin{tikzcd}
g_X^{-1}(m) \arrow[rr, "g_X"] \arrow[dr, swap, "\widetilde{\psi}|_{\text{\rm rest}}"] & & \{m\} \arrow[dl, "f_X"] \\
& \{f_X(m)\},
\end{tikzcd}
\end{equation*}
we obtain a section $\sigma (f_X(m)) = m$. Since the images of $\psi$ and $\widetilde{\psi}$ coincide, there is a closed point $q \in Q$ mapping to $f_X(m)$ via $\widetilde{\psi}|_{\mathrm{rest}}$. Then
\[g_X(q) =  \sigma \circ \widetilde{\psi}|_{\mathrm{rest}}(q) = \sigma(f_X(m)) = m. \]
\end{proof}

As a consequence of the above lemma, we see that the closed points of $\mathrm{M}^{\mathrm{ss}}_X(n, \Delta)$ being images of closed points of $Q$, represent semistable bundles. We need to establish that the closed points of $\mathrm{M}^{\mathrm{ss}}_X(n, \Delta)$ actually represent $\mathrm{S}$-equivalence classes of bundles.

First, we need some crucial observations. Let $C$ be a smooth projective connected $k$-curve. Let $E$ be a family of bundles on $X$ of rank $n$ and determinant $\Delta$, parameterized by $C$. Since the semistable locus $C'$ of the family $E$ is an open sub-curve of $C$, the coherent sheaf $\mathrm{R}^1p_{C,*}\left( E \otimes q_C^*W \right)$ is a torsion sheaf on $C$ of finite length that is supported in the complement of $C'$. Using Langton's Theorem~\cite[Theorem~6.4]{Hein}, we can perform elementary transformations and replace $E$ by another family $E'$ of rank $n$ and determinant $\Delta$ such that the fiber over every point $c \in C$ is semistable, and $E_{c'} \cong E'_{c'}$ for all points $c' \in C'$. Denote the generalized $\Theta$-line bundle $\mathcal{O}_C(R \cdot \Theta)$ on $C$ associated to $E'$ by $\mathcal{O}_C(\Theta_C)$. The following well-known results establish the connection between the degree of $\mathcal{O}_C(\Theta_C)$ and the $\mathrm{S}$-equivalence in the family $E'$. 
\begin{proposition}\label{prop_positivity}
Under the above notation, we have the following.
\begin{enumerate}
\item The line bundle $\mathcal{O}_C(\Theta_C)$ is base point free. In particular, $\mathrm{deg}\left(\mathcal{O}_C(\Theta_C)\right) \geq 0$.\label{pos:1}
\item $\mathrm{deg}\left(\mathcal{O}_C(\Theta_C)\right) = 0$ if and only if for any points $c, \, c' \in C$, the bundles $E'_c$ and $E'_{c'}$ are $\mathrm{S}$-equivalent.\label{pos:2}
\end{enumerate}
\end{proposition}

\begin{proof}
The above statements are \cite[Proposition~7.4 and Theorem~7.5]{Hein} when $n = 2$ and $\Delta = \omega_X$. The same arguments holds under our hypothesis, see \cite[Section~8.3(5)]{Hein}.
\end{proof}

Let us note the following important result.
\begin{lemma}\label{lem_curve_contraction_equivalences}
Suppose that $C$ is a smooth projective sub-curve of $\mathbb{P}_X$. Then the following conditions are equivalent.
\begin{enumerate}
\item $C \cap Q$ gets contracted under $\psi$;\label{contra:1}
\item the total transform $\pi^{-1}(C)$ under the blowing-up map $\pi$ gets contracted under $\widetilde{\psi}$;\label{contra:2}
\item the image of $\pi^{-1}(C)$ under $\widetilde{\psi}$ is a point $a \in \mathbb{P}^N$, and there exists a hypersurface $H_a$ in $\mathbb{P}^N$ not passing through the point $a$ such that $\widetilde{\psi}^* H_a$ is the exceptional divisor $\pi^{-1}(B)$;\label{contra:3}
\item the intersection number $\pi^{-1}(C) \cdot \pi^{-1}(B)$ is zero;\label{contra:4}
\item $C \cdot B = 0$.\label{contra:5}
\end{enumerate}
\end{lemma}

\begin{proof}
Since $\psi$ is proper map, conditions~\eqref{contra:1} and \eqref{contra:2} are equivalent. If the image of $\pi^{-1}(C)$ under $\widetilde{\psi}$ is a point $a \in \mathbb{P}^N$, we can choose a hypersurface $H_a$ not passing through $a$; as $\mathcal{O}_{\mathbb{P}^N}(H_a) \cong \mathcal{O}_{\mathbb{P}^N}(1)$, we have $\widetilde{\psi}^* H_a = \pi^{-1}(B)$.The converse is immediate, proving \eqref{contra:2}$\Leftrightarrow$\eqref{contra:3}. The equivalence of \eqref{contra:3} and \eqref{contra:4} is obvious. Finally, using the projection formula under the proper morphism $\pi$ and noting that $C = \pi (\pi^{-1}(C))$, the last two statements are equivalent.
\end{proof}

Now we are ready to prove that $\text{\rm M}^{\mathrm{ss}}_X(n,\Delta)$ is the desired moduli space.

\begin{theorem}\label{thm_moduli_classical}
The integral projective variety $\text{\rm M}^{\mathrm{ss}}_X(n,\Delta)$ is the Coarse moduli space for the functor $\mathcal{M}^{\mathrm{ss, S}}_X(n,\Delta)$.
\end{theorem}

\begin{proof}
By Lemma~\ref{lem_iamge_of_Q}, the closed points of $\mathrm{M}^{\mathrm{ss}}_X(n,\Delta)$ are the images of the closed points of $Q$ under $g_X$. This defines a surjective morphism $Q \longrightarrow \mathrm{M}^{\mathrm{ss}}_X(n,\Delta)$. We claim that if two distinct closed points $q, \, q' \in Q$ map to the same point $m \in \mathrm{M}^{\mathrm{ss}}_X(n,\Delta)$, then the semistable bundles represented by the points $q$ and $q'$ are $\mathrm{S}$-equivalent. To see this, consider a smooth projective connected sub-curve $C$ in $\mathbb{P}_X$, passing through $q$ and $q'$. Restriction of the universal bundle $\mathcal{E}_X$ on $C \times X$ produces a family of bundles on $X$ of rank $n$ and determinant $\Delta$ whose semistable locus is $C \cap Q$. We can perform elementary transformations to produce a family $E$ on $C \times X$ that is semistable at every point of $C$. Since $f_X^{-1}\{m\}$ is a closed connected sub-variety of $\tilde{Q}$, containing the distinct points $q$ and $q'$ of the connected curve $C \cap Q$, the curve $C \cap Q$ gets mapped to the point $f_X(m)$ under $\psi$. By Lemma~\ref{lem_curve_contraction_equivalences}, this is equivalent to $C \cdot B = 0$. Equivalently, the degree of the generalized $\Theta$-line bundle $\mathcal{O}_C(\Theta_C)$, defined by the family $E$, is zero on $C$ by our construction. By Proposition~\ref{prop_positivity}, the bundles $E_q$ and $E_{q'}$ are $\mathrm{S}$-equivalent.

We want to show the existence of a morphism $\mathcal{M}_X^{\mathrm{ss, S}}(n,\Delta) \longrightarrow \mathrm{Hom}(-, \mathrm{M}^{\mathrm{ss}}_X(n, \Delta))$ that is initial among the morphisms to $k$-schemes. Recall that the algebraic stack $\mathcal{M}_X^{\mathrm{ss, S}}(n,\Delta)$ was defined using the $\mathrm{S}$-equivalence as the quotient of the stack $\mathcal{M}_X^{\mathrm{ss}}(n,\Delta)$ which parameterize semistable bundles of rank $n$ and determinant $\Delta$. So it is enough to construct a morphism $\mathcal{M}_X^{\mathrm{ss}}(n,\Delta) \longrightarrow \mathrm{Hom}(-, \mathrm{M}^{\mathrm{ss}}_X(n, \Delta))$ such that any two $k$-points of $\mathcal{M}_X^{\mathrm{ss}}(n,\Delta)$ map to the same point if and only if the corresponding bundles (seen as families over $\mathrm{\Spec}(k)$) are $\mathrm{S}$-equivalent.

Let $S$ be any $k$-scheme of finite type. Consider any equivalence class of a family $\mathcal{E}$ of semistable bundles on $X$ of rank $n$ and determinant $\Delta$, parameterized by $S$, under the usual equivalence. By the local universal property of such families (cf. \Cref{thm_local_families}), there is a Zariski open covering $S = \cup_i S_i$ together with morphisms $\phi_i \colon S_i \longrightarrow \mathbb{P}_X$ such that $\mathcal{E}|_{S_i} \cong (\phi_i \times \mathrm{id})^*\mathcal{E}_X$. Each map $\phi_i$ has image in $Q$, and composition with the surjective map $Q \longrightarrow \mathrm{M}^{\mathrm{ss}}_X(n,\Delta)$, we obtain maps $S_i \longrightarrow \mathrm{M}^{\mathrm{ss}}_X(n,\Delta)$. By \cite[Section~8.1, Functoriality]{Hein}, these maps glue together, and we obtain the desired map that is functorial in $S$, and such that two closed points $s, \, s' \in S$ are mapped to the same closed point $m \in \mathrm{M}^{\mathrm{ss}}_X(n, \Delta)$ if and only if $\mathcal{E}_s \sim_{\mathrm{S}} \mathcal{E}_{s'}$. This shows that $\mathrm{M}^{\mathrm{ss}}_X(n,\Delta)$ is the Coarse moduli space for the functor $\mathcal{M}^{\mathrm{ss, S}}_X(n,\Delta)$.
\end{proof}

\section{Extension Classes}
One of the key steps in our construction of the moduli space comes from the observation that any $P$-semistable bundle of a fixed rank and determinant on a connected orbifold curve $(X,P)$ can be represented as an extension of suitable bundles on $(X,P)$ (cf. Proposition~\ref{prop_main_bounded}). This leads us to to study extensions of a bundle $E_2$ by a bundle $E_1$ over $(X,P)$, i.e. short exact sequences of the form
\[ 0 \longrightarrow E_1 \longrightarrow E \longrightarrow E_2 \longrightarrow 0\]
in $\text{\rm Vect}(X,P)$. A concise treatment of extensions of coherent sheaves over a flat Noetherian scheme over a Noetherian base scheme is given in \cite{Lange}. We will see that these results also hold in the case of orbifold curves.

We note some generalities needed for our purpose. Let $Y$ be a Noetherian $k$-scheme. Consider a finite type flat proper $k$-morphism $f \colon \mathfrak{Y} \longrightarrow Y$ where $\mathfrak{Y}$ is an integral smooth proper separated Deligne Mumford stack of finite type over $k$. In whatever follows, $S$ denote a Noetherian $k$-scheme with structure morphism $S \overset{g}\longrightarrow \text{\rm Spec}(k)$, and we write
\[
S \overset{p_S}\longleftarrow S \times \mathfrak{Y} \overset{q_S}\longrightarrow \mathfrak{Y}
\]
for the projection maps. For any bundle $E_2$ on $\mathfrak{Y}$, and any $i \geq 0$, we have the functor
\[
\sheafext^i_f(E_2,-) \colon \text{\rm Coh}_{\mathfrak{Y}} \longrightarrow \text{\rm Coh}_{Y}
\]
where $\sheafext^i_f(E_2 , E_1) \coloneqq \mathrm{R}^i \left( f_* \sheafhom_{\mathcal{O}_{\mathfrak{Y}}}(E_2 , -) \right)(E_1)$. Since $E_2$ is a bundle, the functor
\[
\sheafhom_{\mathcal{O}_{\mathfrak{Y}}}(E_2 , -) \cong E_2^\vee \otimes_{\mathcal{O}_{\mathfrak{Y}}} -
\]
is an exact functor. Then for any $i \geq 0$, we have
\[
\mathrm{R}^i \left( f_* \sheafhom_{\mathcal{O}_{\mathfrak{Y}}}(E_2 , -) \right) \cong \mathrm{R}^i f_* \circ \left( E_2^\vee \otimes - \right).
\]
In particular, if $Y = \text{\rm Spec}(k)$, and $f$ is the structure morphism, then $\sheafext^i_f(E_2 , E_1) \cong \mathrm{H}^i(\mathfrak{Y}, E_2^\vee \otimes E_1) \otimes_k \mathcal{O}_k$ is the bundle on $\text{\rm Spec}(k)$ associated to the finite dimensional $k$-vector space $\mathrm{H}^i(\mathfrak{Y}, E_2^\vee \otimes E_1)$. In this case, the base change map
\begin{eqnarray}\label{eq_base_change}
\sigma^i(g) \colon g^* \sheafext^i_f(E_2 , E_1) \longrightarrow \sheafext^i_{p_S}(q_S^* E_2 , q_S^* E_1)
\end{eqnarray}
is an isomorphism. Also note that since $\text{\rm Coh}_{\mathfrak{Y}}$ is a Grothendieck abelian category, by \cite[\href{https://stacks.math.columbia.edu/tag/06XP}{Section 06XP}]{SP}, we have the following:
\[
\mathrm{H}^0 \left( \text{\rm Spec}(k), \sheafext^0_f(E_2 , E_1) \right) \cong \mathrm{H}^0(\mathfrak{Y}, E_2^\vee \otimes E_1) \cong \text{\rm Hom}(E_2 , E_1),
\]
and $\mathrm{H}^0\left( \text{\rm Spec}(k), \sheafext^1_f(E_2 , E_1) \right) \cong \text{\rm Ext}^1(E_2 , E_1)$ is the set of isomorphism classes of extensions of $E_2$ by $E_1$.

Now consider the case where $\mathfrak{Y} = (X,P)$ is a connected orbifold curve. Let $E_1, \, E_2 \in \text{\rm Vect}_{(X,P)}$. Write $f \colon (X,P) \longrightarrow \text{\rm Spec}(k)$ for the structure morphism. We have noted that $\sheafext^i_f(E_2, E_1)$ is the locally free sheaf associated to the vector space $\mathrm{H}^i((X,P), E_2^\vee \otimes E_1)$ on $\text{\rm Spec}(k)$, and the formation of $\sheafext^i_f(E_2, E_1)$ commutes with arbitrary base change. Consider the contravariant functor
\[
\text{\rm PE} \colon \text{Noetherian } k\text{-schemes} \longrightarrow \text{\rm Sets}
\]
that to a Noetherian $k$-scheme $g \colon S \longrightarrow \text{\rm Spec}(k)$, associates the set $\text{\rm PE}(S)$ of invertible quotients of the locally free sheaf $\sheafext^1_{p_S}(q_S^*E_2, q_S^* E_1)^\vee$. By~\eqref{eq_base_change}, and since $g^*\mathcal{O}_{\text{\rm Spec}(k)} \cong \mathcal{O}_S$, we have the following isomorphisms:
\[
\sheafext^1_{p_S}(q_S^*E_2, q_S^*E_1) \cong g^* \sheafext^1_f(E_2 , E_1) \cong \mathrm{H}^1((X,P)), E_2^\vee \otimes E_1) \otimes_k \mathcal{O}_S.
\]
Hence the functor $\text{\rm PE}$ is representable by the projective space bundle $\mathbb{P}(\sheafext^1_f(E_2,E_1)^\vee)$ on $\text{\rm Spec}(k)$. In other words, the functor $\text{\rm PE}$ is representable by the projective space $\mathbb{P}(V)$ where $V = \mathrm{H}^1((X,P)), E_2^\vee \otimes E_1)^*$, and the $k$-points of $\mathbb{P}(V)$ correspond to the one-dimensional subspaces in $V^* \cong \mathrm{H}^1((X,P)), E_2^\vee \otimes E_1)$.

The universal element of $\text{\rm PE}(\mathbb{P}(V))$ is given as follows. Set
\[
\mathcal{H} \coloneqq \sheafext^1_f(E_2, E_1) \cong \mathrm{H}^1((X,P), E_2^\vee \otimes E_1) \otimes_k \mathcal{O}_k,
\]
which is the trivial bundle on $\text{\rm Spec}(k)$ associated to the $k$-vector space $\mathrm{H}^1((X,P), E_2^\vee \otimes E_1)$. By the universal property of the projective space $\mathbb{P}(V)$, we have the isomorphism:
\[
\mathcal{H}^\vee \cong f_{\mathbb{P}(V), *} \mathcal{O}_{\mathbb{P}(V)}(1)
\]
where $f_{\mathbb{P}(V)}$ denote the structure morphism of $\mathbb{P}(V)$. Then the image of the identity map $\text{\rm id}_{\mathcal{H}}$ under the canonical isomorphisms
\begin{eqnarray*}
\text{\rm End}(V) \cong \mathrm{H}^0 \left( \text{\rm Spec}(k), \sheafend(\mathcal{H}) \right) \cong \mathrm{H}^0 \left( \text{\rm Spec}(k), \mathcal{H} \otimes \mathcal{H}^\vee \right) \\
\cong \mathrm{H}^0\left( \text{\rm Spec}(k), \mathcal{H} \otimes f_{\mathbb{P}(V), *} \mathcal{O}_{\mathbb{P}(V)}(1) \right) \cong \mathrm{H}^0\left( \text{\rm Spec}(k), f_{\mathbb{P}(V), *} \left( f_{\mathbb{P}(V)}^* \mathcal{H} \otimes \mathcal{O}_{\mathbb{P}(V)}(1)\right) \right) \\
\cong \mathrm{H}^0 \left( \mathbb{P}(V), f_{\mathbb{P}(V)}^* \mathcal{H} \otimes \mathcal{O}_{\mathbb{P}(V)}(1) \right) \cong \mathrm{H}^0 \left( \mathbb{P}(V), \sheafext^1_{p_{\mathbb{P}(V)}}(q_{\mathbb{P}(V)}^* E_2, q_{\mathbb{P}(V)}^* E_1 ) \otimes \mathcal{O}_{\mathbb{P}(V)}(1) \right)
\end{eqnarray*}
is a non-vanishing section, and defines the universal quotient
$$\sheafext^1_{p_{\mathbb{P}(V)}}(q_{\mathbb{P}(V)}^* E_2, q_{\mathbb{P}(V)}^* E_1)^\vee \longrightarrow \mathcal{O}_{\mathbb{P}(V)}(1) \longrightarrow 1.$$

We also see that $\text{\rm PE}(S)$ is the set of non-vanishing sections of $\sheafext^1_{p_S}(q_S^* E_2, q_S^* E_1 \otimes p_S^* N)$ with arbitrary $N \in \text{\rm Pic}_S$, modulo the canonical action of $\mathrm{H}^0\left( S, \mathcal{O}_S^\times \right)$. In particular, for a reduced Noetherian $k$-scheme $S$, it follows that $\text{\rm PE}(S)$ is the set of all families of nowhere splitting extensions $(e_s)_{s \in S}$ of $q_S^*E_2$ by $q_S^*E_1 \otimes p_S^*N$ for arbitrary $N \in \text{\rm Pic}(S)$ modulo the the canonical operation of $\mathrm{H}^0(S, \mathcal{O}_S^\times)$; here a family $(e_s)_{s \in S}$ is defined to be a collection of elements $e_s$ which are extensions of $E_2$ by $E_1$ on $\mathfrak{Y}$, parameterized by $S$, such that for each closed point $s \in S$, there is a Zariski open subset $U$ of $S$ with $e_t$ and $e_s$ being isomorphic extensions for each $t \in U$. Further, the family $(e_s)_{s \in S}$ is nowhere splitting means that $e_s$ is not the trivial extension for any $s \in S$. Finally, using similar arguments from \cite[Corollary~4.4, Corollary~4.5]{Lange}, we conclude the following.

\begin{proposition}\label{prop_universal_general}
Let $(X,P)$ be a connected orbifold curve, $E_1, \, E_2 \in \text{\rm Vect}_{(X,P)}$. Then there is a family of extensions $(e_t)_{t\in \mathbb{P}(V)}$ of $q_{\mathbb{P}(V)}^*E_2$ by $q_{\mathbb{P}(V)}^*E_1 \otimes \mathcal{O}_{\mathbb{P}(V)}(1)$ over the projective space $\mathbb{P}(V)$, the Grassmannian variety of one  dimensional sub-spaces of $V^* = \mathrm{H}^1((X,P), E_2^\vee \otimes E_1)$, that is universal in the category of reduced Noetherian $k$-schemes for the classes of families of nowhere splitting extensions of $q_S^*E_2$ by $q_S^*E_1 \otimes p_S^* N$ for arbitrary $N \in \text{\rm Pic}_S$ modulo the canonical action of $\mathrm{H}^0(S, \mathcal{O}_S^\times)$.

Further, if also $\mathrm{H}^0((X,P), E_2^\vee \otimes E_1) = 0$, there is an extension
\begin{equation}\label{eq_universal_ext}
0 \rightarrow q_{\mathbb{P}(V)}^*E_1 \otimes p_{\mathbb{P}(V)}^* \mathcal{O}_{\mathbb{P}(V)}(1) \rightarrow \mathcal{E}_{V} \rightarrow q_{\mathbb{P}(V)}^* E_2 \rightarrow 0 \hspace{1em} (e_V)
\end{equation}
on the smooth proper Deligne Mumford stack $\mathbb{P}(V) \times (X,P)$, which is universal in the category of Noetherian $k$-schemes for the classes of extensions of $q_S^* E_2$ by $q_S^* E_1 \otimes p_S^* N$ on $S \times (X,P)$ with arbitrary $N \in \text{\rm Pic}_S$, which split nowhere over $S$, modulo the canonical operation of $\mathrm{H}^0(S, \mathcal{O}_S^\times)$.
\end{proposition}

Next, we provide a direct construction of the universal family $\mathcal{E}_V$ in~\eqref{eq_universal_ext} using a K\"{u}nneth formula: for any coherent sheaves of modules $F$ on $\mathbb{P}(V)$, and $E$ on $(X,P)$, we have the isomorphism
\begin{equation}\label{eq_Kun}
\mathrm{H}^m(\mathbb{P}(V) \times (X,P), F \boxtimes E) \cong \oplus_{i + j = m} \mathrm{H}^i(\mathbb{P}(V), F) \otimes \mathrm{H}^j((X,P), E)
\end{equation}
for any $m \geq 0$, and $F \boxtimes E \coloneqq p_{\mathbb{P}(V)}^* F \otimes q_{\mathbb{P}(V)}^* E$.

We first argue that the above isomorphism holds. For any \'{e}tale atlas $U$ of $(X,P)$, we have an \'{e}tale atlas $\mathbb{P}(V) \times U$ for $\mathbb{P}(V) \times (X,P)$. We have the spectral sequences:
\begin{eqnarray*}
E_1^{a,b} = \mathrm{H}^a(U_b, E_b) \Rightarrow \mathrm{H}^{a+b}((X,P), E), \\
E_1^{a,b} = \mathrm{H}^a((\mathbb{P}(V) \times U)_b, (F \boxtimes E)_b) \Rightarrow \mathrm{H}^{a+b}(\mathbb{P}(V) \times (X,P), F \boxtimes E)
\end{eqnarray*}
where $U_b$ (respectively, $(\mathbb{P}(V) \times U)_b$) is the $(b+1)$-fold product of $U$ (respectively, $\mathbb{P}(V) \times U$) over $(X,P)$ (respectively, $\mathbb{P}(V) \times (X,P)$), and $E_{b}$ (respectively, $(F \boxtimes E)_b)$) being the pullback of $E$ (respectively, $F \boxtimes E$) on $U_b$ (respectively, $(\mathbb{P}(V) \times U)_b$), as in Section~\Cref{sec_descent}. Then
\[
(\mathbb{P}(V) \times U)_b \cong \mathbb{P}(V) \times U_b \text{ and }
\]
\[
(F \boxtimes E)_b  \cong F \boxtimes E_{b} \text{ on } \mathbb{P}(V) \times U_b.
\]
Then the isomorphism~\eqref{eq_Kun} follows from the usual K\"{u}nneth Formula for $\mathbb{P}(V) \times U_b$ together with the uniqueness of the cohomologies in the above spectral sequences.

Under the hypothesis $\mathrm{H}^0((X,P), E_2^\vee \otimes E_1) = 0$, we obtain the universal extension $\mathcal{E}_V$ as the canonical extension class in
\begin{align}
\text{\rm Ext}^1(\mathcal{O}_{\mathbb{P}(V)}(-1) \boxtimes E_2, E_1) \cong \mathrm{H}^1\left( \mathbb{P}(V) \times (X,P), \mathcal{O}_{\mathbb{P}(V)}(1) \boxtimes \left( E_2^\vee \otimes E_1 \right) \right) \label{eq_Universal_Family}\\
\cong \mathrm{H}^0(\mathbb{P}(V), \mathcal{O}_{\mathbb{P}(V)}(1)) \otimes \mathrm{H}^1((X,P), E_2^\vee \otimes E_1) \nonumber \\
\cong V \otimes V^* \cong \text{\rm End}(V) \nonumber
\end{align}
corresponding to the identity map on $V$.

Finally, it is important to note the relation of the extension classes of orbifold bundles with those of the corresponding equivariant bundles. As before, we fix a $\Gamma$-Galois cover $g \colon Z \longrightarrow (X,P)$ where $Z$ is a smooth projective connected $k$-curve, and $g$ factors as a composition
\[
g \colon Z \overset{u}\longrightarrow [Z/\Gamma] = (X, \tilde{P}) \overset{\tau}\longrightarrow (X,P)
\]
of the natural atlas $u$ followed by the tamely ramified cover $\tau$ of connected orbifold curves, induced by $\text{\rm id}_X$. Suppose that $E_1, \, E_2 \in \text{\rm Vect}(X,P)$ satisfy the following condition:
\begin{equation}\label{eq_cohomo_zero}
\mathrm{H}^0(Z, g^*(E_2^\vee \otimes E_1)) = 0.
\end{equation}

By Corollary~\ref{cor_cohomologies}~\eqref{i:1}, we have $\mathrm{H}^0((X,P), E_2^\vee \otimes E_1) = 0$. By base change, the sheaves $\sheafext^0_f(E_1,E_2)$ and $\sheafext^0_{f \circ g}(g^*E_1, g^*E_2)$ on $\text{\rm Spec}(k)$ are both zero. Further, the formation of the locally free sheaves $\sheafext^1_f(E_1,E_2)$ and $\sheafext^1_{f \circ g}(g^*E_1, g^*E_2)$ commute with base change. By Proposition~\ref{prop_universal_general}, we have the universal family $\mathcal{E}_V$ on $\mathbb{P}(V) \times (X,P)$ where $V^* = \mathrm{H}^1((X,P), E_2^\vee \otimes E_1)$. Similarly, either by \cite[Corollary~4.5]{Lange} or by applying Proposition~\ref{prop_universal_general} to $Z$, there is a universal family $\mathcal{E}_W$ on $\mathbb{P}(W) \times Z$ with $W^* = \mathrm{H}^1(Z, g^*E_2^\vee \otimes g^*E_1)$. Further, the $\Gamma$-action on the equivariant bundle $g^*(E_2^\vee \otimes E_1)$ induces a $k[\Gamma]$-module structure on $W$, and on $W^*$. We have the following result.

\begin{proposition}\label{prop_relation_universal}
Under the above notation, we have the following.
\begin{enumerate}
\item The vector space $V^*$ is a linear subspace of the $\Gamma$-fixed points of $W^*$, and hence determines a closed subspace
\[
\mu \colon \mathbb{P}(V) \hookrightarrow \mathbb{P}(W).
\]\label{uni:1}
\item We have a natural isomorphism of bundles
\[
(\text{\rm id}_{\mathbb{P}(V)} \times g)^* \mathcal{E}_V \cong (\mu \times \text{\rm id}_Z)^*\mathcal{E}_W
\]
on $\mathbb{P}(V) \times Z$.\label{uni:2}
\end{enumerate}
\end{proposition}

\begin{proof}
The first statement follows directly from Corollary~\ref{cor_cohomologies}~\eqref{i:2} and the naturality of the functor $\text{\rm Proj}$. For the second statement, note that under the isomorphism~\eqref{eq_Universal_Family}
\[
\text{\rm End}(V) \cong \text{\rm Ext}^1\left( p_{\mathbb{P}(V)}^*\mathcal{O}_{\mathbb{P}(V)}(-1) \otimes q_{\mathbb{P}(V)}^*E_2, q_{\mathbb{P}(V)}^*E_1 \right),
\]
the universal family $\mathcal{E}_V$ is the image of the identity endomorphism $\text{\rm id}_V$. Similarly, under the isomorphism
\[
\text{\rm End}(W) \cong \text{\rm Ext}^1\left( p_{\mathbb{P}(W)}^*\mathcal{O}_{\mathbb{P}(W)}(-1) \otimes q_{\mathbb{P}(W)}^*g^*E_2, q_{\mathbb{P}(W)}^*g^*E_1 \right),
\]
the universal family $\mathcal{E}_W$ is the image of the identity endomorphism $\text{\rm id}_W$. We have the following natural isomorphisms
\[
(\mu \times \text{\rm id}_Z)^* q_{\mathbb{P}(W)}^* g^* E_i \cong (\text{\rm id}_{\mathbb{P}(V)} \times g)^* q_{\mathbb{P}(V)}^* E_i \, \text{ for } i = 1, 2, \, \text{ and}
\]
\[
(\mu \times \text{\rm id}_Z)^* p_{\mathbb{P}(W)}^* \mathcal{O}_{\mathbb{P}(W)}(1) \cong (\mu \times \text{\rm id}_Z)^* p_{\mathbb{P}(W)}^* \mu^* \mathcal{O}_{\mathbb{P}(V)}(1) \cong (\text{\rm id}_{\mathbb{P}(V)} \times g)^* p_{\mathbb{P}(V)}^* \mathcal{O}_{\mathbb{P}(V)}(1).
\]
This shows that the pullback of any extension of $\mathcal{O}_{p_{\mathbb{P}(W)}^*\mathbb{P}(W)}(-1) \otimes q_{\mathbb{P}(W)}^*g^*E_2$ by $q_{\mathbb{P}(W)}^*g^*E_1$ under $\mu \times \text{\rm id}_Z$ is the pullback of an extension of $\mathcal{O}_{p_{\mathbb{P}(V)}^*\mathbb{P}(V)}(-1) \otimes q_{\mathbb{P}(V)}^*E_2$ by $q_{\mathbb{P}(V)}^*E_1$ under the map $\text{\rm id}_{\mathbb{P}(V)} \times g$, and vice versa. Since the identity endomorphism $\text{\rm id}_V$ extends to the $\Gamma$-equivariant endomorphism $\text{\rm id}_W$, and the identity endomorphism $\text{\rm id}_W$ is naturally $\Gamma$-equivariant, the second statement follows.
\end{proof}

\bibliographystyle{amsplain}

\end{document}